\documentclass[sn-mathphys,Numbered]{sn-jnl} 

\usepackage{graphicx}%
\usepackage{subcaption}
\usepackage{multirow}
\usepackage{diagbox}
\usepackage{amsmath,amssymb,amsfonts}
\usepackage{mathrsfs}%
\usepackage[title]{appendix}%
\usepackage[table,xcdraw]{xcolor}
\usepackage{textcomp}%
\usepackage{manyfoot}%
\usepackage{booktabs}%
\usepackage{algorithm}%
\usepackage{algorithmicx}%
\usepackage{algpseudocode}%
\usepackage{listings}%
\usepackage{comment}
\usepackage{enumitem}
\setlist[enumerate]{itemsep=0mm}
\usepackage{lscape}
\usepackage{bm}
\usepackage{dsfont}
\usepackage{pdflscape}
\usepackage{bbold}
\usepackage{tabularx}
\usepackage{upgreek}
\usepackage{accents}
\usepackage{tikz}

\newcommand{\R}{\mathbb{R}}
\newcommand{\N}{\mathbb{N}}
\newcommand{\change}[1]{{\color{black}#1}}
\newcommand{\es}[1]{{\color{olive}#1}}

\newcommand{\moha}[1]{{\color{orange}#1}}
\newcommand{\BR}{B-rex}
\newcommand{\Ball}{\mathcal{B}}
\newcommand{\Binf}{\mathcal{B}_{\infty}}
  {\end{list}}

\makeatletter

  \gdef\listctr{list\romannumeral\the\@listdepth}\expandafter
  
\makeatother

\newcommand{\ve}[1]{\mathbf{#1}}
\def\x{\ve{x}}
\def\y{\ve{y}} 
\def\u{\ve{u}}

\def\z{\ve{z}}
\def\e{\ve{e}}
\def\v{\ve{v}}
\def\a{\ve{a}}

\def\d{\ve{d}}
\def\A{\ve{A}}
\def\Cc{\mathcal{C}}

\def\R{\mathbb R}

\def\prox{{\mbox{prox}}}

\def\prox{\mathrm{prox}} \def\arg{\mathrm{arg}}
\newcommand{\argmin}{\operatornamewithlimits{argmin}}

\newcommand{\indic}{\mathds{1}}
\newcommand{\reg}{$\ell_0\ell_2$}
\newcommand{\prx}{\operatorname{prox}_{\rho B_\Psi}}

\newtheorem{theorem}{Theorem}
\newtheorem{proposition}[theorem]{Proposition}%
\newtheorem{example}{Example}%
\newtheorem{remark}{Remark}%
\newtheorem{lemma}{Lemma}
\newtheorem{coro}{Corollary}
\newtheorem{definition}{Definition}%
\newtheorem{assump}{Assumption}
\graphicspath{{figures/}}

\begin{document}

\title[Article Title]{Exact continuous relaxations of $\ell_0$-regularized criteria with non-quadratic data terms}


\author*[1]{\fnm{M'hamed} \sur{Essafri}}\email{mhamed.essafri@irit.fr}

\author[2]{\fnm{Luca} \sur{Calatroni}}\email{luca.calatroni@unige.it}

\author[1]{\fnm{Emmanuel} \sur{Soubies}}\email{emmanuel.soubies@irit.fr}

\affil[1]{IRIT, Université de Toulouse, CNRS, Toulouse, France}
\affil[2]{\orgdiv{MaLGa Center, DIBRIS}, \orgname{Università degli studi di Genova \&  MMS, Istituto Italiano di Tecnologia}, \city{Genoa}, \country{Italy}}


\abstract{
We consider the minimization of $\ell_0$-regularized criteria involving non-quadratic data terms such as the Kullback-Leibler divergence and the logistic regression, possibly combined with an $\ell_2$ regularization. We first prove the existence of global minimizers for such problems and characterize their local minimizers. Then, we propose a new class of continuous relaxations of the $\ell_0$ pseudo-norm, termed as  $\ell_0$ Bregman Relaxations (\BR{}). They are defined in terms of suitable Bregman distances and
lead to  \emph{exact} continuous relaxations of the original $\ell_0$-regularized problem in the sense that they do not alter its set of global minimizers and reduce its non-convexity by eliminating certain local minimizers.  Both features  make such relaxed problems more amenable to be solved by standard non-convex optimization algorithms. 
In this spirit, we consider the proximal gradient algorithm  
and provide explicit computation of proximal points for the \BR{} penalty in several cases.
Finally, we report a set of numerical results illustrating the geometrical behavior of the proposed \BR{} penalty for different choices of the underlying Bregman distance, its relation with convex envelopes, as well as its exact relaxation properties in 1D/2D and higher dimensions.
}

\keywords{$\ell_0$ regularization, non-convex/non-smooth optimization, Bregman distances, proximal gradient algorithm, exact relaxations, convex envelopes.}


\maketitle
\section{Introduction}
Sparse models have attracted extensive attention over the last decades due to their importance in various fields such as statistics, computer vision, signal/image processing and machine learning. 
The main goal in sparse optimization is to compute a sparse solution with only few representative variables. This is crucial in many real-world scenarios, where high-dimensional and often redundant data are either too large or too noisy to be effectively processed as a whole.

In this paper, we consider \reg-regularized minimization
problems of the form:
\begin{equation}\label{eq:problem_setting}
    \hat{\x} \in \argmin_{\x \in \Cc^N} \; J_0(\x)  \; \text{ with } \; J_0(\x) := F_\y(\A \x)  +\lambda_0 \|\x\|_0+\frac{\lambda_2}{2}\|\x\|^2_2
    ,
\end{equation}
where $\A \in \R^{M \times N}$ is (usually) an underdetermined matrix ($M \ll N$), $\y \in \mathcal{Y}^M$ is the vector of observations, and $\lambda_0>0$ and $\lambda_2 \geq 0$ are two hyperparameters controlling respectively the strengths of the $\ell_0$ sparsity-promoting term and the $\ell_2$ ridge regularization term. The case $\lambda_2=0$ corresponds to a pure sparsity-promoting model while choosing $\lambda_2>0$ can also be of interest from a statistical viewpoint~\cite{hazimeh2022sparse}.
The set  $\Cc \subseteq \R$ is a constraint set which we consider to be either $\Cc = \R$ (unconstrained minimization) or $\Cc = \R_{\geq 0}$ (non-negativity constraint). Finally, the term {$\|\cdot\|^2_2$ denotes the standard squared $\ell_2$ norm}, and $\|\cdot\|_0$ stands for the $\ell_0$ pseudo-norm defined on $\mathbb{R}^N$ which counts the non-zero components of its arguments, that is
\begin{equation}
    \|\x\|_0 := \sharp\left\lbrace
    x_n , \; n \in [N] \; : x_n \neq 0 
             \right\rbrace
             ,
\end{equation}
where~$\sharp$ denotes cardinality and $[N] = \{1, \ldots, N\}$.

The data fidelity function  $F_\y : \R^M \to \R_{\geq 0}$ is a measure of fit between the model $\A\x$ and the data $\y$. In the context of signal and image processing, for instance,  following a classical Bayesian paradigm, its  form depends on the noise statistics assumed on the data, see, e.g., \cite{Stuart_2010}. In this work, we make the following assumption on $F_\y$.

\begin{assump}\label{assump}
The data fidelity function $F_\y$ is coordinate-wise separable, i.e. $F_\y(\z) =  \sum_{m=1}^M f(z_m; y_m)$, where for each $y \in \mathcal{Y}$, $f(\cdot; y)$ is strictly convex, proper, twice differentiable on $\mathrm{int}(\Cc)$ and bounded from below. 
Note that, to simplify notations, we define $f'(x;y) = (f(\cdot;y))'(x)$ and $f''(x;y) = (f(\cdot;y))''(x)$.
\end{assump} 

The exemplar data fidelity terms considered in this paper are reported in Table~\ref{tab:data_fidelty}. 
The \change{fifth} column refers to the constraint imposed on the ridge parameter $\lambda_2$. Note that the need for $\lambda_2>0$ arises solely for the logistic regression data fidelity in order to ensure the existence of global minimizers (see Section~\ref{sec:existance}). Finally, the last column defines the constraint set $\Cc$ ensuring that the problem is  properly defined.
\begin{table}[t]
    \renewcommand{\arraystretch}{2} 
    \centering
    \begin{tabular}{|c|c|c|c|c|c|}
    \hline
       \cellcolor[HTML]{FFFFC7} \textbf{Data terms} & \cellcolor[HTML]{FFFFC7}   $F_\y(\A\x)$ & \cellcolor[HTML]{FFFFC7}   $\A$ & \cellcolor[HTML]{FFFFC7}   $\mathcal{Y}$
       & \cellcolor[HTML]{FFFFC7} 
 $\lambda_2$ & \cellcolor[HTML]{FFFFC7}  $\Cc$\\ \hline
       Least squares (LS) & $\displaystyle \frac{1}{2}\sum_{m=1}^M\left([\A\x]_m- y_m\right)^2$ &
       $\R^{M \times N}$
&
       $\R$
       &
       $ \geq 0$ & $\R$ \\ \hline 
        Logistic regression (LR) & $\displaystyle \sum_{m=1}^M\log \left( 1 + e^{[\A\x]_m}\right)-y_m[\A\x]_m$ &
       $\R^{M \times N}$
&
        $\left\{0,1\right\}$
        &
        $ >0$ & $\R$ \\ \hline 
        Kullback-Leibler (KL) & $\displaystyle \sum_{m=1}^M [\A\x]_m+b-y_m\log\left([\A\x]_m+b\right)$, &
       $\R^{M \times N}_{\geq 0}$
&
        $\R_{\geq 0}$
        &
        $ \geq 0$ & $\R_{\geq 0} $ \\ \hline 
    \end{tabular}
    \caption{Some exemplar data fidelity terms considered in this work with minimal bounds on $\lambda_2$ to ensure existence of a global solution (see Theorem~\ref{th:existence-minimizers}) and the corresponding constraint sets $\Cc$  required to make Problem \eqref{eq:problem_setting} well defined. Note that for KL we should have $b > 0$.}
    \label{tab:data_fidelty}
\end{table}

Although the $\ell_0$ pseudo-norm is the natural measure of sparsity,  solving problems like \eqref{eq:problem_setting} 
 is known to be a NP-hard task~\cite{Natarajan1995-np-hard,nguyen-np-hard}. Also, being the $\ell_0$ pseudo-norm a non-convex function discontinuous at the origin, the design of optimization algorithms tailored to \eqref{eq:problem_setting} is challenging. 

 In the following section, we provide an overview of the  previous work addressing some of the main challenges encountered in both the modeling and the numerical optimization of problems analogous to \eqref{eq:problem_setting}.

\paragraph{Related works} 
There exists a vast literature dedicated to the aforementioned challenges encountered when dealing with pure $\ell_0$-minimization problems (i.e., $\lambda_2=0$)~\cite{tillmann2024cardinality}.  First, note that Problem~\eqref{eq:problem_setting} can be seen as a mixed binary integer program. The two standard ways for showing this correspond to express $\|\x\|_0 = \sum_{n \in [N]}z_n$ with $\z \in \{0,1\}^N$ along with some constraints either in the form $\x (1 - \z) = \mathbf{0}$ (with component-wise products) or $\x \leq \z M$ (big-M constraint). For moderate-size problems involving hundreds of variables, such mixed-integer problems can be solved exactly via branch-and-bound algorithms at a reasonable computational cost~\cite{bourguignon2015exact,bertsimas2016best,guyard2022node,delle2023novel}. Alternatively, other approaches (see, e.g., \cite{frangioni2006perspective,gunluk2010perspective,wei2022ideal,convexifications2023, wei2024convex,rank2024}) provide good convex relaxations of the set of constraints involved in such problems for their solution.

Other approaches, known in the literature under the name of greedy algorithms, share the common idea of computing solutions of $\ell_0$-regularized least-squares criteria by iteratively  modifying the support of the solution (e.g., adding, removing or swapping components)  according to a given rule. Examples include matching pursuit (MP)~\cite{Mallat}, orthogonal matching pursuit (OMP)~\cite{Pati1993Orthogonal},  single best replacement (SBR)~\cite{Soussen2011FromBD}, or the CowS algorithm proposed in~\cite{beck2018proximal}. While the use of sophisticated update rules (with increased combinatorics) can ensure the convergence to local minimizers verifying more restrictive necessary optimally conditions~\cite{beck2018proximal,soubies2020new}, this can significantly increase the computational cost.   


Another popular strategy aims at defining relaxations of the original Problem~\eqref{eq:problem_setting} through the replacement of the $\ell_0$ term by a  continuous approximation. Replacing the $\ell_0$ pseudo-norm by the convex and continuous $\ell_1$ norm is the most popular approach as it allows one to leverage efficient convex optimization tools~\cite{tibshirani1996regression}. Moreover,  the seminal works~\cite{Candes,donoho2006most} showed that under some conditions on the model operator $\A$, the $\ell_1$ relaxation enjoys the theoretical guarantees of retrieving solutions of the original $\ell_2$-$\ell_0$ problem. However, these conditions are based on randomness assumptions on $\A$ which are not met for many practical problems, such as classical imaging inverse problems where, for instance, $\A$ is a convolution operator. Non-convex continuous relaxations have thus been extensively studied as alternatives to the $\ell_1$ relaxation. They include (but are not limited to) the transformed-$\ell_1$~\cite{Nikolova2000LocalSH}, capped-\(\ell_1\)~\cite{zhang2008multi}, \(\ell_p\)-norms \((0<p<1)\)~\cite{Foucart2009SparsestSO}, \(\log\)-sum penalty~\cite{Cands2007EnhancingSB}, smoothed \(\ell_0\) penalty~\cite{Mohimani2008AFA}, smoothly clipped absolute deviation (SCAD)~\cite{Fan2001VariableSV}, minimax concave penalty (MCP)~\cite{Cun-Hui}, exponential approximation~\cite{Mangasarian1996}, ratio \(\ell_p/\ell_q\)~\cite{Repetti2015,cherni2020spoq} and reverse Huber penalty~\cite{Pilanci}.
Among this plethora of continuous non-convex relaxations of Problem~\eqref{eq:problem_setting}, one may wonder which one to choose in practice?

From a statistical perspective, according to the authors in~\cite{Fan01,antoniadis2001regularization}, a ``good'' penalty function should lead to an estimator which is  i)~unbiased when the true solution is large (to avoid unnecessary modelling bias),
ii)~a thresholding rule that enforces sparsity (to reduce model complexity), and iii)~continuous with respect to the data (to avoid instabilities in model prediction). Such properties are the root of MCP and SCAD penalties. 

From an optimization point of view, a series of works have proposed to tune the parameters of non-convex penalties so as to maintain the convexity of the whole relaxed objective function~\cite{selesnick2017sparse,lanza2022convex}. As such, one can leverage efficient convex optimization tools while keeping the relaxation ``closer'' to the initial Problem~\eqref{eq:problem_setting} than that obtained  with the $\ell_1$ relaxation. These can be seen as the initial convex relaxations of  graduated non-convexity approaches~\cite{nikolova1998inversion}.
A different idea consists in defining relaxations that ``reduce" the non-convexity of Problem~\eqref{eq:problem_setting} (in terms, e.g., of fewer local minimizers and wider basins of attraction) while preserving its solution(s), which are both appealing properties in the context of non-convex optimization. Relaxations with such properties are  referred to as \emph{exact} continuous relaxations. 

Early works on this topic date back to \cite{Bradley} where the authors proved that, for a certain class of functions \(F_\y\), the intersection between the solution set of the \(\ell_0\)-penalized criteria and the one of the relaxed criteria obtained with the exponential penalty~\cite{Mangasarian1996} is non-empty. Similar results were obtained later on with the \(\log\)-sum penalty~\cite{Rinaldi2010ConcavePF} and \(\ell_p\)-norms for \(p \leq 1\)~\cite{Fung2011EquivalenceOM}. From a  different perspective, asymptotic connections in terms of global minimizers have been shown for a class of smooth non-convex approximations of the \(\ell_0\) pseudo-norm~\cite{Chouzenoux2011AMS}. When the composite $\ell_0$-dependent functional to minimize possesses further structure, other type of approaches can be considered as well. For instance,  when the data fidelity term can be expressed as the difference of two convex functions (DC function) and \(\Cc\) is a polyhedral convex set, a family of continuous DC approximations has been proposed in~\cite{Thi2015} where a precise link between the resulting relaxation and the original problem is made: any minimizer of the relaxed problem lies in an \(\varepsilon\)-neighborhood of a minimizer of the initial problem. Moreover, under some assumptions on \(F_\y\), optimal solutions of the relaxed problem are included in those of the initial problem. A stronger result is shown therein also for the capped-\(\ell_1\) penalty (see also~\cite{LeThi2014FeatureSI}), for which global solutions of the relaxation exactly coincide with those of the initial functional. However, these last two results are limited to global solutions and do not hold for local minimizers, meaning that this relaxation can potentially add spurious local minimizers.

The authors in~\cite{soubies2017unified} defined a class of sparse penalties leading to exact relaxations of the $\ell_0$-regularized least-squares problem. Interestingly, as the inferior limit of this class of exact penalties, they retrieve a special instance of MCP, referred to as CEL0 (continuous exact $\ell_0$) that was initially analyzed in~\cite{Soubies2015}. Still in the context of least-squares problems, the author in~\cite{Carlsson2019} showed that the CEL0 relaxation can actually be obtained by computing the quadratic envelope (also known as proximal hull~\cite{rockafellar2009variational}) of \(\lambda_0\|\cdot\|_0\) defined by
\begin{equation}\label{eq:quadratic}
\mathcal{Q}_\gamma(\lambda_0 \|\cdot\|_0)(\x)=\sup _{\alpha \in \R, \z \in \R^N}\left\{\alpha-\frac{\gamma}{2}\|\x-\z\|_2^2: \alpha-\frac{\gamma}{2}\|\cdot-\z\|_2^2 \leq \lambda_0\|\cdot\|_0\right\}.
\end{equation}
Furthermore, it was shown in~\cite{Carlsson2019} that, for a range of \(\gamma\) values (i.e., \(\gamma>\|\A\|^2\)), the exact relaxations properties hold. 

The study of similar properties in the case of non-quadratic data fidelity terms is  less developed.
We are only aware of the works~\cite{liu2018equivalent,Weigeneral,Lazzaretti2021}.
In~\cite{liu2018equivalent}, the authors propose a class of MPEC (mathematical programs with equilibrium constraints)  exact reformulations. In~\cite{Weigeneral} the authors demonstrated that the capped-\(\ell_1\) penalty leads to an exact relaxation of \eqref{eq:problem_setting} whenever $F_\y$ is Lipschitz continuous. In~\cite{Lazzaretti2021} a weighted-CEL0 relaxation is defined for $\ell_0$-penalized problems coupled with a weighted-$\ell_2$ data term to model signal-dependent noise in fluorescence microscopy inverse problems. 

Finally, although the literature related to \reg-regularization  (i.e., $\lambda_2>0$) is less extensive, we can mention branch-and-bounds types methods~\cite{hazimeh2022sparse,LRhazim}, safe-screening rules~\cite{atamturk20a}, MPEC reformulations~\cite{zhang2023zero}, and exact relaxation properties of the capped-$\ell_1$ penalty for Lipschitz continuous data fidelity terms~\cite{li2022difference}.

\paragraph{Contributions and outline}
In this paper, we derive a new class of exact continuous relaxations for Problem~\eqref{eq:problem_setting}. It extends the previous works~\cite{Soubies2015,Carlsson2019} to problems involving non-quadratic data terms.
Our framework heavily relies  on the use of separable Bregman divergences $D_\Psi(\cdot,\cdot)$~\cite{Bregman1967,CensorZenios1997,Burger2016} generated by a family $\Psi=\left\lbrace \psi_n \right\rbrace_{n \in [N]}$ of strictly convex functions  $\psi_n: \Cc \to \R$, so as to replace the squared Euclidean distance in \eqref{eq:quadratic}. Specifically, we define the class of \textit{$\ell_0$ Bregman relaxations} (\BR{}) as 
\begin{equation}\label{eq:continuous_form} 
    B_\Psi(\x)=\sup _{\alpha \in \R} \sup_{\z \in \Cc^N}\left\{\alpha-D_\Psi(\x,\z): \alpha-D_\Psi(\cdot,\z) \leq \lambda_0\|\cdot\|_0\right\}
    .
\end{equation}
Then, we derive sufficient conditions  (independent on $\A$) on the family $\Psi$ such that the continuous relaxation of $J_0$ defined by  
\begin{equation}\label{eq:relaxation_Jpsi}
        J_\Psi(\x) = F_\y(\A \x) + B_\Psi(\x)+\frac{\lambda_2}{2} \|\x\|^2_2,
\end{equation}
is an exact relaxation of $J_0$ in the sense that the following two properties hold 
 \begin{align}
         & \argmin_{\x \in \Cc^N} J_\Psi(\x) = \argmin_{\x \in \Cc^N} J_0(\x) \label{pr:whiched_form_glb},
         \tag{P1}\\
         & \hat{\x} \text{ local (not global) minimizer of } J_\Psi \; \Rightarrow \; \hat{\x} \text{ local (not global) minimizer of } J_0 .\label{pr:whiched_form_local}
         \tag{P2}
\end{align}

    
In other words,  $J_\Psi$ preserves  global minimizers of $J_0$, can potentially remove some local minimizers of $J_0$, and do not add new local minimizers.

To establish such results, it is first necessary to ensure that Problem~\eqref{eq:problem_setting} is well-posed in the sense that its solution set is non-empty. 
We provide in Section~\ref{sec:existance} results on the existence of global minimizers of $J_0$. It turns out to be a special case of a more general result proved in Appendix~\ref{appendix:existence}, which entails the existence of global minimizers of $J_\Psi$ as well. A characterization of (strict) local minimizers of $J_0$ completes Section~\ref{sec:existance}.

In Section~\ref{sec:brex},  we introduce the \BR{} and demonstrate its exact relaxation properties. More precisely,
we provide sufficient conditions on $\psi_n$ to ensure the validity of both properties~\eqref{pr:whiched_form_glb} and~\eqref{pr:whiched_form_local}. Under these conditions, $J_\Psi$ is therefore an exact continuous (non-convex) relaxation of $J_0$ that lies below it. Moreover, we identify the local minimizers of $J_0$ that are eliminated by $J_\Psi$. Along with these results, we characterize, critical points as well as (strict) local minimizers of $J_\Psi$.

In Section~\ref{sec:choices}, we give insights on how to choose the generating functions $\Psi$. In particular, we draw connections between the relaxed functionals and convex envelopes in the simple case where $\A$ is diagonal. More specifically, we show that choosing $\psi_n(x)=f(x;y_n) + \frac{\lambda_2}{2} x^2$ (note that $M=N$) entails  that $J_\Psi$ is the l.s.c. convex envelope of $J_0$.

Section~\ref{sec:algo} focuses on the minimization of the relaxed functional $J_\Psi$ via the proximal gradient algorithm. We present a formula for the proximal operator of \BR{}, along with explicit expressions for different choices of Bregman distance.

To conclude, we present in Section \ref{sec:application}  numerical results in 1D/2D and higher dimensions covering a wide range of problems involving both least-squares data terms (as reference) and widely-used non-quadratic data terms (Kullback-Leibler divergence and logistic regression).

\paragraph{Notations}
We make use of the following notation:
\begin{itemize}
    \item $\mathbb{R}_{\geq 0} = \left\{x\in\mathbb{R} \; : \; x\geq 0\right\}$;
      \item $\Ball(\x;\varepsilon)$, the open ball of center $\x \in \R^N$ and radius $\varepsilon>0$.
    \item $\mathbf{I}\in\R^{N\times N}$, the identity matrix;
    \item $[N]=\left\lbrace 1,\dots,N\right\rbrace $;
    \item $\x^{(n)}=\left(x_1, \dots,x_{n-1},0,x_{n+1},\dots,x_N \right) \in \mathbb{R}^N$;
    \item $\sigma \left( \x \right) = \left\lbrace i \in [N]~:~x_i \neq 0\right\rbrace$, the support of $\x \in \mathbb{R}^N$;
    \item $\sharp \omega$ denotes the cardinality of the set $\omega \subseteq [N]$;
    \item $\A_\omega= \left( \a_{\omega[1]}, \dots, \a_{\omega[\sharp \omega]} \right) \in \mathbb{R}^{M \times \sharp \omega}$, the submatrix of $\A$ formed by selecting only the columns indexed by the elements of $\omega \subseteq [N]$;
    \item $\x_\omega = \left(x_{\omega[1]}, \ldots, x_{\omega[ \sharp\omega]}\right) \in \mathbb{R}^{\sharp \omega}$, the restriction of $\x \in \mathbb{R}^N$ to the entries indexed by the elements of $\omega \subseteq [N]$;
    \item $\e_n \in \R^N$, the unitary vector of the standard basis of $\R^N$.
    \item $\|\cdot\|=\|\cdot\|_2$, the $\ell_2$ norm.
  
\end{itemize}

\section{On the minimizers of $J_0$}\label{sec:existance}

In this section we prove the existence of a solution to Problem \eqref{eq:problem_setting} and characterize  local minimizers.

\subsection{Existence}  \label{sec:existence}

One observes that whenever $\lambda_2>0$, existence of a solution to~\eqref{eq:problem_setting}  trivially holds since the functional is lower semi-continuous and coercive. 
However, for pure sparsity problems (where $\lambda_2=0$), establishing existence results is more challenging since the functional $J_0$ may not be coercive. To address this challenge, we exploit the notion of asymptotically level stable functions~\cite{AuslenderTeboulle} in the case where $F_\y$ is coercive (note that $F_\y(\A\cdot)$ is not necessary coercive). Existence results for the $\ell_0$-regularized least-squares problem have been established by Nikolova in~\cite{nikolova2013description} using this notion. Here, we provide a general proof ensuring the existence of a solution for other types of data terms under the mild Assumption~\ref{assump}. 
\begin{theorem}[Existence of solutions to~\eqref{eq:problem_setting}] \label{th:existence-minimizers}
Let $F_\y$ be coercive or $\lambda_2>0$. Then, the solution set of~\eqref{eq:problem_setting}  is non-empty.
\end{theorem}
\begin{proof}
If $\lambda_2 > 0$, then the result is trivial due to the coercivity of the squared $\ell_2$ norm and the fact that both the data fidelity term (Assumption~\ref{assump}) and the $\ell_0$ term are bounded from below. If $\lambda_2 = 0$ and $F_\y$ is coercive, the proof is a particular case of Theorem~\ref{th:existence-general}, whose statement and proof are given in Appendix~\ref{appendix:existence}, under the choice $\Phi(\cdot) = \lambda_0 \|\cdot\|_0$ in~\eqref{eq:jphi}.
\end{proof}

\subsection{Characterization of local minimizers} \label{sec:characterization}

The following proposition gives a characterization of local minimizers of Problem~\eqref{eq:problem_setting}. It  shows that the task of finding local minimizers is easy. It corresponds to solving a convex problem for a given  support  (i.e., the subset of $[N]$ identifying the non-zero entries of the considered minimizer).

\begin{proposition}[Local minimizers of $J_0$~\cite{Thi2015}]\label{prop:local_min_J0}
A point $\hat{\x} \in \Cc^N$ is a local minimizer of $J_0$ if and only if it solves
\begin{equation}
\label{eq:restriction_j0}
    \hat{\x}_{\hat{\sigma}} \in  \argmin_{\z \in \Cc^{\sharp\hat{\sigma}}}  F_\y(\A_{\hat{\sigma}} \z)+\frac{\lambda_2}{2}\|\z\|^2_2
\end{equation}
where $\hat{\sigma} = \sigma(\hat{\x})$ stands for the support of $\hat{\x}$. 
\end{proposition}
\begin{proof}
   Although this result is known from~\cite{Thi2015}, we prove it in Appendix~\ref{appendix:characterization} for completeness.
\end{proof}

\begin{coro}\label{coro:criter_locmin_j0}
  Let~$\hat{\x} \in \Cc^N$. Then $\hat{\x}$ is a local minimizer of~$J_0$ if and only if  
    \begin{equation}
        \hat{\x}_{\hat{\sigma}} \quad \text{solves} \quad \A_{\hat{\sigma}}^T \nabla F_\y\left(\A_{\hat{\sigma}}\hat{\x}_{\hat{\sigma}}  \right)+\lambda_2 \hat{\x}_{\hat{\sigma}}=\mathbf{0}
    \end{equation}
    for $\hat{\sigma} = \sigma(\hat{\x})$.
\end{coro}
\begin{proof}
The proof directly stems from Proposition~\ref{prop:local_min_J0} and the first-order optimality condition for the convex problem~\eqref{eq:restriction_j0}. Note that in the case $\mathcal{C}=\mathbb{R}_{\geq 0}$ we have $\hat{\x}_{\hat{\sigma}} > 0$, so that there is no need to consider the complementary conditions.
\end{proof}

     Given these results, we see that the difficulty in finding a global minimizer  lies in the determination of the correct support. The amplitudes of non-zero coefficients are then easy to obtain.

\subsection{Characterization of strict local minimizers} \label{sec:strict_minimizers}

The following results provide a characterization of the strict (local) minimizers of $J_0$.

\begin{definition}\label{def:strict-minimizer}
A point $\hat{\x} \in \Cc^N$ is called a strict (local) minimizer for Problem~\eqref{eq:problem_setting} if there exists $\varepsilon >0$ such that 
\[
J_0(\hat{\x}) < J_0(\x), \quad \forall \x \in (\Ball(\hat{\x};\varepsilon) \cap \mathcal{C}^N) \setminus \{\hat{\x}\}
.
\]
\end{definition}
\begin{lemma}\label{lemma:0-is-strict-min}
     $J_0$ has a strict (local) minimizer at $\hat{\x}=\mathbf{0} \in \mathcal{C}^N$.
\end{lemma}
\begin{proof}
The proof is given in Appendix~\ref{appendox:0strict}.
\end{proof}

\begin{theorem}[Strict local minimizers of $J_0$]\label{thr:strict-minimizer}
    Let $\hat{\x}$ be a (local) minimizer of $J_0$. Define $\hat{\sigma}=\sigma(\hat{\x})$. Then $\hat{\x}$ is strict if and only if $\operatorname{rank}(\A_{\hat{\sigma}})=\sharp \hat{\sigma}$ or $\lambda_2 >0$.
\end{theorem}
\begin{proof}
The proof is given in Appendix~\ref{appendix:strict-min}.
\end{proof}

\begin{theorem}\label{th:glob_min_J0_strict}
    Global minimizers of $J_0$ are strict.
\end{theorem}
\begin{proof}
 The proof is given in Appendix~\ref{sec:proof_glob_min_J0_strict}.
\end{proof}

Theorem~\ref{th:glob_min_J0_strict} shows that, among local minimizers of $J_0$, strict ones are of great interest as they contains global minimizers. Moreover, from Theorem~\ref{thr:strict-minimizer} one can see that strict local minimizers of $J_0$ are countable, although growing exponentially with the dimension.

\section{The \BR{}  and its exact relaxation properties}\label{sec:brex}

\subsection{Definition}

In this section, we introduce the $\ell_0$ Bregman relaxation (\BR{}), a continuous approximation of $\lambda_0 \|\cdot\|_0$. We then provide a geometrical interpretation of this relaxation as well as its relation to generalized conjugates. 

\begin{definition}[The $\ell_0$ Bregman relaxation] \label{def:Breg_L0} Let $\Psi=\{\psi_n\}_{n \in [N]}$ be a family  generating functions $\psi_n: \Cc \to \R$ which are strictly convex, proper,  twice differentiable over $\mathrm{int}(\Cc)$, and such that $z \mapsto \psi_n'(z)z - \psi_n(z)$ is coercive. Then,  we define the $\ell_0$-Bregman relaxation (\BR{}) associated to $\Psi$ as 
\begin{equation}~\label{eq:prop:def_bergman1}
    B_\Psi(\x): = \sup_{\alpha\in \mathbb{R}} \; \sup_{\z \in \Cc^N}~ \left\lbrace \alpha- D_{\Psi}(\x,\z): \alpha -  D_{\Psi}(\cdot,\z) \leq \lambda_0 \|\cdot\|_0  \right\rbrace ,
\end{equation} 
where $D_{\Psi} : \Cc^N \times \Cc^N \to \R_{\geq 0}$ is the separable Bregman distance associated to $\Psi$, 
\begin{equation}\label{eq:bregman-distance}
    D_\Psi(\x,\z) = \sum_{n=1}^N d_{\psi_n}(x_n,z_n) \quad \text{ with } \quad d_{\psi_n}(x,z) =  \psi_n(x) - \psi_n(z) - \psi_n'(z) (x-z) 
\end{equation}
for all $\x,\z \in\Cc^N$ and $x,z\in\Cc$.
\end{definition}

Example of classical generating functions used in this paper are reported in Table~\ref{tab:generating-functions}.

\begin{table}
\renewcommand{\arraystretch}{2} 
    \centering
    \begin{tabular}{|c|c|c|}
        \hline
       \cellcolor[HTML]{FFFFC7} \textbf{Function} & \cellcolor[HTML]{FFFFC7}\textbf{$\mathcal{C}$} &  \cellcolor[HTML]{FFFFC7}\textbf{$\psi_n$} \\
        \hline
        Power Function & $\mathbb{R}$ & $\displaystyle \frac{\gamma_n}{p(p-1)}|x|^p$,\quad $p > 1$  \\ 
        \hline
        Shannon Entropy  & $\mathbb{R}_{\geq 0}$ & $\displaystyle \gamma_n(x \log(x) - x + 1)$ \\ 
        \hline
        Kullback Leibler & $\mathbb{R}_{\geq 0}$ & $\displaystyle \gamma_n(x + b - y\log(x + b))$,\quad $y,b > 0$ \\
        \hline
    \end{tabular}
    \caption{Generating functions of the form $\psi_n =\gamma_n \psi$ satisfying assumptions in Definition~\ref{def:Breg_L0}.}
    \label{tab:generating-functions}
\end{table}

\begin{remark}
    The requirement that $z \mapsto \psi_n'(z)z - \psi_n(z)=d_{\psi_n}(0,z) \moha{-\psi_n(0)}$ is coercive is rather mild and holds for all the common generating functions considered in this paper (see Table~\ref{tab:generating-functions}). By \cite[Corollary 3.11]{Bauschke1997} and observing that the functions $\psi_n$ are Legendre functions, we can equivalently assume that $\text{dom} ~\psi_n^*$ (where $\psi_n^*$ denotes the Fenchel conjugate of $\psi_n$)  is open to guarantee the required coercivity (actually, in this case we have the coercivity of $d_\psi(x, \cdot)$ for all $x$). Such condition is necessary to ensure the existence of the $\alpha_n^-$ and $\alpha_n^+$ in Proposition~\ref{prop:Breg_l0} which are quite central to the subsequent analysis.  Note that, for a fixed $\lambda_0$, we can relax this condition by requiring the functions $\psi_n$ to be  such that the $\lambda_0$-sublevel set of   $d_{\psi_n}(0,\cdot)$ is bounded.
\end{remark}

 When choosing $\psi_n(x)=\gamma \frac{x^2}{2}$ for $\gamma>0$, the Bregman distance $D_{\Psi}$ corresponds to the  standard Euclidean distance. In this case, the \BR{} appeared in~\cite[Example 1.44]{rockafellar2009variational} under the name of ``proximal hulls'', for which the link with convex envelops and exact relaxation were not studied. This link was explored in~\cite{Carlsson2019}, where the \BR{} penalty with quadratic generating functions was studied under the name of ``quadratic envelopes'' (cf. equation~\eqref{eq:quadratic}). This term refers indeed to the use of the Euclidean distance in Definition~\ref{def:Breg_L0} and to the use of quadratic data terms in \eqref{eq:problem_setting}. The intuition behind our proposal consists in changing the underlying geometry in terms of suitable Bregman divergences in order to study the link with  convex envelopes and to derive exact relaxations in the case of non-quadratic data terms.
The following proposition provides a tool for computing the \BR{} penalty explicitly, given the familly of generating functions $\Psi$.

\begin{proposition}\label{prop:Breg_l0}
Let $\Psi$ be as in Definition~\ref{def:Breg_L0}. Then, for every $\x\in\Cc^N$, we have
$
    B_\Psi(\x) = \sum_{n=1}^N \beta_{\psi_n}(x_n)
$ with
\begin{equation}\label{eq:prop1_equa2} 
    \beta_{\psi_n}(x) = \left\lbrace 
    \begin{array}{ll}
        \psi_n(0) - \psi_n(x) +  \psi_n'(\alpha_n^-)x, \quad &\quad  \text{ if } x \in [\alpha_n^-,0] \\
          \psi_n(0) - \psi_n(x) +  \psi_n'(\alpha_n^+)x,   \quad & \quad \text{ if } x \in [0,\alpha_n^+] \\
      \lambda_0,  \quad  & \quad  \text{ otherwise}
    \end{array}
    \right.,
\end{equation}
where the interval $[\alpha_n^-,\alpha_n^+] \ni 0$ defines the $\lambda_0$-sublevel set of $ d_{\psi_n}(0,\cdot)$. Note that $\alpha_n^-=0$ in the case $\Cc=\R_{\geq 0}$.
\end{proposition}
\begin{proof}
    The proof can be found in Appendix~\ref{appendix:structure_bcel0}. Note that the existence of $\alpha_n^-$ and $\alpha_n^+$ is ensured with the assumption that $z \mapsto \psi_n'(z)z - \psi_n(z)$ is coercive which implies that  $ d_{\psi_n}(0,\cdot)$ is coercive and has bounded sublevel sets. 
\end{proof}

 In other words, Proposition~\ref{prop:Breg_l0} states that we can get a closed-form expression of \BR{} as soon as we have access, for all $n\in [N]$, to  $\psi_n'$ and to  the bounds $\alpha^-_n$ and $\alpha^+_n$ of the $\lambda_0$-sublevel set of $d_{\psi_n}(0,\cdot)$. The closed-form expressions of \BR{} for the generating functions of Table~\ref{tab:generating-functions} are reported in Table~\ref{tab:B-rex}. Technical details on the computation of the points $\alpha_n^{\pm}$ are reported in Appendix~\ref{appendix:compute-brex}. We plot the graphs of these \BR{} in Figure~\ref{fig:norm-zero-brex}. Finally, note that whenever the \BR{} penalty is associated with the power function, by setting $p=2$ we get the relaxation proposed in~\cite{Soubies2015,Carlsson2019} for $\gamma_n=\|\a_n\|^2$ where $\a_n$ denote the $n$th column of the matrix $\A$.

\begin{table}
\captionsetup{width=0.9\textwidth}
\renewcommand{\arraystretch}{1.5} 
    \centering
   \begin{tabular}{|c|c|c|c|}
   \hline
  \cellcolor[HTML]{FFFFC7}     $\psi_n$ &  \cellcolor[HTML]{FFFFC7}  $ \alpha_n^-$ &  \cellcolor[HTML]{FFFFC7}   $\alpha_n^+$ &  \cellcolor[HTML]{FFFFC7} $\psi_n'$ \\ \hline
    Power Function &  $  -\left(\frac{p\lambda_0}{\gamma_n} \right)^{\frac{1}{p}}  $ & $ \left(\frac{p\lambda_0}{\gamma_n} \right)^{\frac{1}{p}} $ &  $\; \frac{\gamma_n\operatorname{sign}(x)}{p-1}|x|^{p-1}  \;$  \\ \hline
    Shannon Entropy &  $0$ & $\frac{\lambda_0}{\gamma_n}$& $\gamma_n\log(x)$  \\ \hline
    Kullback-Leibler  &   $0$ & $ \; \frac{-b}{W(-b\e^{-\kappa})}-b$ & $\gamma_n(1-\frac{y}{x+b}) \;$ \\ \hline
    \end{tabular}
\bigskip
   \begin{tabular}{|c|c|}
    \hline
  \cellcolor[HTML]{FFFFC7}     $\psi_n$ & \cellcolor[HTML]{FFFFC7}   $\beta_{\psi_n}(x)$ \\ \hline
    Power Function &  $\begin{cases}
    \frac{-\gamma_n}{p(p-1)}|x|^p-\frac{\gamma_n}{p-1}  \left( \frac{p\lambda_0}{\gamma_n} \right)^{\frac{p-1}{p}}  x & \text{ if } x \in [\alpha_n^- ,0] \\
    \frac{-\gamma_n}{p(p-1)}|x|^p +  \frac{\gamma_n}{p-1} \left( \frac{p\lambda_0}{\gamma_n} \right)^{\frac{p-1}{p}} x   & \text{ if } x \in [0, \alpha_n^+ ] \\ 
    \lambda_0   &  \text{ otherwise}.
    \end{cases}$ \\ \hline
    Shannon Entropy & $ \begin{cases}
    \gamma_n x \left(\log\left(\frac{\lambda_0}{\gamma_n x} \right)+1 \right),  & \text{ if } x \in \left[ 0,\frac{\lambda_0}{\gamma_n} \right]\\
    \lambda_0, &\text{ if }  x \geq \alpha_n^+
    \end{cases}$ \\ \hline
    Kullback-Leibler  &   $\begin{cases}
    \gamma_n y \left(\log \left ( \frac{x+b}{b} \right )+\frac{W(-be^{-\kappa})}{b} x \right)  & \text{ if } x \in [0 , \alpha_n^+] \\
    \lambda_0,   &  \text{ if } x \geq \alpha_n^+
    \end{cases}$\\ \hline
    \end{tabular}
    \caption{Top: quantities $\alpha_n^-$, $\alpha_n^+$ and $\psi_n'$ for the  generating functions $\psi_n$ of Table~\ref{tab:generating-functions}. Bottom: Corresponding expressions of the \BR{} penalty.  $W(\cdot)$ denotes the Lambert function and $k=\frac{\lambda_0}{y\gamma_n}+\log(b)+1$. }
    \label{tab:B-rex}
\end{table}

\begin{figure}
  \centering
  \begin{subfigure}[b]{0.45\textwidth}
    \includegraphics[width=\textwidth]{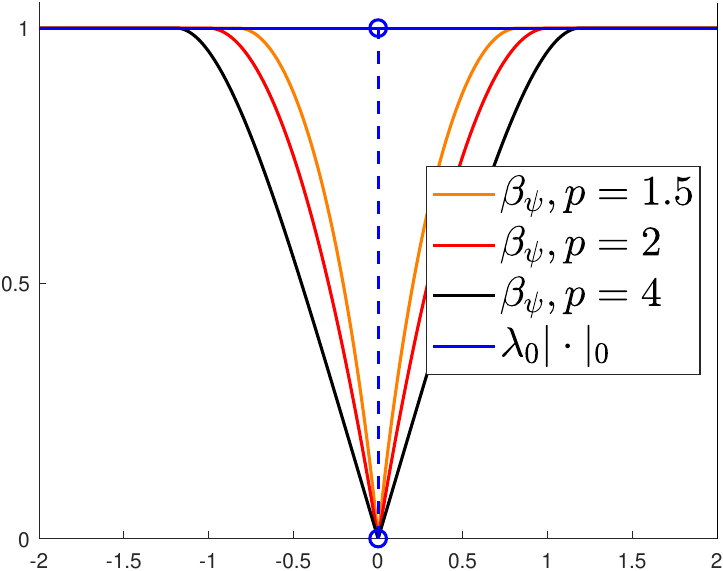}
    \caption{The \BR{} penalty on $\mathcal{C}=\R$ for different choices of $p$-power functions $\psi$.}
    \label{fig:brex_power_function}
  \end{subfigure}
  \hfill
  \begin{subfigure}[b]{0.45\textwidth}
    \includegraphics[width=\textwidth]{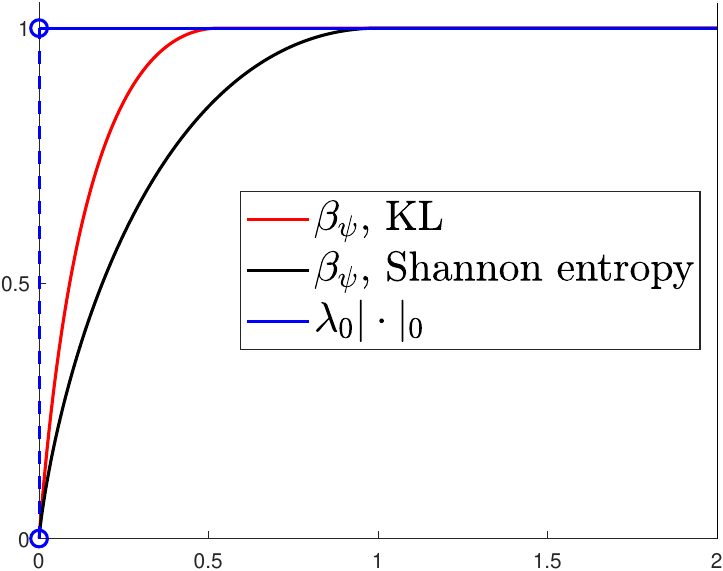}
    \caption{The \BR{} penalty on $\mathcal{C}=\R_{\geq 0}$ for $\psi=\mathrm{KL}$ and $\psi=$ Shannon entropy.}
    \label{fig:brex_kl_shannon}
  \end{subfigure}
  \caption{A plot of $\lambda_0|\cdot|_0$ on $\Cc$ along with its \BR{} penalty for different choices of  generating functions $\psi$ (see Tables~\ref{tab:generating-functions} and~\ref{tab:B-rex}).}
  \label{fig:norm-zero-brex}
\end{figure}

\begin{remark}
    As in Definition~\ref{def:Breg_L0}, all our theoretical results will be expressed with a general family of generating functions $\Psi$. However, in all our illustrations, we will consider generating functions of the form $\psi_n(x) = \gamma_n \psi(x)$ for a given $\psi$ and  $\gamma_n>0$ (cf Table~\ref{tab:generating-functions}).  As such, exact relaxation properties are fully controlled by the values of parameters $\gamma_n$ which simplifies the presentation.  
\end{remark}

\paragraph{Geometrical interpretation}

We provide in Figure~\ref{fig:def_illu}  a geometrical interpretation of the  proposed \BR{} of Table~\ref{tab:B-rex}. More precisely, we plot the graphs of the $\ell_0$ pseudo-norm in 1D, alongside its \BR{} relaxation computed by formula \eqref{eq:prop1_equa2}. We also plot in blue the minorants functions $x \to \alpha - d_{\psi}(x,z)$, for different $z$, and for the $\alpha$ attaining  the $\sup$ in Definition~\ref{def:Breg_L0} (i.e., the largest $\alpha$ which ensures the function to remain below $\lambda_0 |x|_0$). Then, taking the supremum with respect to $z$ leads to the \BR{} curve shown in red. Furthermore, note that \BR{} is constant and equal to $\lambda_0$ outside $[\alpha^{-}, \alpha^+]$, and concave on intervals $[\alpha^-,0]$ and $[0,\alpha^+]$. Such concavity arises from the convexity of Bregman distances with respect to the first argument, for all choices of $\psi$. 

\begin{figure}
\centering
\begin{subfigure}[t]{0.85\textwidth} 
\centering
    \includegraphics[width=\linewidth]{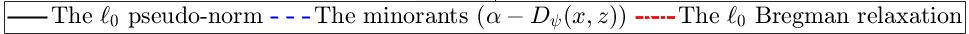}
\end{subfigure}

\begin{subfigure}[t]{0.45\textwidth} 
\centering
\includegraphics[width=\linewidth]{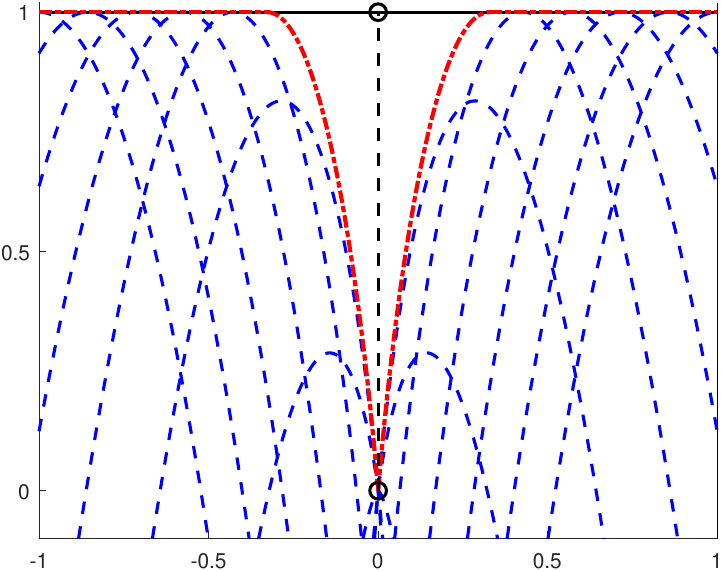}
\caption{ Power function with $p=3/2$.}
\label{subfig:def_illu_p2}
\end{subfigure}
\hfill
\begin{subfigure}[t]{0.45\textwidth}
\centering
\includegraphics[width=\linewidth]{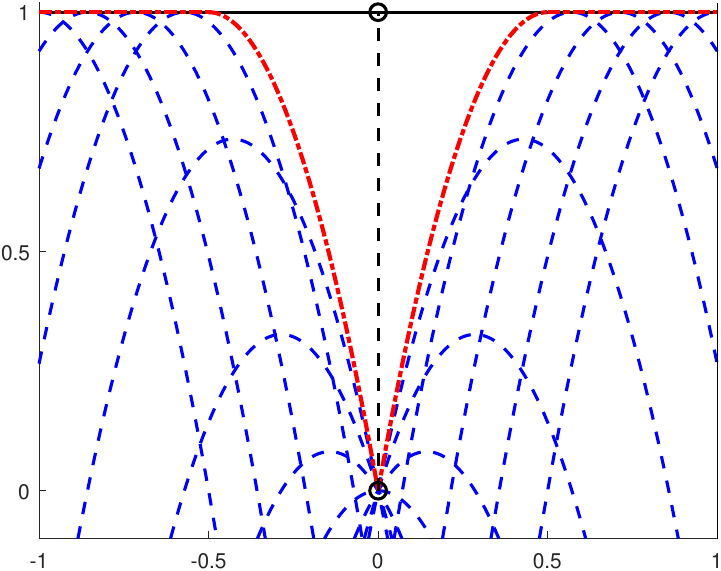}
\caption{Power function with $p=2$.}
\label{subfig:def-illu-p1.5}
\end{subfigure}
\begin{subfigure}[t]{0.45\textwidth}
\centering
\includegraphics[width=\linewidth]{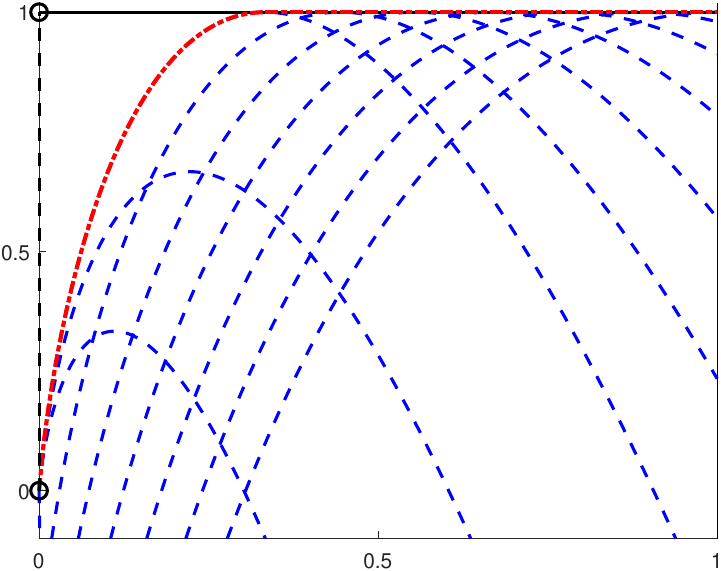}
\caption{Shannon entropy.}
\label{subfig:def_illu_shannon}
\end{subfigure}
\hfill
\begin{subfigure}[t]{0.45\textwidth}
\centering
\includegraphics[width=\linewidth]{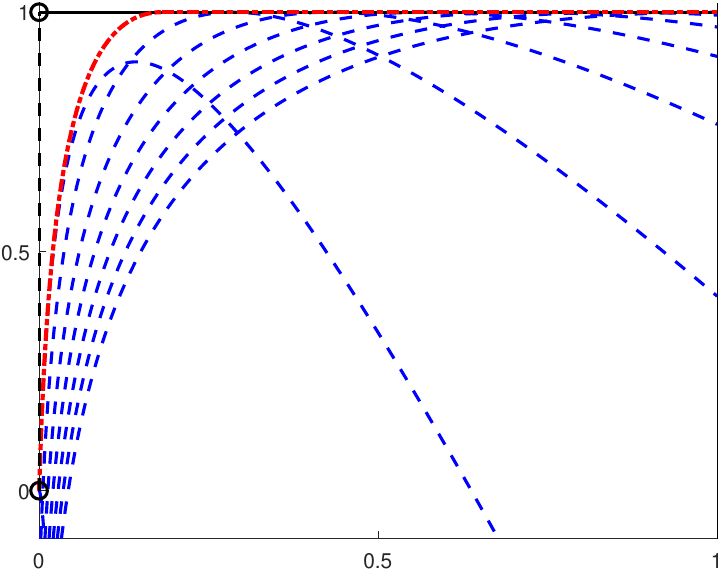}
\caption{KL.}
\label{subfig:def_illu_kl}
\end{subfigure}
\caption{One-dimensional geometrical illustration of the \BR{}  of Table~\ref{tab:B-rex}, over their respective domain $\Cc$. }
\label{fig:def_illu}
\end{figure}

\paragraph{Relation to generalized conjugates} 

We  follow \cite[Section 11.L]{RockWets98}  to draw connections between the proposed \BR{} penalty and a generalization of standard Fenchel conjugates.
 Consider any (not necessarily symmetric) function $\Phi: \Cc^N \times \Cc^N \to  \R \cup \left\lbrace \pm \infty \right\rbrace$. Generalized conjugates of a l.s.c. function~$g : \Cc^N \to \Cc \cup \left\lbrace \pm \infty \right\rbrace$ can be defined by simply changing the standard Fenchel coupling term $\left<\z,\x\right>$ to $\Phi(\x,\z)$ (or $\Phi(\z,\x)$).
 \begin{definition}[Generalized conjugate~\cite{RockWets98}]\label{def:general-fenchel-conjugate}
 Let $\Phi: \Cc^N \times \Cc^N \to \R \cup \left\lbrace \pm \infty \right\rbrace$. For $g: \R^N \to \R \cup \left\lbrace \pm \infty \right\rbrace$, the  $\Phi$-conjugate $g^{\Phi}$ of $g$ is defined by:
\begin{equation}\label{eq:Phi-conjugate-def}
    g^{\Phi}(\z) := \sup_{\x \in \Cc^N} ~ \Phi(\x,\z)-g(\x).
\end{equation}
\end{definition}

For instance, this notion of generalized conjugacy has been used to define a coupling for which the $\ell_0$ pseudo-norm equals its biconjugate~\cite{chancelier2021hidden}.
Here, we use it to introduce a generalized version of the $S$-transform proposed in~\cite[Section 3]{Carlsson2019}.
\begin{definition}
 Let $\Psi$ be defined as in Definition~\ref{def:Breg_L0}.
    The generalized $S_{\Psi}$ transform is the ($-D_{\Psi}$)-conjugate of $\lambda_0\|\cdot\|_0$, i.e. 
    \begin{equation}\label{eq:definition:s_transfrom}
   S_{\Psi}(\z):=\sup_{\x \in \Cc^N}~-\lambda_0 \|\x\|_0 - D_{\Psi}(\x,\z).
   \end{equation}
 \end{definition}
 
Note that an analogous notion which has been used mostly in the framework of optimal transport is the one of $c$-transform, see, e.g., ~\cite[Definition 1.8]{Ambrosio2013}. According to this definition, $S_\Psi$ is the \change{$(-D_{\Psi})_{-}$-conjugate} of the function $\lambda_0 \|\cdot\|_0$. Also, note that by considering a family of the form $\left\lbrace \gamma \psi \right\rbrace$ with $\gamma >0$, $S_{\Psi}$ corresponds to the left Bregman-Moreau envelope computed with constant $\gamma^{-1}$ \cite{Bauschke2017RegularizingWB}. 

\begin{proposition}\label{prop:Bregman-composition-of-S}
 Let $\Psi$ be defined as in Definition~\ref{def:Breg_L0}. Then, $S_{\Psi} : \Cc^N \to \R_{\leq 0}$ is a continuous function. Moreover, we have  
\begin{equation}\label{eq:B_Psi_composition}
     B_{\Psi} = S_{\Psi} \circ S_{\Psi}
     ,
 \end{equation}
that is:
\begin{equation}\label{eq:c-transform-brex}
 B_{\Psi}(\x)= \sup_{\u \in \Cc^N} \left( \inf_{\v \in \Cc^N} \lambda_0 \|\v\|_0+D_\Psi(\v,\u) \right)-D_\Psi(\x,\u)
 .
\end{equation}
 Finally, $B_\Psi : \Cc^N \to \R_{\geq 0}$ is continuous on $\operatorname{int}(\operatorname{dom}(B_\Psi))$.
 \end{proposition}
 \begin{proof}
 The proof is given in Appendix~\ref{appendix:Bregman-composition-of-S}.
 \end{proof}
 
In \cite{Carlsson2019} the fact that the quadratic envelope corresponds to a double $S$-transform and its link to generalized conjugates is widely exploited to demonstrate exact relaxation results. However, the extension of these proofs to the non-quadratic case and to Bregman divergences turned-out to be very challenging (quadratics often lead to handy simplifications that do not hold anymore for more complex cases). We thus took a different route, inspired by~\cite{soubies2017unified,soubies2020direction}, to prove the exact relaxations properties of the proposed \BR{} penalty.



\subsection{Characterization of the critical points of $J_\Psi$}

    This section is devoted to the characterization of critical points of $J_\Psi$. This will be useful to describe the local minimizers of $J_0$ which are eliminated by $J_\Psi$ (Proposition~\ref{Local_minimizers_of_J0_preserved_by_JPsi}). We start  by providing the expression of the Clarke's subdifferential $\partial B_\Psi$ in the unconstrained case ($\Cc = \R$). It can be computed directly from the generalized derivative of the one-dimensional \BR{}, thanks to its separability. For all $\x \in \R^N$, we indeed have $\partial B_\Psi (\x) =\prod_{n \in [N]} \partial  \beta_{\psi_n} (x_n)$ with 
\begin{equation}\label{eq:CP_beta}
\partial \beta_{\psi_n}(x)=\begin{cases}
    -\psi_n'(x)+\psi_n'\left(\alpha_n^-\right), &  \text{if} \quad x \in \left[\alpha_n^- , 0 \right),  \\
    -\psi_n'(x)+\psi_n'\left(\alpha_n^+\right), &  \text{if} \quad x \in \left( 0, \alpha_n^+ \right],
    \\
    \left[\ell^-_n  , \ell^+_n  \right], & \text{if} \quad x=0,
    \\
    0, & \text{otherwise},
\end{cases}
\end{equation}
where $\left[ \ell_n^- , \ell_n ^+ \right]= \left[ \psi_n'\left(\alpha_n^- \right) -\psi_n'(0),\psi_n'\left(\alpha_n^+ \right) -\psi_n'(0)\right]$. Details of the computation are given in Appendix~\ref{appendix:formula_clark}.

In the case $\Cc = \R_{\geq 0}$, one can deduce, following the same calculations as in Appendix~\ref{appendix:formula_clark}, that
$\partial \beta_{\psi_n}(0) = \{\ell_n^+\}$. Thus, the subdifferential of $\beta_{\psi_n}$ is given by:
\begin{equation}\label{eq:CP_beta_R+}
\partial \beta_{\psi_n}(x)=\begin{cases}
    -\psi_n'(x)+\psi_n'\left(\alpha_n^+\right), &  \text{if} \quad x \in \left[ 0, \alpha_n^+ \right],
    \\
    0, & \text{if} \quad x\geq \alpha_n^+
    .
\end{cases}
\end{equation}


\begin{proposition}[Critical points of $J_\Psi$]\label{prop:critical_point_Jpsi} The point $\hat{\x} \in \Cc^N$ is a critical point of~$J_\Psi$ if and only if, $\forall n \in [N]$,
\begin{equation}\label{eq:critical_point}
 \begin{cases}
-\left<\a_n, \nabla F_\y \left( \A\hat{\x}\right)\right> \in \left[\ell_n^- ,\ell_n^+  \right] &  \text{if}\quad \; \hat{x}_n = 0 ,\\
\left<\a_n, \nabla F_\y \left( \A\hat{\x}\right)\right>+\lambda_2\hat{x}_n-\psi'_n(\hat{x}_n)+\psi'_n \left( \alpha_n^- \right)=0 & \text{if}\quad \; \hat{x}_n \in \left[ \alpha_n^-, 0 \right),\\
     \left<\a_n, \nabla F_\y \left( \A\hat{\x}\right)\right>+\lambda_2\hat{x}_n-\psi'_n(\hat{x}_n)+\psi'_n \left( \alpha_n^+ \right)=0 & \text{if}\quad \; \hat{x}_n \in \left(0, \alpha_n^+ \right],\\
     \left<\a_n, \nabla F_\y \left( \A\hat{\x}\right)\right>+\lambda_2\hat{x}_n=0 & \text{if}\quad \; \hat{x}_n \in \R \backslash \left[ \alpha_n^- , \alpha_n^+ \right]
 .    
 \end{cases}
\end{equation}
In the case $\Cc=\R_{\geq 0}$, we have $\alpha_n^-=0$ and $\ell_n^-=-\infty$.
\end{proposition}
\begin{proof}
The proof is given in Appendix~\ref{appendix:charachterization_critical_points}.
\end{proof}

\subsection{Exact relaxations properties}\label{sec:exact_relax}

In this section, we study the relaxed functional $J_\Psi$ in~\eqref{eq:relaxation_Jpsi} obtained by replacing in \eqref{eq:problem_setting} the $\ell_0$ pseudo-norm with the \BR{} penalty. In particular, we are interested in studying the relationship between the minimizers of $J_\Psi$ and those of $J_0$. We first study the existence of global minimizers for $J_\Psi$ in the following theorem. 

\begin{theorem}[Existence of solutions to the relaxed problem]\label{coro:existence_min_global_jpsi}
 Let $F_\y$ be  coercive or $\lambda_2>0$. Then, the solution set of the relaxed problem with~$J_\Psi$  is nonempty.
\end{theorem}
\begin{proof}
    The proof come directly by taking $\Phi(\cdot)=B_{\Psi}(\cdot)$ in Theorem~\ref{th:existence-general}.
\end{proof}

 We now provide a sufficient condition, referred to as \emph{concavity-condition}, on the Bregman generating functions $\psi_n$ to ensure that $J_\Psi$ is a continuous exact relaxation of $J_0$.
It reads,  for all $n \in [N]$ and $\x \in \Cc^N$ 
\begin{equation}\label{eq:general-exact-condition}
 g(t):=J_\Psi(\x^{(n)} + t \e_n) \; \text{is strictly concave on} \; (\alpha^-_n, 0) \; {\text{and}} \; (0, \alpha_n^+).
 \tag{CC}
\end{equation}
Since $f(\cdot;y)$ and $\psi_n$ are assumed to be twice differentiable on $\operatorname{int}(\Cc)$ and using the definition of \BR{} (Proposition~\ref{prop:Breg_l0}), the condition \eqref{eq:general-exact-condition} can be translated into its second order characterization which is
\begin{equation}\label{eq:2-diff-c}
g''(t)=\frac{\partial^2}{\partial t^2 } F_\y\left(\A\left(\x^{(n)}+t\e_n\right)\right) -\psi_n''(t)+\lambda_2 < 0 \; \text{for all} \; t \in (\alpha^-_n, 0) \cup (0, \alpha_n^+)
.
\end{equation}
\begin{theorem}[Sufficient condition for exact relaxation]\label{th:exact_relax}
    Let the family $\Psi = \{\psi_n\}_{n \in[N]}$ be such that condition~\eqref{eq:general-exact-condition} holds.
    Then, the functional $J_\Psi$ defined  in~\eqref{eq:relaxation_Jpsi}
     is an exact continuous relaxation of $J_0$. In other words, we have
     \begin{align}
         & \argmin_{\x \in \Cc^N} J_\Psi(\x) = \argmin_{\x \in \Cc^N} J_0(\x) \label{eq:result_global_min}\\
         & \hat{\x} \text{ local (not global) minimizer of } J_\Psi \; \Longrightarrow \; \hat{\x} \text{ local (not global) minimizer of } J_0 \label{eq:result_local_min}
     \end{align}
     Moreover, for each local minimizer $\hat \x$ of $J_\Psi$, we have $J_\Psi(\hat{\x}) = J_0(\hat{\x})$.
   
   \end{theorem}
\begin{proof}
     The proof is deferred to Appendix~\ref{appendix:exact_relax}.
\end{proof}

The main idea behind the proposed condition \eqref{eq:general-exact-condition} 
is to set $\Psi$ such that the relaxation $J_\Psi$ cannot have minimizers with components within the intervals $(\alpha_n^-,0)$ and $(0,\alpha_n^+)$, for all $n \in [N]$, due to the strict concavity. As such, minimizers $\hat \x \in  \Cc^N$ of $J_\Psi$ are such that $J_\Psi(\hat{\x}) = J_0(\hat{\x})$ which, combined with the fact that $J_\Psi \leq J_0$ (by definition of $J_\Psi$), allows us to prove the stated result.

By a straightforward computation, for all $t \in (\alpha_n^-,\alpha_n^+) \backslash \{0\}$ we have $g''(t)=\sum_{m=1}^M a_{mn}^2 f''([\A\x^{(n)}]_m + t a_{mn};y_m)-\psi_n''(t)+\lambda_2$. Thus, one  can note that it suffices to set $\Psi$ such that 
 \begin{equation}\label{eq:sup_inf_cond}
 \forall n \in [N], \; \inf_{t \in (\alpha_n^-,\alpha_n^+) \backslash \{0\}} \psi_n''(t) >  \lambda_2+\sum_{m=1}^M a_{mn}^2 \sup_{z \in \mathrm{dom}(f(\cdot;y_m))} f''(z;y_m),
 \end{equation}
to impose the concavity of \BR{} to be  strictly larger than the convexity of $F_\y(\A \cdot) + \frac{\lambda_2}{2}\|\cdot\|^2$ on the required intervals. Condition~\eqref{eq:sup_inf_cond} is clearly coarser than~\eqref{eq:general-exact-condition}, although it presents the advantage of being easier to manipulate as it decouples quantities related to the data fidelity term to those associated to the regularizer. Although we
 will  resort to the use of this simplified condition in the following numerical illustrations, it is worth mentioning that, for a specific data fidelity term, better relaxations can be achieved by carefully selecting an appropriate family of $\mathbf{\Psi}$ functions exploiting condition~\eqref{eq:general-exact-condition} more tightly (a task we leave for future work).   All quantities involved in the exact relaxation condition~\eqref{eq:sup_inf_cond} for the data fidelity terms $F_\y$ of Table~\ref{tab:data_fidelty} and the generating functions $\psi_n$ of Table~\ref{tab:generating-functions} are gathered in Table~\ref{tab:exact-conditions}.

\begin{table}
\renewcommand{\arraystretch}{1.7} 
    \centering
    \begin{tabular}{|c|c||c|c|}
    \hline
  \cellcolor[HTML]{FFFFC7}     $\psi_n$  & \cellcolor[HTML]{FFFFC7}  
 $\inf\limits_{t \in (\alpha_n^{-},\alpha_n^{+}) \setminus \{0\}} \psi_n''(t)$ &
 \cellcolor[HTML]{FFFFC7} $ F_\y$ &  \cellcolor[HTML]{FFFFC7}  $\lambda_2+\sum_{m=1}^M a_{mn}^2 \sup_{z \in \Cc} f''(z;y_m)$ \\ \hline
    Power Function  & $\gamma_n^{2/p}\left(p\lambda_0 \right)^{\frac{p-2}{p}}, \; p \in ]1,2]$   & 
    LS  & $\lambda_2 + \|\a_n\|_2^2$\\ \hline
    Shannon Entropy & $\frac{\gamma_n^2}{\lambda_0}$  
    & LR  & $\lambda_2 + \frac14 \|\a_n\|_2^2$  \\ \hline
    Kullback-Leibler  & $\frac{\gamma_n}{b^2}W^2(-b\e^{-\kappa})$  
    & KL  & $\lambda_2 + \sum_{m=1}^M a_{mn}^2 \frac{y_m}{b^2}$\\ \hline
    \end{tabular}
    \caption{Quantities required for condition \eqref{eq:sup_inf_cond}  for the data terms $F_\y$ in Table~\ref{tab:data_fidelty} and the generating functions $\psi_n$ in Table~\ref{tab:generating-functions}. For the Kullback-Leibler case $W(\cdot)$ denotes the Lambert function and $\kappa=\frac{\lambda_0}{y\gamma_n}+\log(b)+1$.}
    \label{tab:exact-conditions}
\end{table}

\begin{coro}[Strict local minimizers of $J_\Psi$]\label{coro:stric_glob_min_Jpsi}
Under the condition \eqref{eq:general-exact-condition} 
the set of strict local minimizers of $J_\Psi$ is included in the set of strict local minimizers of~$J_0$. Moreover, global minimizers of $J_\Psi$ are strict.
\end{coro}
\begin{proof}
    The first statement is a direct consequence of Theorem~\ref{th:exact_relax} together with the fact that $J_\Psi \leq J_0$. The second statement is a consequence of Theorems~\ref{th:glob_min_J0_strict}  and~\ref{th:exact_relax}.
\end{proof}

\begin{remark}
    We can let the inequality in~\eqref{eq:2-diff-c} (equivalently in~\eqref{eq:sup_inf_cond}) be non-strict. In that case, as with the CEL$0$ penalty for quadratic data terms~\cite{Soubies2015}, we can map each minimizer (local and global) $\hat \x \in \Cc^N$ of the relaxation $J_\Psi$ to one of $J_0$ by thresholding to $0$ each  component $\hat{x}_n$ that belongs to the interval $(\alpha_n^-,\alpha_n^+) \setminus \{0\}$,  $n \in [N]$ (if any). In other words, $\hat{\x}_0 \in \Cc^N$ defined as
    \begin{equation}
        \forall n \in [N], \; [\hat{\x}_0]_n = \left\lbrace
        \begin{array}{ll}
          0   & \text{ if } \hat{x}_n \in (\alpha_n^-,\alpha_n^+) \setminus \{0\} \\
          \hat{x}_n   & \text{ otherwise,}
        \end{array}\right.
    \end{equation}
    is a minimizer of $J_0$. This is due to the fact that, with such a relaxed (non-strict) condition~\eqref{eq:general-exact-condition}, some 1D restrictions of $J_\Psi$ to a variable $x_n$ may be constant (instead of strictly concave) over $[\alpha_n^-,0)$ or $(0,\alpha_n^+]$. It is worth mentioning that, in this case, Corollary~\ref{coro:stric_glob_min_Jpsi} does not hold anymore.
\end{remark}

From Theorem~\ref{th:exact_relax}, we get that $J_\Psi$ can remove some local (not global) minimizers of $J_0$. In Proposition~\ref{Local_minimizers_of_J0_preserved_by_JPsi}, we derive a necessary and sufficient condition for a local minimizer of $J_0$ to be preserved by $J_\Psi$. 

\begin{proposition}[Local minimizers of $J_0$ preserved by $J_\Psi$]\label{Local_minimizers_of_J0_preserved_by_JPsi} Let $\hat{\x}$ be a local minimizer of $J_0$. Then, under condition \eqref{eq:general-exact-condition}
, $\hat \x$ is a local minimizer of $J_\Psi$ if and only if 
\begin{align}
  &\forall n \in \sigma(\hat{\x}), \; \hat{x}_n \in \Cc \backslash  [\alpha_n^-,\alpha_n^+], \label{eq:condition1} \\
  & \forall n \in \sigma^c(\hat{\x}), \; -\left<\a_n ,\nabla F_\y\left(\A\hat{\x} \right) \right>  \in \left[ \ell_n^-,\ell_n^+  \right],
  \label{eq:condition2}
\end{align}
where $\ell_n^-=-\psi_n'(0)+\psi_n'(\alpha_n^-)$ in the case $\Cc=\R$ and $\ell_n^-=-\infty$ in the case $\Cc=\R_{\geq 0}$.

Moreover, $\hat{\x}$ is a strict local minimizer of $J_\Psi$ (and thus of $J_0$) if and only if, in addition to the above two conditions, $\lambda_2 >0$ or $\mathrm{rank}(\A_{\hat{\sigma}}) = \sharp \hat{\sigma}$.
\end{proposition}
\begin{proof}
The proof is presented in~Appendix \ref{appendix:jpsi-removes-loc}.
\end{proof}

Hence $J_\Psi$ eliminates all the local minimizers of $J_0$ having at least one non-zero component within an interval $[\alpha_n^-,\alpha_n^+]$ or for which at least one partial derivative of the data fidelity term $F_\y$ associated to an off-support variable (i.e., $n \in \sigma^c(\hat{\x})$) has a too large amplitude (outside of $\left[ - \ell_n^+, -\ell_n^- \right]$).

To conclude this section, we provide in Corollary~\ref{coro:critical_poin_to_loc_min} a necessary and sufficient condition to recognize a local minimizer of $J_\Psi$. This is of practical interest as, in general, on the shelf non-convex optimization algorithms only ensure the convergence to critical points (necessary but not sufficient condition to be a local minimizer).
\begin{coro}\label{coro:critical_poin_to_loc_min}
    Under the condition \eqref{eq:general-exact-condition}, $\hat \x \in \Cc^N$ is a local minimizer of $J_\Psi$ if and only if $\hat \x$ is a critical point of $J_\Psi$ and $\forall n \in \sigma(\hat{\x}), \; \hat{x}_n \in \Cc \backslash  [\alpha_n^-,\alpha_n^+]$. Moreover it is a strict local minimizer  if and only if, in addition to the above  condition, $\lambda_2 >0$ or $\mathrm{rank}(\A_{\hat{\sigma}}) = \sharp \hat{\sigma}$.
\end{coro}
\begin{proof}
    The implication ($\Longleftarrow$) is already demonstrated in the proof of Proposition~\ref{Local_minimizers_of_J0_preserved_by_JPsi} (first bullet). Regarding the reverse implication ($\Longrightarrow$), we have that if $\hat \x \in \Cc^N$ is a local minimizer of $J_\Psi$ then it is a critical point. Moreover, under the condition \eqref{eq:general-exact-condition}, we have $\forall n \in \sigma(\hat{\x}), \; \hat{x}_n \in \Cc \backslash  (\alpha_n^-,\alpha_n^+)$. Finally, the fact that the points $\alpha_n^-$ and  $\alpha_n^+$ are also excluded is due to fact that the 1D restrictions in condition~\eqref{eq:general-exact-condition} 
     are convex over $(-\infty, \alpha_n^-]$ and $[\alpha_n^+, +\infty)$. As such, if either  $\alpha_n^-$ or  $\alpha_n^+$ is a critical point for one of these 1D restrictions, then it is a saddle point at the interface between a strictly concave region and a convex region of this 1D restriction. It cannot thus be a local minimizer.
\end{proof}

\change{
\paragraph{A boarder class of exact penalties} As already mentioned, our primary motivation for \BR{} is to adapt the geometry of the penalty to the one of the data term through the choice of the generating function $\psi_n$ (see also Proposition~\ref{prop2_section2}). Yet, one may observe that the exact relaxation result of Theorem~\ref{th:exact_relax} does not explicitly depends on $\psi_n$. In Corollary~\ref{coro:exact_folded} below, we thus report that under some assumptions more general folded concave penalties~\cite{Fan2012STRONGOO} can lead to exact relaxations.

\begin{coro}\label{coro:exact_folded}
    Let $J_\Phi(\x) := F_\y(\A\x) + \Phi(\x) + \frac{\lambda_2}{2}\|\x\|^2$, where $\Phi(\x) = \sum_{n=1}^N \phi_n(x_n)$ and each $\phi_n: \Cc \to \R$ is a folded concave penalty satisfying $\phi_n(0)=0$ and $\phi_n(x) = \lambda_0$ for all $x \in \Cc \setminus (\eta_n^-, \eta_n^+)$ and some $\eta^-_n<0$ and $\eta^+_n>0$. If in addition we have that 
    \begin{equation}\label{eq:general-exact-condition}
        g(t) := J_\Phi(\x^{(n)} + t \e_n) \text{ is strictly concave on } (\eta_n^-, 0) \text{ and } (0, \eta_n^+),
    \end{equation}
    then $J_\Phi$ is an exact continuous relaxation of $J_0$.
\end{coro}
\begin{proof}
    By assumptions, the set of global minimizers of $J_\Phi$ is nonempty. This is trivial for $\lambda_2>0$ and given by Theorem~\ref{th:existence-general} for $\lambda_2=0$. Then the proof follows exactly the same arguments as the ones of Theorem~\ref{th:exact_relax}.
\end{proof}
}

\section{On the choice of the Bregman distance}\label{sec:choices}

In this section, we comment on the choice of the family $\Psi$ of generating functions used to define \BR{} (Definition \ref{def:Breg_L0}). We distinguish two cases. The first one corresponds to the situation where $\A$ is diagonal. Here, we show that, for a suitable choice of $\Psi$, the proposed relaxation $J_\Psi$ is nothing but the convex envelope of $J_0$, that is, the largest convex function bounding $J_0$ from below. Then, we provide few comments on the more complex case  where  $\A$ is not diagonal.


\subsection{Case $\A$ diagonal: Link with convex envelopes. }

Without loss of generality, we only discuss  the case $\A = \mathbf{I}$. Here, $J_0$ is separable and we thus restrict our analysis to the 1D case by drooping the index $n$.
In Proposition~\ref{prop2_section2} we prove that upon a particular choice of the generating function $\psi$ consistent with the functional made of the data fidelity term plus the ridge regularization, $J_\psi$ is indeed the convex envelope of $J_0$. 

\begin{proposition}\label{prop2_section2}
Let $\gamma>0$ and set $\psi(\cdot):=\gamma\left(f(\cdot;y)+\frac{\lambda_2}{2}|\cdot|^2\right)$. Then,  $\beta_{\psi}(\cdot)+\gamma\left( f(\cdot;y)+\frac{\lambda_2}{2}|\cdot|^2 \right)$ is the l.s.c.~convex envelope of $\lambda_0  |\cdot|_0+\gamma\left( f(\cdot;y)+\frac{\lambda_2}{2}|\cdot|^2 \right)$, that is:
\begin{equation} \label{eq:prposition2_envelope_convex}
\left( \lambda_0 | \cdot |_0 + \gamma\left( f(\cdot;y)+\frac{\lambda_2}{2}|\cdot|^2 \right) \right) ^{**} (x)= \beta_{\psi} (x)+\gamma\left( f(x;y)+\frac{\lambda_2}{2}x^2 \right),\quad\forall x\in \R
.
\end{equation}
\end{proposition}
\begin{proof}
    To prove this proposition, we will use the converse of the result in~\cite[Theorem 4.1]{Carlsson2019}, which states that if for an l.s.c function  $g$, its l.s.c convex envelope $g^{**}$ is coercive, then there exists a unit vector $\nu$ and $t_0 > 0$ such that the function $t\mapsto g^{**}(x + t\nu)$ is affine on $(-t_0, t_0)$ for all $x$ such that $g(x) \neq g^{**}(x)$. Let  $\psi(\cdot)=\gamma\left(f(\cdot;y)+\frac{\lambda_2}{2}x^2\right)$ for $\gamma>0$.
Then, from~\eqref{eq:prop1_equa2}, we get that
\begin{equation*}
\beta_\psi(x) + \gamma \left( f(x;y)+\frac{\lambda_2}{2}x^2\right)= \left\lbrace 
    \begin{array}{ll}
       \gamma \left(  f(0;y) +  \left(f'(\alpha^\pm;y)+\lambda_2 \alpha^\pm \right)x\right),  & \text{ if } x \in [\alpha^-,\alpha^+],
       \\
      \lambda_0 |x|_0 +\gamma \left( f(x;y)+\frac{\lambda_2}{2}x^2\right),   &  \text{ otherwise}.
    \end{array}
    \right.
\end{equation*}
The function is affine with respect to $x$ on $[\alpha^-,\alpha^+]$, hence all the conditions of \cite[Theorem 4.1]{Carlsson2019} are satisfied. Therefore, $\beta_{\psi}(\cdot)+\gamma\left( f(\cdot;y)+\frac{\lambda_2}{2}|\cdot|^2 \right)$ is the l.s.c convex envelope of  $\lambda_0  |\cdot|_0+\gamma\left( f(\cdot;y)+\frac{\lambda_2}{2}|\cdot|^2 \right)$. 
\end{proof}

Intuitively, this result comes from the fact that, for this specific choice of $\psi$, the concavity of $\beta_\psi$ on $(\alpha^-,0)$ and $(0,\alpha^+)$ exactly matches the convexity of $f(\cdot;y) + \frac{\lambda_2}{2}(\cdot)^2$ so as to make the sum linear. 



\subsection{Case $\A$ not diagonal: Discussion} \label{sec:Anotdiag_disc}

For an arbitrary matrix $\A$, 
concavity-condition \eqref{eq:general-exact-condition} is a condition under which exact relaxation properties are ensured for $J_\Psi$. In this case,  the choice of the family $\Psi$ of generating functions can be made according to different, not necessarily compatible, criteria.  On the one hand, it is advisable to select a family $\Psi$ such that the relaxed functional $J_\Psi$ removes as many local minimizers of $J_0$ as possible. From Proposition~\ref{Local_minimizers_of_J0_preserved_by_JPsi}, this would require to select $\Psi$  that leads, under the exact relaxation condition \eqref{eq:general-exact-condition}
,  to the largest interval $[\alpha_n^-,\alpha_n^+]$ and the smallest interval $[\ell_n^-,\ell_n^+]$.
On the other hand, a more practical criterion is that \BR{} should be computable, which can be difficult for general functions $\psi_n$ since, as shown in 
Proposition~\ref{prop:Breg_l0}, getting a closed form expression requires to solve the equation $d_{\psi_n}(0,z)=\lambda_0$ in order to get the values $\alpha^{\pm}_n$. Finally, although enjoying better properties than $J_0$, the relaxation $J_\Psi$ is still non-convex. As such, our capacity to avoid local minimizers is not only due to the landscape of $J_\Psi$ but also to the non-convex optimization algorithm we consider. There is no theoretical guarantee that a given algorithm will perform better  (i.e., reach a stationary point with a lower objective value)  in minimizing the exact relaxation that removes the largest amount of  local minimizers. Yet, in our experiments, we  observe the superiority of relaxations that remove a larger amount of local minimizer of $J_0$ (cf. Sections~\ref{sec:numeric_algo} and~\ref{sec:high_dim_ex}).

\section{Proximal gradient algorithm for minimizing $J_\Psi$}\label{sec:algo}

The link between the minimizers of the original and the relaxed problems (Theorem~\ref{th:exact_relax}) motivates us to address Problem~\eqref{eq:problem_setting} by minimizing $J_\Psi$, which possesses better optimization properties. Although non-convex, $J_\Psi$ is continuous, which makes it amenable to be minimized by means of standard non-convex optimization algorithms. For instance, $J_\Psi$ enjoys the structural properties needed to apply  Majorization-Minimization techniques, such as the Iterative Reweighted $\ell_1$ algorithm \cite{Ochs2015OnIR} and, for some choices of $\Psi$, Difference of Convex Functions (DC) algorithms, see, e.g., \cite{LeThi2014FeatureSI}.

In this section, we  describe how the proximal gradient algorithm can be deployed efficiently to minimize $J_\Psi$. For that, the computation of the proximal operator of the proposed \BR{}  penalty in correspondence with different families  $\Psi$ (see Definition~\ref{def:Breg_L0}) is required. In the following, we derive a general formula of the proximal operator of $B_\Psi$ and, 
in correspondence of some of the generating functions listed in Table~\ref{tab:generating-functions}, we then compute explicitly the corresponding proximal operators. 

For $k\geq 0$ we start recalling the proximal gradient iteration~\cite{Attouch2013ConvergenceOD} applied to minimize $J_\Psi$, which reads:
\begin{equation}\label{eq:prox-grad-algo}
\x^{k+1} \in \prox_{\indic_{\mathcal{C}^N}+\rho B_{\Psi}} \left(\x^k -\rho \left(\A^T \nabla F_\y(\A\x^k) + \lambda_2 \x^k\right)\right)
,
\end{equation}
where $\rho>0$ is a step-size and the (possibly multi-valued) proximal operator of $\indic_{\mathcal{C}^N}+\rho B_\Psi$ is defined by:
\begin{equation}\label{prox-def}
\prox_{\indic_{\mathcal{C}^N}+\rho B_\Psi}(\x)= \argmin_{\x \in \Cc^N} \left\lbrace B_\Psi(\z)+\frac{1}{2\rho} \|\z-\x\|^2 \right\rbrace    
.
\end{equation}

When $J_\Psi$ satisfies the Kurdyka-\L ojaseiwicz property, a constant step-size  $0 < \rho < 1/L$, with $L$  the Lipschitz constant of the gradient of $F_\y(\A\cdot) + \frac{\lambda_2}{2} \|\cdot\|_2^2$, ensures the convergence of the generated sequence $\left\lbrace \x^k \right\rbrace_{k \in \N}$ to a critical point of  $J_\Psi$~\cite[Theorem 5.1]{Attouch2013ConvergenceOD}.  While this property should be checked for each specific instance of Problem~\eqref{eq:problem_setting}, it is known to be verified by a rich class of functions~\cite{attouch2010proximal}. As an alternative, in some of our numerical experiments  (see Section \ref{sec:high_dim_ex}) we resorted to a backtracking strategy to estimate the step-size at each iteration and improve convergence speed. 


Recalling $\Cc \in \{ \R, \R_{\geq 0} \}$, since $\indic_{\mathcal{C}^N}$   and $B_\psi$ are both separable 
 and $B_\psi$ is a non-negative symmetric function, there holds $\prox_{\indic_{\mathcal{C}}+\rho \beta_{\psi_n}}=\operatorname{proj}_{\indic_{\mathcal{C}}}\circ ~\prox_{\rho \beta_{\psi_n}}$. As such,  the computation of \eqref{prox-def} can be addressed by solving the following 1D problem 
\begin{equation}\label{prox-def-1D}
\prox_{\rho \beta_{\psi_n}}(x)= \argmin_{z \in \R} \left\lbrace \beta_{\psi_n}(z)+\frac{1}{2\rho} \|z-x\|^2 \right\rbrace
,
\end{equation}
since,  for all $\x \in \R^N$, there holds
 \begin{equation}
        \prox_{\indic_{\mathcal{C}^N}+\rho B_{\Psi}}(\mathbf{x}) =\left(\operatorname{proj}_{\indic_{\mathcal{C}}}~\circ ~\prox_{\rho\beta_{\psi_1}}(x_1),\dots, \operatorname{proj}_{\indic_{\mathcal{C}}}~\circ ~\prox_{\rho\beta_{\psi_N}}(x_N) \right)
      .
 \end{equation}
 The following proposition provides a general formula of the proximal operator of \BR{}.
\begin{proposition}[Proximal operator]\label{prop:prox_formula}
   Let $\rho>0$ and $n \in [N]$. For $x \in \R$, the proximal operator of $\beta_{\psi_n}$ is given by 
   \begin{equation}\label{eq:prox-formula-general}
\prox_{\rho \beta_{\psi_n}}(x)= \argmin_{u \in \mathcal{U}(x)} \left\lbrace \beta_{\psi_n}(u)+\frac{1}{2\rho} \|u-x\|^2 \right\rbrace
,
\end{equation}
 where $\mathcal{U}(x)=\left\lbrace 0,x\right\rbrace \cup S_x$ with $S_x=\{u \in \R : u-\rho\psi_n'(u)=x-\rho\psi'_n(\alpha_n^\pm)\}$ and $\alpha^\pm_n$ are defined in Proposition~\ref{prop:Breg_l0}.
\end{proposition}
 \begin{proof}
     The proof is given in Appendix~\ref{appendix:bregman_prox}.
 \end{proof}

 Below, we compute the set $S_x$ defined in Proposition~\ref{prop:prox_formula} for the specific choices of  generating functions given in Table~\ref{tab:generating-functions}. The details of computations can be found in Appendix~\ref{appendix:expl-prox}.
\begin{example}[Power functions]\label{expl:prox-power-func}
Let $n \in [N]$. If $\psi_n$ is defined as a $p$-power function with $p \in (1,2]$, then $S_x$ is the set containing the solutions $u\in\mathbb{R}$ to
\begin{equation}
u-\frac{\rho\gamma_n}{p-1}\operatorname{sign}(u) |u|^{p-1}=x-\rho\psi_n'(\alpha^\pm_n),
\end{equation}
where taking $\alpha_n^-$ or $\alpha_n^+$ depends  on the sign of $x$. In particular, we have the following:
\begin{enumerate}[label=\roman*)]
\item  If $p=2$: 
$$
    S_x=\left\lbrace\frac{x-\rho\psi_n'(\alpha^\pm_n)}{1-\rho\gamma_n}\right\rbrace
$$
\item If $p=3/2$:  $S_x=\left\lbrace x_1, x_2 \right\rbrace \cap \R$, with 
\begin{align*}
    & x_1=x-\rho\psi_n'(\alpha_n^\pm)\pm 2(\rho\gamma_n)^2+ 2\rho\gamma_n  \sqrt{(\rho\gamma_n)^2+x-\rho\psi_n'(\alpha_n^\pm)} \\
    & x_2=x-\rho\psi_n'(\alpha_n^\pm)\pm 2(\rho\gamma_n)^2-2\rho\gamma_n  \sqrt{(\rho\gamma_n)^2+x-\rho\psi_n'(\alpha_n^\pm)}
\end{align*}
\item If $p=4/3$: $S_x=\left\lbrace \pm x_1^3,\pm x_2^3 , \pm x_2^3 \right\rbrace \cap \R $, where 
$$x_1=A+B, \quad x_2=\omega A+ \omega^2 B \quad \text{ and } \quad x_3=\omega^2 A+ \omega B$$ with $ 
A=\sqrt[3]{\frac{\pm(x-\rho\psi_n'(\alpha_n^\pm))}{2}+\frac{1}{2}\sqrt{\Delta}}
$, $ 
B=\sqrt[3]{\frac{\pm(x-\rho\psi_n'(\alpha_n^\pm))}{2}-\frac{1}{2}\sqrt{\Delta}}$, $\omega=-\frac{1}{2}+i\frac{\sqrt{3}}{2}$
and  $\Delta=(x-\rho\psi_n'(\alpha_n^+))^2-4 (\rho\gamma)^3$. Taking $-x_i^3$ or $x_i^3$ for $i \in \{1,2,3\}$ (similarly $-(x-\rho\psi_n'(\alpha_n^-))$ or $(x-\rho\psi_n'(\alpha_n^+)$) depends on the  sign as $x$.
\end{enumerate}
\end{example}
\begin{example}[Shannon entropy]\label{expl:prox-shannon}
    Let $n \in [N]$. If $\psi_n$ is defined as the Shannon entropy, we have that
$$
    S_x=\left\lbrace -\rho\gamma_n W_{-1}\left(-\frac{1}{\rho \gamma_n}   e^{-\frac{x-\rho\psi_n'(\alpha_n^+)}{\rho \gamma_n}}\right), -\rho\gamma_n W_0\left(-\frac{1}{\rho \gamma_n}   e^{-\frac{x-\rho\psi_n'(\alpha_n^+)}{\rho \gamma_n}}\right) \right\rbrace \cap \R_{\geq 0}
    .
$$
where $W_0(\cdot)$ and  $W_{-1}(\cdot)$ are the principal and the negative branches of the Lambert function, respectively.
\end{example}

\begin{example}[KL]\label{expl:prox_kl} Let $n \in [N]$. If $\psi_n$ is defined as the Kullback-Leibler divergence with $b>0$ and $y>0$, we have that 
$$
   S_x=\left\lbrace \frac{1}{2}\left(x+\rho\gamma_n-\rho\psi_n'(\alpha_n^+)-b \pm \sqrt{\Delta(x)}\right) \right\rbrace \cap \R_{\geq 0}
   ,
$$
where $\Delta(x)=(x+\rho\gamma_n-\rho\psi_n'(\alpha_n^+)-b)^2+4(bx+\rho\gamma_n b-b\rho\psi_n'(\alpha_n^+)-y\rho\gamma_n)$.
\end{example}

In Proposition~\ref{prop:prox_formula}, the set $\mathcal{U}(x)$ contains candidate solutions for the prox problem. The global one can be easily computed by comparing the associated objective values. This would lead to (at most) two thresholds corresponding to the different distinctive areas of the prox, namely, $\mathrm{prox}(x) = 0$ near $0$, $\mathrm{prox}(x) = x$ for large $x$, and a value of $\mathcal{S}_x$ for intermediate values of $x$. In Corollary~\ref{coro:prox-cases} we highlight two regimes, depending on $\rho$ for which these thresholds can be made more explicit. 
\begin{coro}\label{coro:prox-cases}
Let $\rho>0$, $n \in [N]$ and $x\in\R$. We distinguish the following cases: 
\begin{enumerate}[label=\roman*)] 
\item If $\underaccent{\bar}{\theta}=\inf_{{t \in (\alpha_n^-,\alpha_n^+)\backslash\{0\}}} \psi_n''(t) > \frac{1}{\rho}$, then the proximal operator of $\rho \beta_{\psi_n}$  is given by
\begin{equation}\label{eq:prox-l0}
   \prox_{\rho\beta_{\psi_n}}(x)=\prox_{\lambda_0 |\cdot|_0}(x)= x \mathbb{1}_{\{|x|>\sqrt{2\rho\lambda_0}\}}+\{0,x\}\mathbb{1}_{\{|x|=\sqrt{2\rho\lambda_0}\}}
   .
\end{equation}
\item If $\bar{\theta}=\sup_{t \in \Cc } \psi_n''(t)$ exists and $\bar\theta < \frac{1}{\rho}$, then the proximal operator of $\rho \beta_{\psi_n}$  is continuous and given by
\begin{equation}\label{eq:prox-continue}
\prox_{\rho\beta_{\psi_n}}(x)= \operatorname{sign}(x) \min \left(|x|, \left(  \left(\operatorname{id}-\rho\psi_n'\right)^{-1} \left(x-\rho \psi_n'(\alpha^+_n) \right)\right)_{+} \right)
.
\end{equation}
where $(x)_+ = \max(x,0)$, $\operatorname{id}$ stands for the identity map and the values $\alpha_n^{\pm}$ are defined as in Proposition \ref{prop:Breg_l0}.
\end{enumerate}
\end{coro}
\begin{proof}
  Let $n \in [N]$, $\rho>0$ and $x \in \R$. Let $g : \Cc \to \R$ be defined by $g(u):=\beta_{\psi_n}(u)+\frac{1}{2\rho}(u-x)^2$.
  \begin{itemize}
      \item Proof of i).  From Proposition~\ref{prop:Breg_l0}, we can easily see that for $\underaccent{\bar}{\theta} > \frac{1}{\rho}$, the functional $g$ is concave on the interval $ (\alpha_n^-,\alpha_n^+)\backslash\{0\}$, hence the set $\mathcal{U}(x)$ defined in Proposition~\ref{prop:prox_formula} is reduced to $\{0,x\}$ with $g(x)=\lambda_0$ and $g(0)=\frac{1}{2\rho}x^2$. We deduce the formula~\eqref{eq:prox-l0} through the comparison of these objective values.
       \item Proof of ii). If $\bar\theta < \frac{1}{\rho}$, the functional $g$ is globally convex. Therefore by the definition of the proximal operator, $\prox_{\beta_{\psi_n}}(\cdot)$ is continuous, and $S_x$ defined in Proposition~\ref{prop:prox_formula} is reduced to the singleton $u^*$ given by $u^*=\left(\operatorname{id}-\rho\psi_n'\right)^{-1}\left(x-\rho\psi'_n(\alpha_n^\pm)\} \right)$ with $\mathcal{U}(x)=\{0,x,u^*\}$, which yields~\eqref{eq:prox-continue}.
  \end{itemize}
\end{proof}
In Corollary~\ref{coro:prox-cases}, we observe that if $\underaccent{\bar}{\theta}> \frac{1}{\rho}$, the proximal operator of \BR{} is discontinuous and equals the proximal operator of the $\ell_0$ pseudo-norm, i.e.~the hard-thresholding operator. Conversely, when the second derivative of the  generating function $\psi_n$ attains a supremum  $\bar\theta$, by selecting $\rho$ such that $\bar\theta < \frac{1}{\rho}$ results in a continuous proximal operator with three pieces.
Finally, choosing  $\underaccent{\bar}{\theta}< \frac{1}{\rho} < \bar\theta$ (or, simply, $\underaccent{\bar}{\theta}< \frac{1}{\rho}$ in cases where $\bar\theta$ does not exist) leads to a proximal operator which is discontinuous at the transition from $0$ to the non-zero values of the prox. 
We illustrate all these cases in Figure~\ref{fig:plot-prox-functions}, where we plot the proximal operator of $\beta_{\psi_n}$ generated by power functions and the KL divergence. In Figures~\ref{subfig:prox-4} and~\ref{subfig:prox-1.5}, we observe that the proximal operator transitions discontinuously from $0$ to a non-zero value before becoming linear. This behavior is associated with the case $\underaccent{\bar}{\theta}< \frac{1}{\rho}$. Note that achieving continuity is not possible here since the supremum of the second derivative is not attained for power functions with $1<p< 2$. For $p=2$ the second derivative is constant ($\underaccent{\bar}{\theta}=\bar\theta$), hence either the first or second case of Corollary~\ref{coro:prox-cases} stands. We depict the second case in Figure~\ref{subfig:prox-2}.
We illustrate all three possible cases using the KL generating function in Figures~\ref{subfig:prox-kl-c1}, \ref{subfig:prox-kl-c2}, and \ref{subfig:prox-kl-c3}.

\begin{figure}
    \centering
    \begin{subfigure}{0.3\textwidth}
        \centering
        \includegraphics[width=\linewidth]{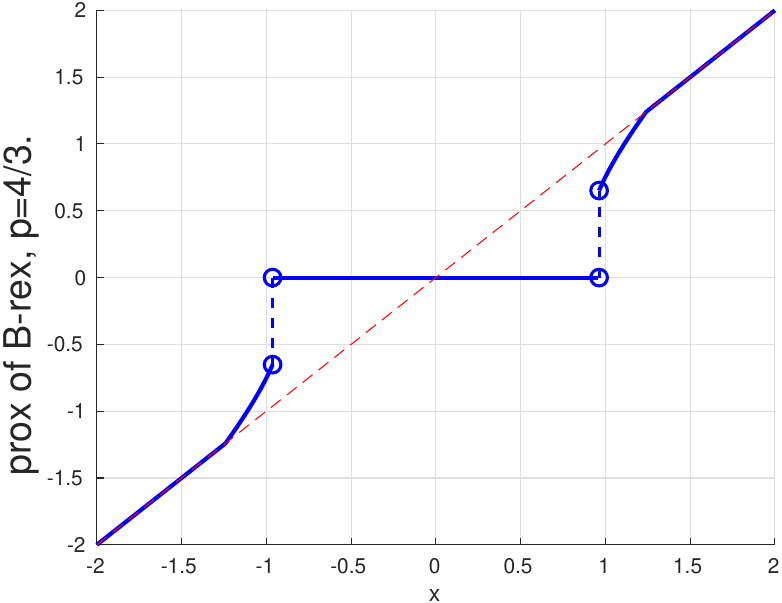}
       \caption{\footnotesize $\frac{4}{3}$-power function, $\underaccent{\bar}{\theta}< \frac{1}{\rho}$.}
        \label{subfig:prox-4}
    \end{subfigure}
       \hfill
    \begin{subfigure}{0.3\textwidth}
        \centering
        \includegraphics[width=\linewidth]{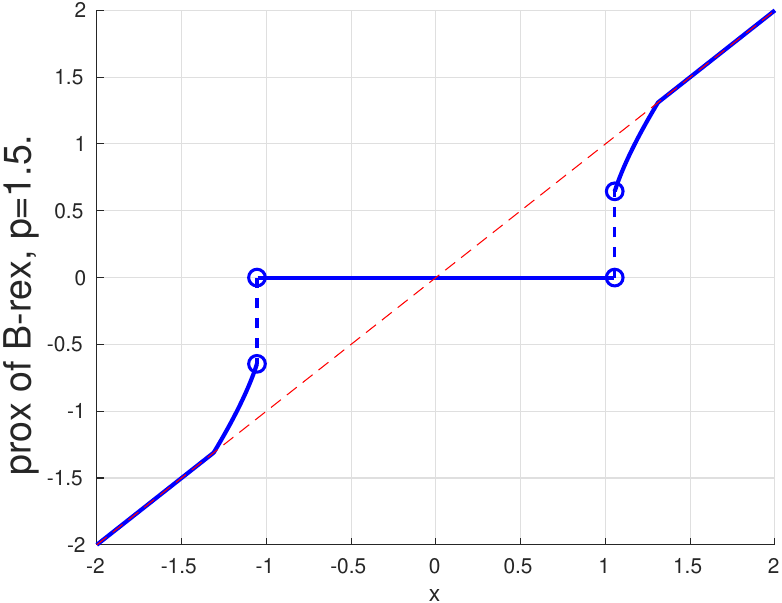}
        \caption{\footnotesize $\frac{3}{2}$-power function, $\underaccent{\bar}{\theta}< \frac{1}{\rho}$.}
       \label{subfig:prox-1.5}
    \end{subfigure}
    \hfill
    \begin{subfigure}{0.3\textwidth}
        \centering
        \includegraphics[width=\linewidth]{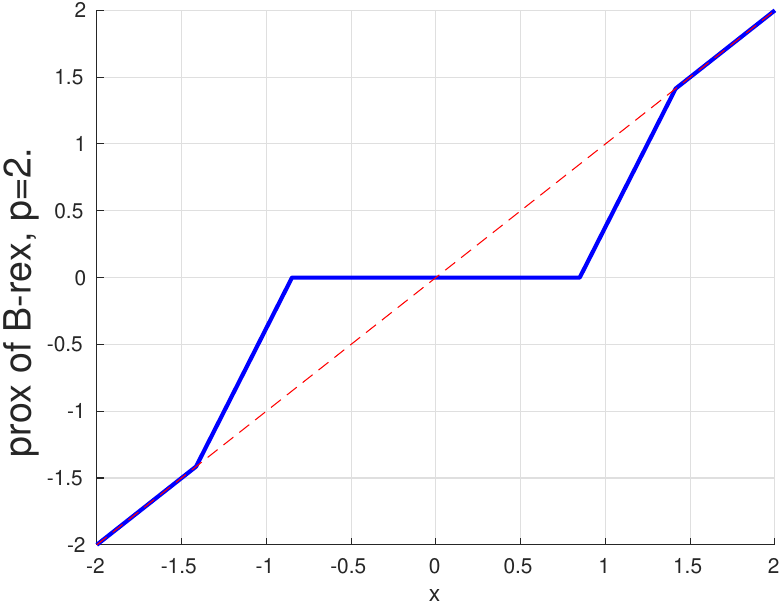}
        \caption{\footnotesize $2$-power function, $\bar\theta < \frac{1}{\rho} $.}
        \label{subfig:prox-2}
    \end{subfigure}
    
       \begin{subfigure}{0.3\textwidth}
        \centering
        \includegraphics[width=\linewidth]{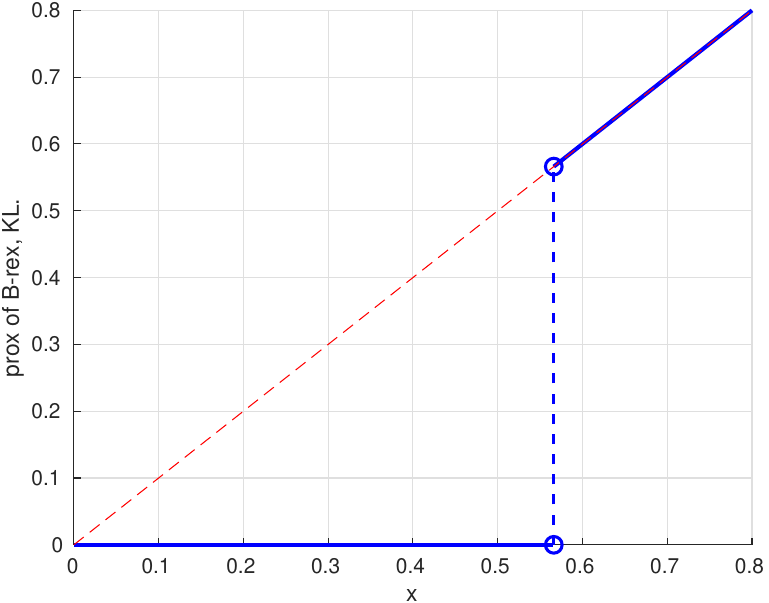}
       \caption{\footnotesize KL, $\underaccent{\bar}{\theta} > \frac{1}{\rho}$.}
        \label{subfig:prox-kl-c1}
    \end{subfigure}
       \hfill
    \begin{subfigure}{0.3\textwidth}
        \centering
        \includegraphics[width=\linewidth]{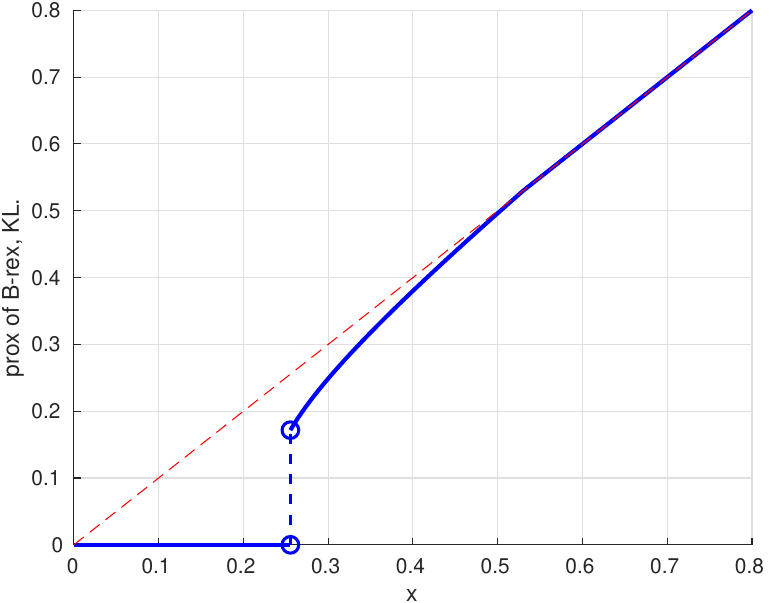}
       \caption{\footnotesize KL, $\underaccent{\bar}{\theta} < \frac{1}{\rho} < \bar\theta$.}
       \label{subfig:prox-kl-c2}
    \end{subfigure}
    \hfill
    \begin{subfigure}{0.3\textwidth}
        \centering
        \includegraphics[width=\linewidth]{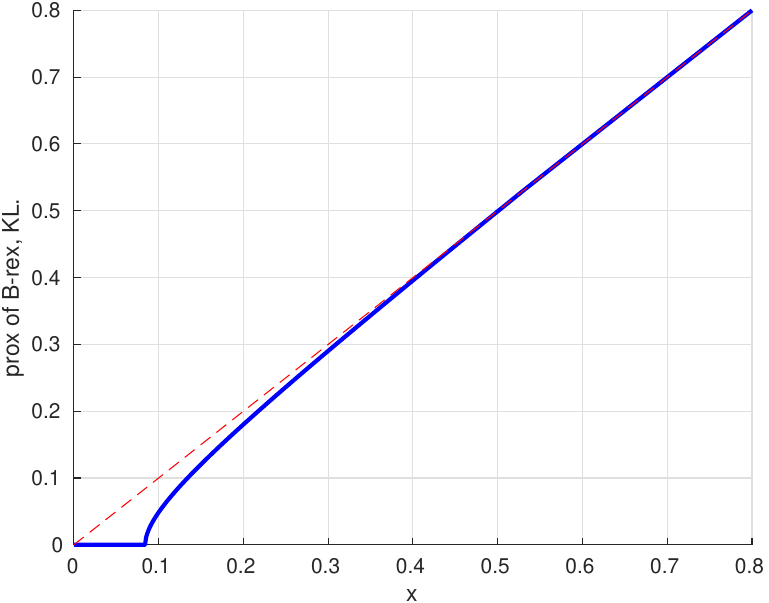}
        \caption{\footnotesize KL, $\bar\theta < \frac{1}{\rho}$.}
       \label{subfig:prox-kl-c3}
    \end{subfigure}
    \caption{Proximal operator of $\beta_{\rho\psi}$ with $\psi$ being a $p$-power function with $p \in \{4/3,3/2,2\}$ and the KL divergence.}
    \label{fig:plot-prox-functions}
\end{figure}

\section{Numerical Illustrations}\label{sec:application}

In this section, we evaluate the performance of the proposed relaxations on a diverse range of applications
where the goal is to estimate a sparse signal $\x\in\Cc^N$ from noisy linear measurements $\y \approx \A\x$. We start our discussion with a standard validation framework where the noise is Gaussian and a classical least-squares fidelity $F_\y$ is used. This model is frequently encountered in applications in compressive sensing~\cite{Bruckstein2009FromSS} and inverse problems in signal/image processing~\cite{10.5555/1895005,Soussen2011FromBD}.

We then focus on non-quadratic data terms and consider a logistic regression data term which, combined with a sparsity-promoting regularizer, aims to predict a sparse model for binary classification \cite{Lee2006EfficientLR}, with numerous applications in document classification \cite{Brzezinski99} and computer vision \cite{Friedman08}. Finally, we consider the Kullback-Leibler divergence data term which is often employed in the framework of inverse problems in biological and astronomical imaging to  describe the presence of Poisson noise, see, e.g.~\cite{Lingenfelter2009,Harmany2012}. The chosen data terms (see Table~\ref{tab:data_fidelty}) have many real-world applications. 

For the following numerical tests, we define the generating functions $\Psi=\left\{ \psi_n\right\}$ as  $\psi_n=\gamma_n \psi$ for $n \in [N]$, where $\psi$ does not depend on $n$. As such, the parameters $\gamma_n>0$ tune the curvature of the \BR{} penalty in order to satisfy the conditions of Theorem~\ref{th:exact_relax} (more precisely here~\eqref{eq:sup_inf_cond}). We consider the generating functions listed in Table~\ref{tab:B-rex} and for which the required quantities are provided in Table~\ref{tab:exact-conditions}. Then, by simple inversion to isolate the parameters $\gamma_n$, we obtain the conditions for exact relaxation reported in Table~\ref{tab:cond-gamma}.  
It suffices to ensure that $\gamma_n$ is greater than the specific bound denoted by $\hat{\gamma}_n$ in Table~\ref{tab:cond-gamma}, which depends on the data, to guarantee the validity of \eqref{eq:sup_inf_cond}. We  then define the vector $\bm{\gamma}_\mathrm{thr} = \left( \hat{\gamma}_n \right)_{n \in [N]}$.
\medskip

In the following experiments we illustrate the theoretical properties of \BR{} showed so far. 
We will focus, in particular, on the potential of the \BR{} relaxations to eliminate local minimizers of the initial Problem \eqref{eq:problem_setting} which makes the use of  off-the-shelf algorithms (e.g., proximal gradient) both possible and more effective. The development of tailored optimization methods exploiting explicitly the known properties of \BR{} (see, e.g., Proposition~\ref{Local_minimizers_of_J0_preserved_by_JPsi}), as well as exhaustive numerical comparisons with other approaches addressing \eqref{eq:problem_setting} will be addressed in future work.


\begin{table}
    \centering
    \begin{tabular}{|c|c|c|}
    \hline 
     \multicolumn{2}{|c|}{\cellcolor[HTML]{EFEFEF}Problem} & \multicolumn{1}{c|}{\cellcolor[HTML]{EFEFEF}Bregman generating function $\psi$} \\ 
    \hline
    \cellcolor[HTML]{FFFFC7} $F_\y$&\cellcolor[HTML]{FFFFC7} $\lambda_2$ &\cellcolor[HTML]{FFFFC7} $p$-power function, $\; p \in (1,2]$\\ \hline
    LS& $0$ &$\gamma_n > \left( p \lambda_0 \right)^{\frac{2-p}{2}} \|\mathbf{a}_n\|_2^p:=\hat{\gamma}_n$ \\ \hline
     LR & $\lambda_2 > 0$ & $\gamma_n > \left( p \lambda_0 \right)^{\frac{2-p}{2}} \left( \frac{1}{4} \|\mathbf{a}_n\|_2^2+\lambda_2 \right)^{\frac{p}{2}}:=\hat{\gamma}_n $\\ \hline
     KL& $0$ & $\gamma_n >  (p\lambda_0)^{\frac{2-p}{2}} \left(\frac{1}{b^2} \sum_{m=1}^M a_{mn}^2 y_m \right)^{\frac{p}{2}}:=\hat{\gamma}_n$ \\ \hline
      \cellcolor[HTML]{FFFFC7}&\cellcolor[HTML]{FFFFC7} &\cellcolor[HTML]{FFFFC7}
    Shannon entropy \\ \hline
    KL&$0$ & $\gamma_n >  \left(\frac{\lambda_0}{b^2} \sum_{m=1}^M a_{mn}^2 y_m \right)^{\frac{1}{2}}:=\hat{\gamma}_n$ \\ \hline
      \cellcolor[HTML]{FFFFC7}&\cellcolor[HTML]{FFFFC7} & \cellcolor[HTML]{FFFFC7}
    KL divergence \\ \hline
    KL &$0$& $\gamma_n W^2(-b\e^{-\kappa})>\sum_{m=1}^M a_{mn}^2 y_m$  \\ \hline
    \end{tabular}
    \caption{Conditions to have exact relaxation properties for different data terms $F_\y$ and Bregman generating functions $\psi$, see \eqref{eq:sup_inf_cond}. $W(\cdot)$ denotes the Lambert function and $\kappa=\frac{\lambda_0}{y\gamma_n}+log(b)+1$.}
    \label{tab:cond-gamma}
\end{table}

\subsection{One-dimensional examples}

\begin{figure}
\centering
\begin{subfigure}[t]{0.3\textwidth}
    \centering
    \includegraphics[width=\textwidth]{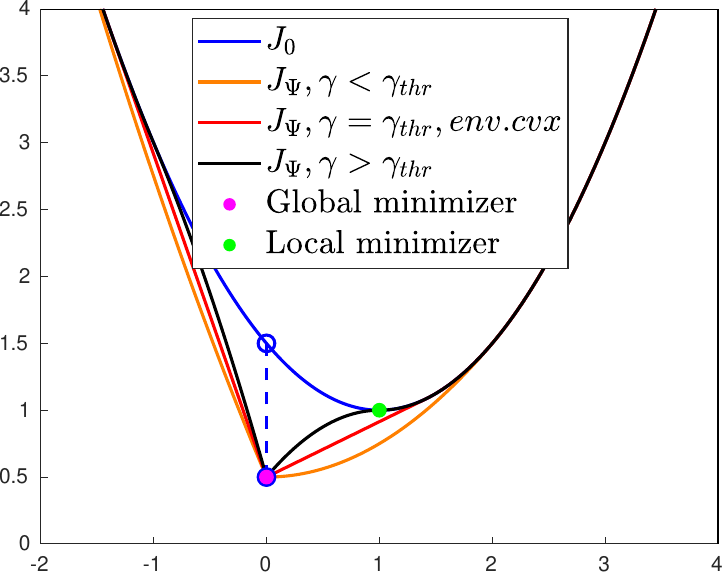}
    \caption{$F_\y$=LS, $\psi$ is the  2-power function, \change{$a=1$, $y=1$ and  $\lambda_2 =0$}.}
    \label{subfig:subfig_1D_LS}
\end{subfigure}
\hspace{1cm}
\begin{subfigure}[t]{0.3\textwidth}
    \centering
    \includegraphics[width=\textwidth]{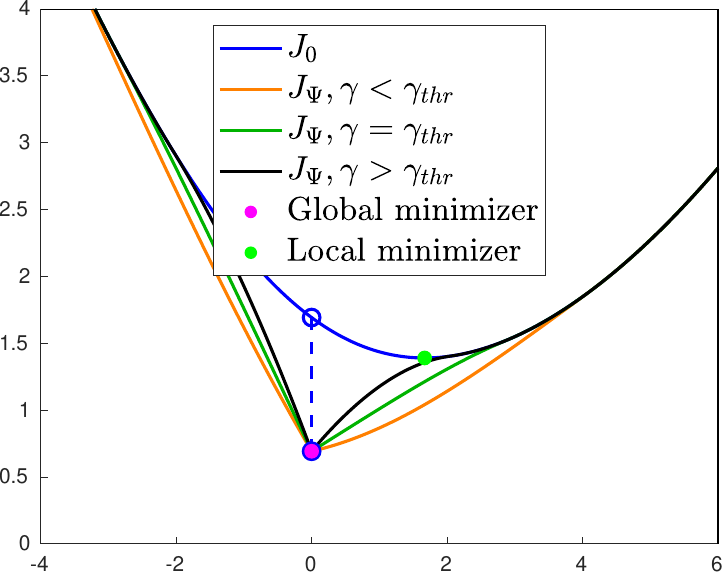}
    \caption{$F_\y$=LR, $\psi$ is the  2-power function, \change{$a=0.75$, $y=1$, and  $\lambda_2 =0.1$}.}
    \label{subfig:subfig_1D-LR}
\end{subfigure}

\begin{subfigure}[t]{0.3\textwidth}
    \centering
    \includegraphics[width=\textwidth]{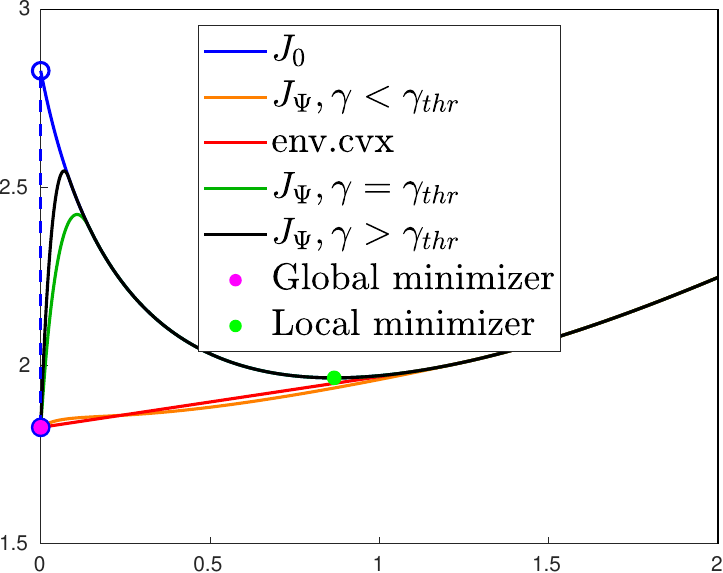}
    \caption{$F_\y=\mathrm{KL}$, $\psi$ is KL.}
    \label{subfig:subfig_1D_KL-kl}
\end{subfigure}
\hfill
\begin{subfigure}[t]{0.3\textwidth}
    \centering
    \includegraphics[width=\textwidth]{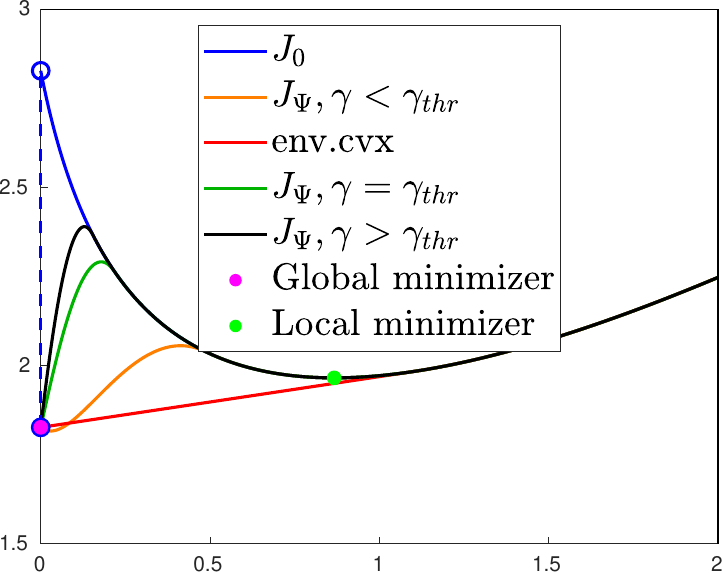}
    \caption{$F_\y=\mathrm{KL}$, $\psi$ is the  2-power function.}
    \label{subfig:subfig_1D_KL-Pow}
\end{subfigure}
\hfill
\begin{subfigure}[t]{0.3\textwidth}
    \centering
    \includegraphics[width=\textwidth]{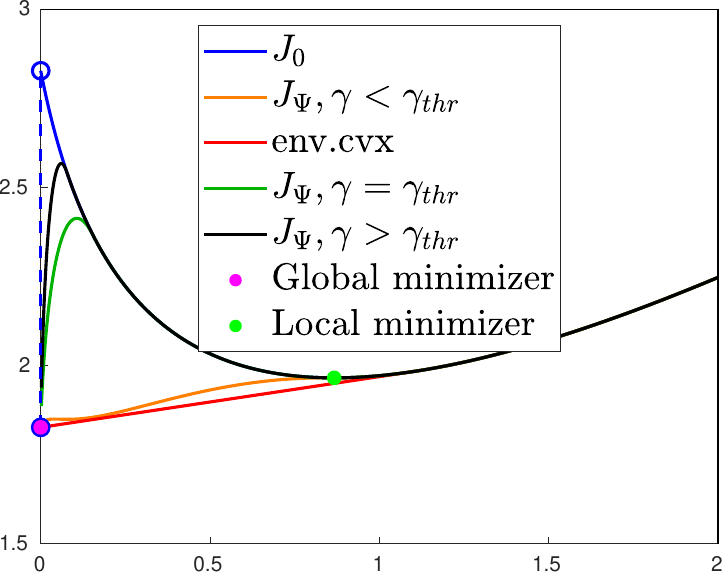}
    \caption{$F_\y=\mathrm{KL}$, $\psi$ is the Shannon-Entropy.}
    \label{subfig:subfig_1D_KL}
\end{subfigure}
 \caption{Relaxations $J_\Psi$ computed in terms of the condition on $\gamma$ as in Theorem~\ref{th:exact_relax} and Table~\ref{tab:cond-gamma} for different data-terms and generating functions. The red curves 
 correspond to the convex envelope of the original function $J_0$ obtained by choosing $\psi=f(a\cdot;y)$. \change{For KL, $(a, y, b, \lambda_2) = (0.75, 0.75, 0.1, 0)$}.}
 \label{fig:1D-data}
\end{figure}

Let $M=N=1$, $\lambda_0=1$, and consider the minimization problem defined by:
\begin{equation}\label{eq:1D-J-0}
J_0(x)=f(ax;y)+|x|_0+ \frac{\lambda_2}{2}x^2,
\end{equation}
where the function $f$ is chosen as specified in Table~\ref{tab:data_fidelty}. 
We aim to investigate the impact of replacing~\eqref{eq:1D-J-0} by
\begin{equation}\label{eq:1D-Jpsi}
J_\psi(x)=f(ax;y)+\beta_\psi(x)+\frac{\lambda_2}{2}x^2,  
\end{equation}
where $\beta_\psi$ is given in Table~\ref{tab:B-rex} for various generating functions. Proposition~\ref{prop2_section2} states that by selecting $\psi(x)=f(ax;y), a>0$, the relaxed function $J_\psi$ turns out to be the convex envelope of $J_0$. In Figure~\ref{fig:1D-data}, we present $J_0$ in the blue curve along with its two minimizers, the first (global) one is $\hat{x}=0$ (Lemma~\ref{lemma:0-is-strict-min}), and the second (local) one $\hat{x} \in \Cc$ solves $af'(a\hat{x};y)+\lambda_2\hat{x}=0$ (Proposition~\ref{prop:local_min_J0}), which gives $\hat{x}=\frac{y}{a}$ and $\hat{x}=\frac{y-b}{a}$ for LS and KL, respectively, while for LR the optimality condition has no explicit formula and requires the use of a root-finding method. The relaxations $J_\psi$ of $J_0$ are displayed using different colors.

We observe that the relaxation $J_\psi$ obtained by choosing $\psi(\cdot)=f(a\cdot;y)$ shown in red is indeed the convex envelope of the original functional $J_0$ as stated in Proposition~\ref{prop2_section2}. The global minimizer of $J_0$ thus corresponds to the unique minimizer of its convex envelope, which matches $J_0$ everywhere except for the intervals $(\alpha_n^-,0)$ and $(0,\alpha_n^+)$ where $J_\Psi$ is affine.  
On the other hand, as expected from Theorem~\ref{th:exact_relax} with~\eqref{eq:sup_inf_cond}, for $\gamma \geq \gamma_\mathrm{thr}$ the relaxed functional preserves the global minimizer of the initial functional and in some cases eliminates its local minimizer.
For $\gamma < \gamma_\mathrm{thr}$, we observe in some cases that the relaxed functional  adds new minimizers, and in particular moves the global one. Note indeed that in this regime $J_\Psi$ is not guaranteed anymore to be an exact relaxation of $J_0$. Finally, it is important  to recall that here $\gamma_\mathrm{thr}$ is computed using the coarser condition~\eqref{eq:sup_inf_cond}.
This explain why in Figures~\ref{subfig:subfig_1D_KL-kl}, \ref{subfig:subfig_1D_KL-Pow} and \ref{subfig:subfig_1D_KL}, $\gamma \geq \gamma_\mathrm{thr}$ leads to  an exact relaxation which does not correspond to the convex envelope. In this simple 1D case, taking $\psi(\cdot)=f(a\cdot;y)$ as stated in Proposition~\ref{prop2_section2} allows to satisfy condition \eqref{eq:general-exact-condition} directly, which leads to the convex envelope.

\begin{figure}
    \centering
    \begin{subfigure}{0.35\textwidth}
        \centering
        \includegraphics[width=\linewidth]{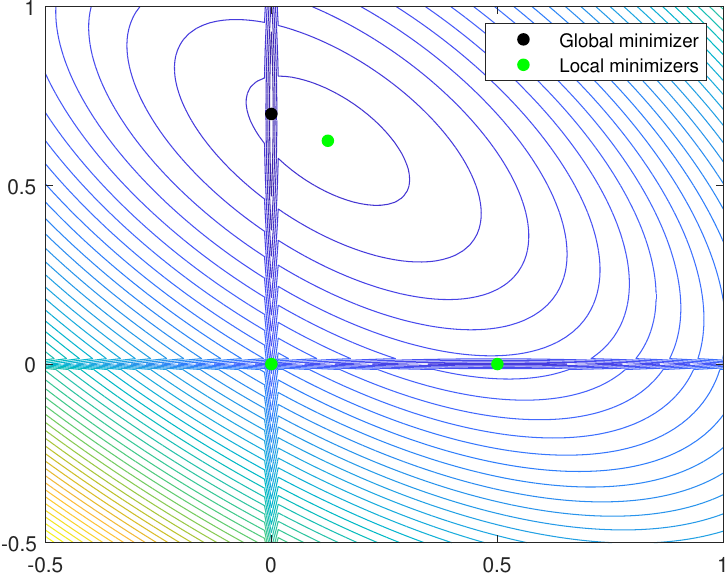}\vspace{-0.2cm}
        \caption{$J_0$.}
         \label{subfig:LS-J0}
    \end{subfigure}
    \begin{subfigure}{0.35\textwidth}
        \centering
        \includegraphics[width=\linewidth]{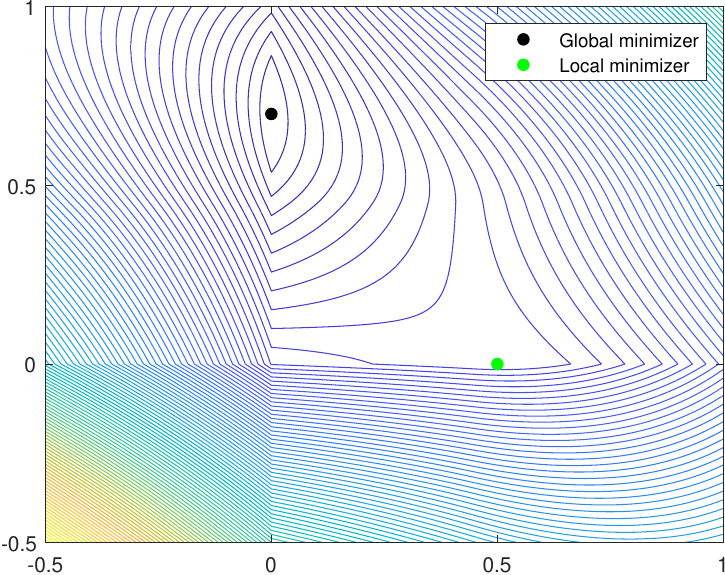}\vspace{-0.2cm}
        \caption{$J_\Psi$, $\bm{\gamma}=\bm{\gamma_{thr}}$.}
        \label{subfig:LS-gamma-thr}
    \end{subfigure}\vspace{-0.2cm}
    \begin{subfigure}{0.35\textwidth}
        \centering
        \includegraphics[width=\linewidth]{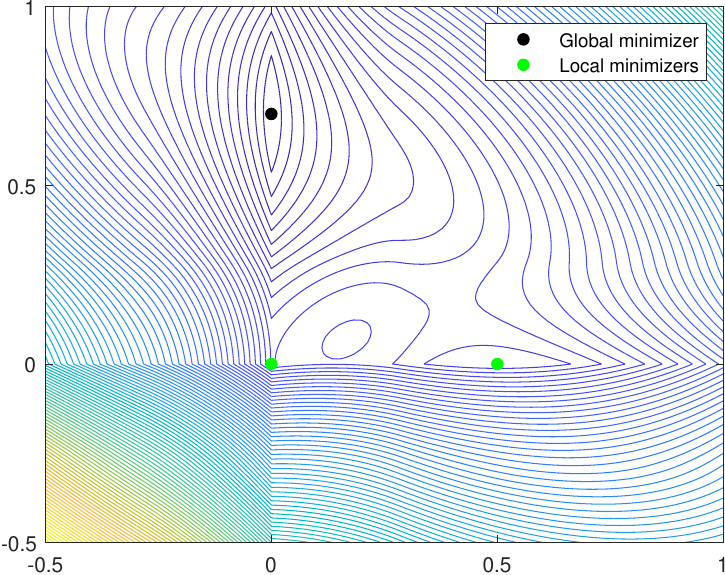}\vspace{-0.2cm}
        \caption{$J_\Psi$, $\bm{\gamma} > \bm{\gamma_{thr}}$.}
         \label{subfig:LS-more-gamma-thr}
    \end{subfigure}
    \begin{subfigure}{0.35\textwidth}
        \centering
        \includegraphics[width=\linewidth]{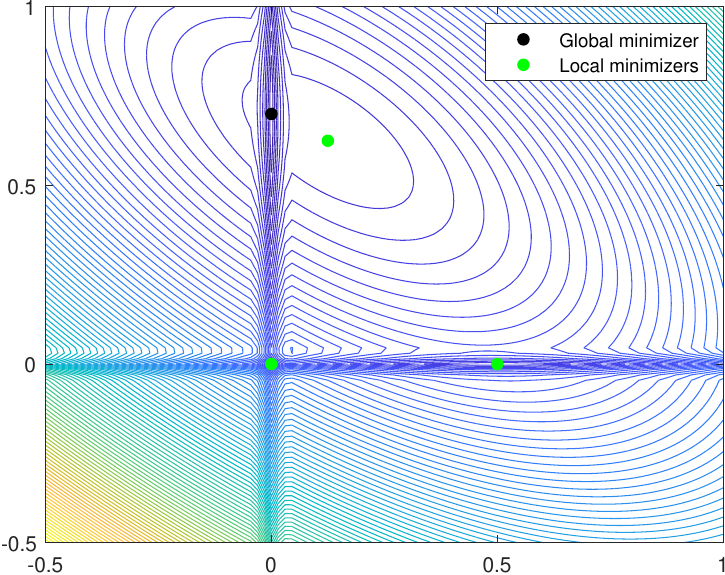}\vspace{-0.2cm}
        \caption{$J_\Psi$, $\bm{\gamma} \gg \bm{\gamma_{thr}}$.}
         \label{subfig:LS-large-gamma-thr}
    \end{subfigure}

    \caption{$F_\y$=LS with  $\A=[3,1;1,3]$, $\y=[1;2]$ and $\lambda_0=0.5$. Level lines of $J_0$  and of $J_\Psi$ for different values of $\bm{\gamma}$. The generating function $\psi$ is the $2$-power function.}
    \label{fig:Level-lines-LS}
\end{figure}

\begin{figure}
    \centering
    \begin{subfigure}{0.35\textwidth}
        \centering
        \includegraphics[width=\linewidth]{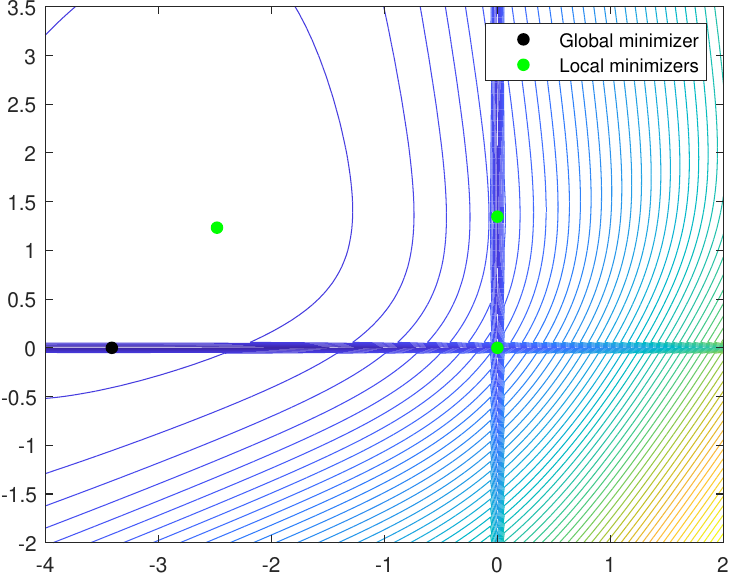}\vspace{-0.2cm}
        \caption{$J_0$.}
         \label{subfig:LR-J0}
    \end{subfigure}
    \begin{subfigure}{0.35\textwidth}
        \centering
        \includegraphics[width=\linewidth]{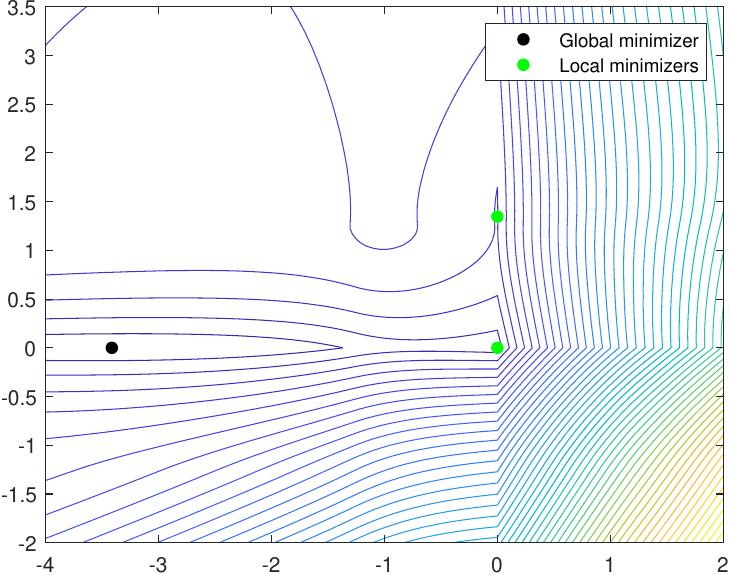}\vspace{-0.2cm}
        \caption{$J_\Psi$, $\bm{\gamma} =\bm{\gamma_{thr}}$.}
         \label{subfig:LR-gamma-thr}
    \end{subfigure}\vspace{-0.2cm}
    \begin{subfigure}{0.35\textwidth}
        \centering
        \includegraphics[width=\linewidth]{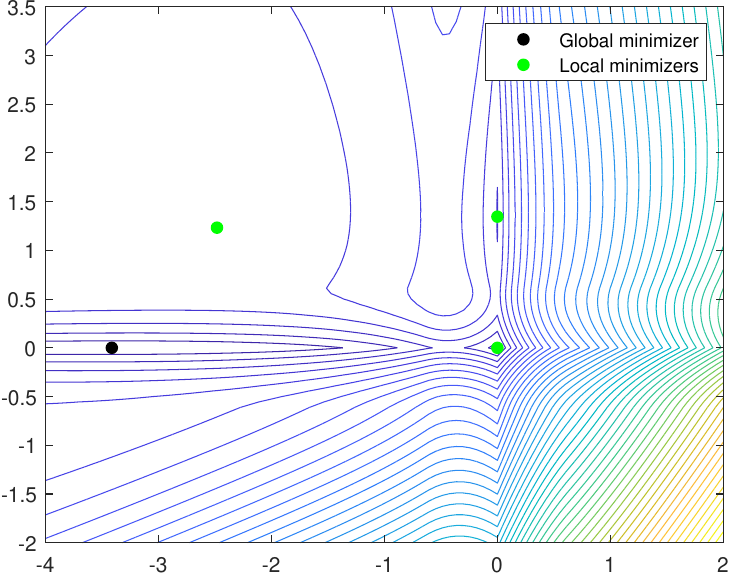}\vspace{-0.2cm}
        \caption{$J_\Psi$, $\bm{\gamma} >  \bm{\gamma_{thr}}$.}
         \label{subfig:LR-more-gamma-thr}
    \end{subfigure}
    \begin{subfigure}{0.35\textwidth}
        \centering
        \includegraphics[width=\linewidth]{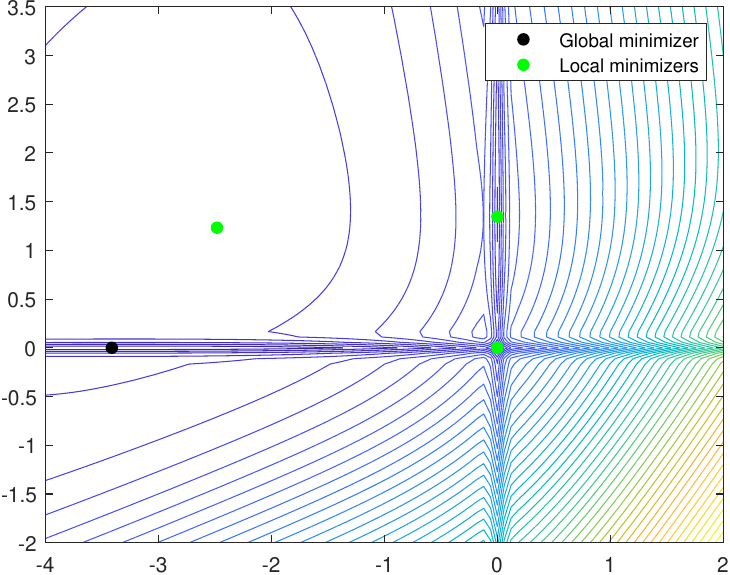}\vspace{-0.2cm}
        \caption{$J_{\Psi}, \bm{\gamma} \gg \bm{\gamma_{thr}}$.}
         \label{subfig:LR-large-gamma-thr}
    \end{subfigure}
    \caption{$F_\y$=LR with $\A=[-1,2;2,0.2]$, $\y=[1;0]$, $\lambda_0=1$, $\lambda_2=0.1$. Level lines of $J_0$ and of $J_\Psi$ for different values of $\bm{\gamma}$. The generating function $\psi$ is the $2$-power function.}
    \label{fig:Level-lines-LR}
\end{figure}

    \begin{figure}[ht!]
    \centering
    \begin{subfigure}{0.3\textwidth}
        \centering
        \includegraphics[width=\linewidth]{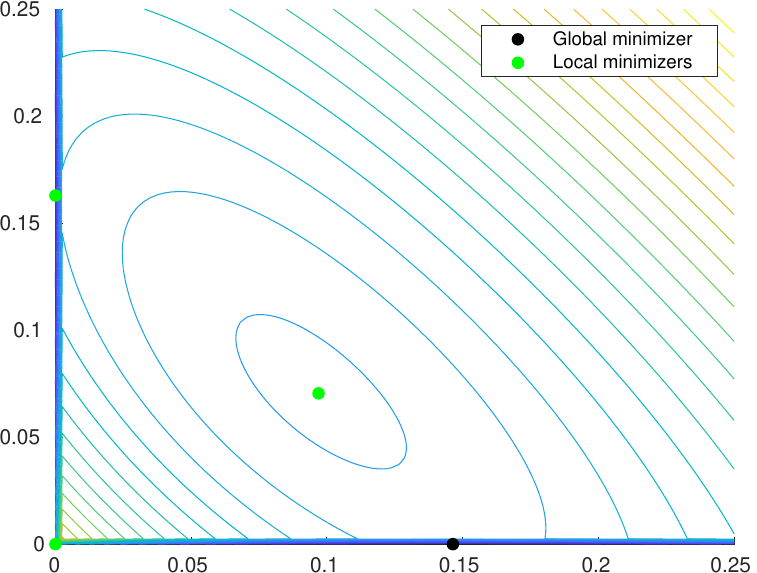}
        \caption{$J_0$.}
        \label{subfig:J0}
    \end{subfigure}
    
    \begin{subfigure}{0.3\textwidth}
        \centering
        \includegraphics[width=\linewidth]{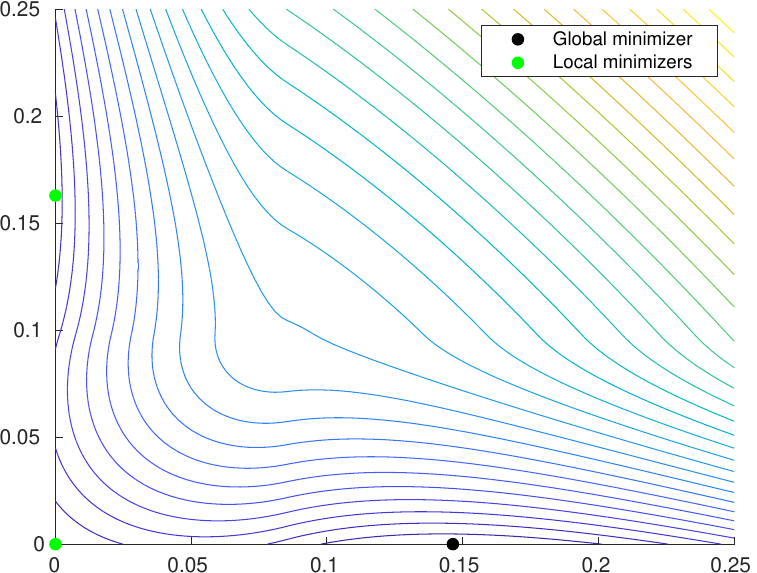}
        \caption{ $p=2$, $\bm{\gamma} = \bm{\gamma_{thr}}$.}
        \label{subfig:pow-thr}
    \end{subfigure}       
      \begin{subfigure}{0.3\textwidth}
        \centering
        \includegraphics[width=\linewidth]{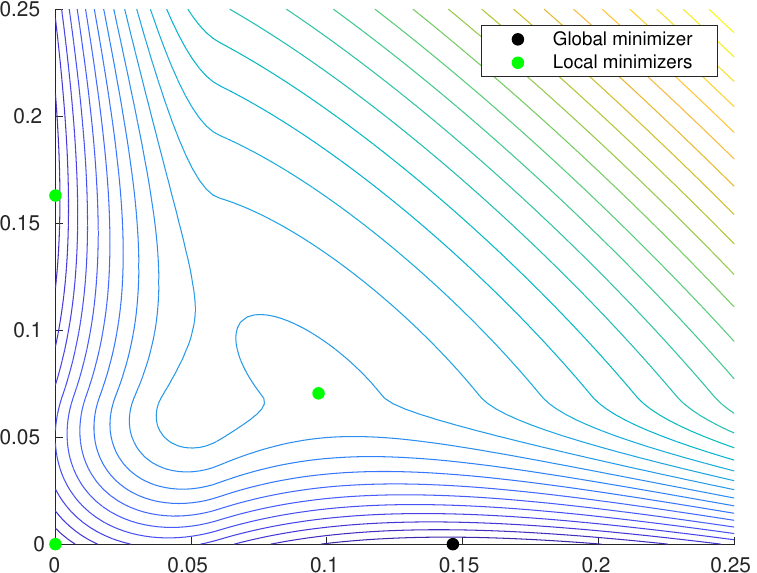}
        \caption{ $p=2$, $\bm{\gamma}>\bm{\gamma_{thr}}$.}
        \label{subfig:pow_2thr}
    \end{subfigure}    
    \begin{subfigure}{0.3\textwidth}
        \centering
        \includegraphics[width=\linewidth]{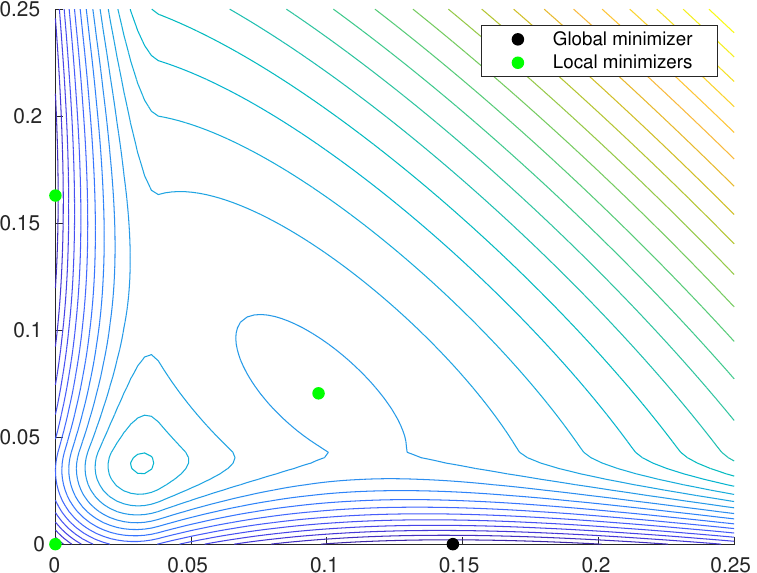}
        \caption{ $p=2$, $\bm{\gamma} \gg \bm{\gamma_{thr}}$.}
        \label{subfig:pow_4thr}
    \end{subfigure}
    \begin{subfigure}{0.3\textwidth}
        \centering
        \includegraphics[width=\linewidth]{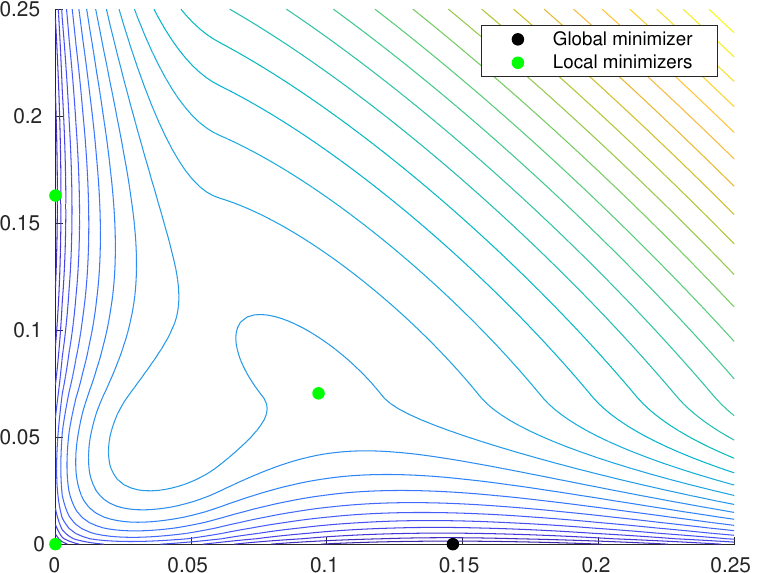}
        \caption{Entropy, $\bm{\gamma} =\bm{\gamma_{thr}}$.}\label{subfig:shan-thr}
    \end{subfigure}    
    \begin{subfigure}{0.3\textwidth}
        \centering
        \includegraphics[width=\linewidth]{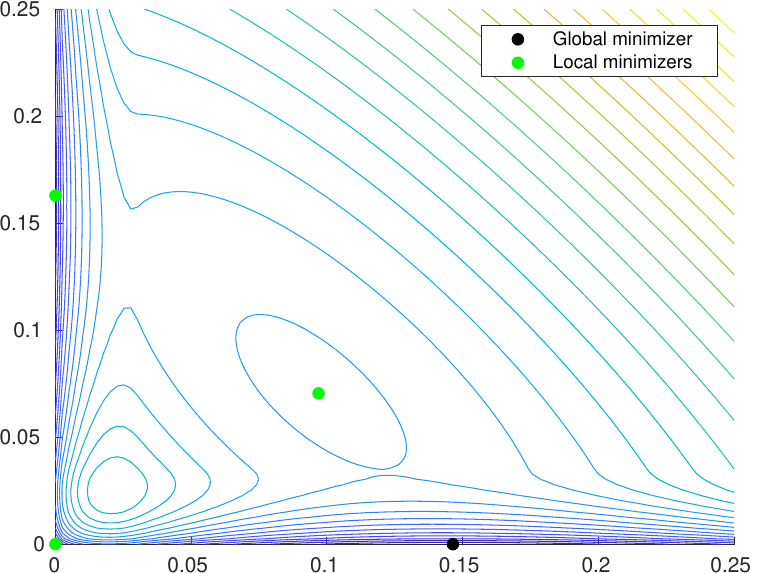}
        \caption{Entropy, $\bm{\gamma} > \bm{\gamma_{thr}}$.}
        \label{subfig:shan-2thr}
    \end{subfigure}    
    \begin{subfigure}{0.3\textwidth}
        \centering
        \includegraphics[width=\linewidth]{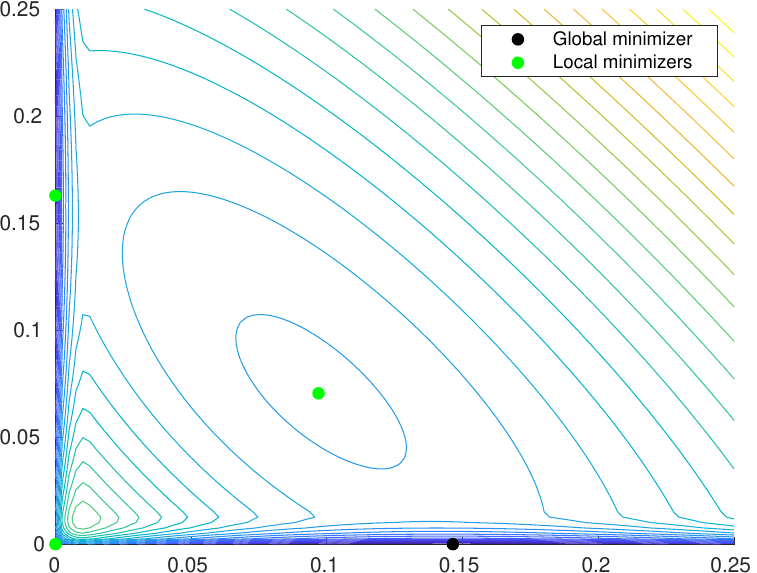}
        \caption{Entropy, $\bm{\gamma} \gg \bm{\gamma_{thr}}$.}\label{subfig:shan-4thr}
    \end{subfigure}    
    \begin{subfigure}{0.3\textwidth}
        \centering
        \includegraphics[width=\linewidth]{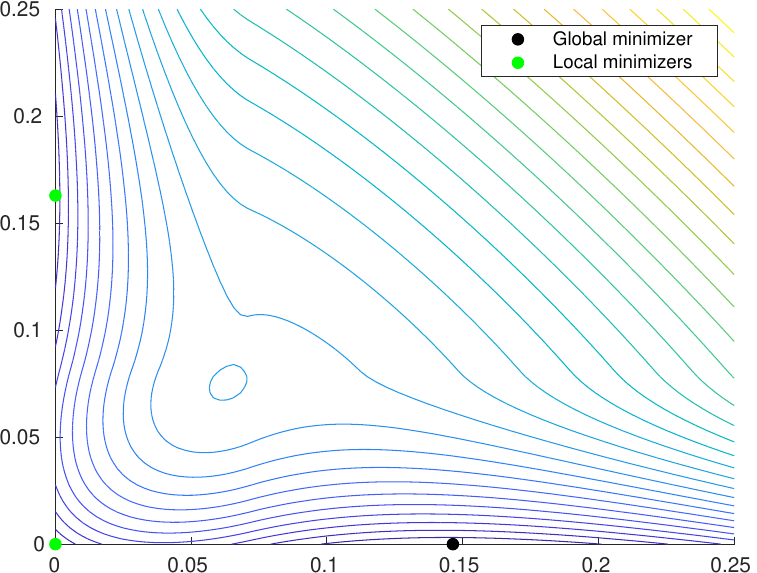}
        \caption{KL, $\bm{\gamma}=\bm{\gamma_{thr}}$.}
        \label{subfig:kl-thr}
    \end{subfigure}    
    \begin{subfigure}{0.3\textwidth}
        \centering
        \includegraphics[width=\linewidth]{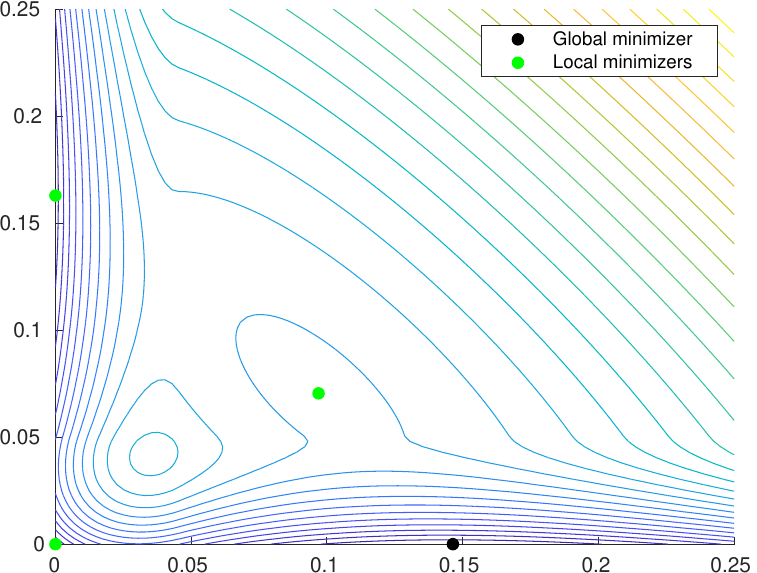}
        \caption{KL, $\bm{\gamma} > \bm{\gamma_{thr}}$.}
        \label{subfig:kl_2thr}
    \end{subfigure}    
    \begin{subfigure}{0.3\textwidth}
        \centering
        \includegraphics[width=\linewidth]{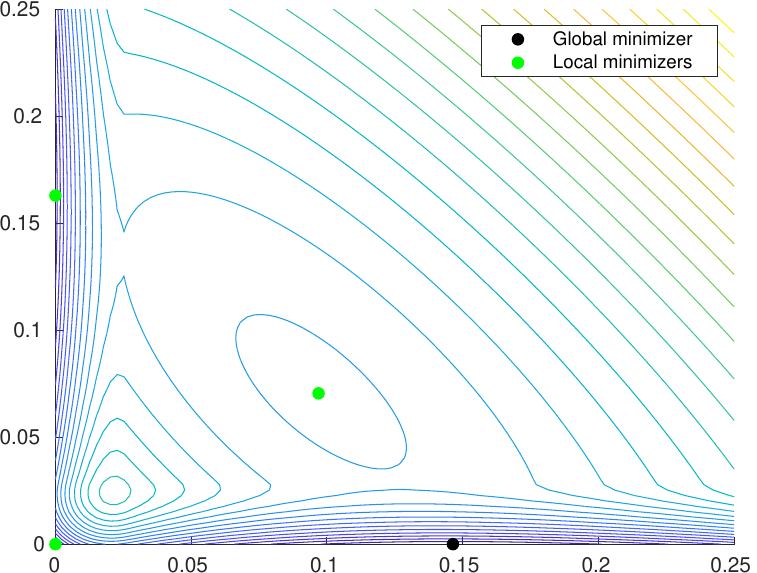}
        \caption{KL, $\bm{\gamma} \gg \bm{\gamma_{thr}}$.}
        \label{subfig:kl_4thr}
    \end{subfigure}
    \caption{$F_\y$=KL with $\A=[0.45,0.8;0.85,0.25]$, $\y=[0.2;0.2]$, $\lambda_0=0.06F_\y(\mathbf{0})$ and $b=0.1$. Level lines of $J_0$ and $J_\Psi$ using the 2-power function, the Shannon-Entropy and the Kullback-Leibler divergence as the generating function $\psi$ for various values of $\bm{\gamma}$.}
    \label{fig:KL-2D}
\end{figure} 

\subsection{Two-dimensional examples}\label{sec:2D-exampls}
We now consider the case $N=M=2$. In these 2D examples, $J_0$ has four minimizers. The first one is $\mathbf{0} \in \R^2$, as stated by Lemma~\ref{lemma:0-is-strict-min}, and the others are obtained through Corollary~\ref{coro:criter_locmin_j0} for $\hat{\sigma} \in \left\lbrace\{1\}, \{2\}, \{1,2\}\right\rbrace$. Figures~\ref{fig:Level-lines-LS} and~\ref{fig:Level-lines-LR} show the iso-levels of $J_0$ with its minimizers as well as those of $J_\Psi$ (using $2$-power generating function) for different values of $\bm{\gamma}$ and for LS and LR data terms. In both examples, the global minimizer of $J_0$ is also the global minimizer of the exact relaxations $J_\Psi$, as ensured by Theorem~\ref{th:exact_relax}. However, we observe that with $\bm{\gamma}=\bm{\gamma}_\mathrm{thr}$, $J_\Psi$ has fewer (local) minimizers in comparison with $J_0$, (Figures~\ref{subfig:LS-gamma-thr} and~\ref{subfig:LR-gamma-thr}). Note that as $\bm{\gamma}$ increases beyond $\bm{\gamma}_\mathrm{thr}$, $J_\Psi$ removes less local minimizers. Finally, for larger values of $\bm{\gamma}$, $J_\Psi$ tends to the original functional, as it can be easily noticed by comparing the iso-levels of $J_\Psi$ with those of $J_0$.
In Figure~\ref{fig:KL-2D}, we plot the iso-levels of $J_0$ with a KL data term and its exact relaxation $J_\Psi$ using 2-power functions, Shannon entropy, and KL as generating functions. The figures show that the global minimizer of $J_0$ is also a global minimizer of $J_\Psi$ in all cases. Moreover, we observe that choosing $2$-power function as generating function leads to an exact relaxation with a better landscape (wider basins and less local minimizers). {Yet, one may be able to find a better relaxation by exploiting directly~\eqref{eq:general-exact-condition} rather than~\eqref{eq:sup_inf_cond}~\cite{essafri2024l0}.}

\subsubsection*{Proximal gradient algorithm performance}\label{sec:numeric_algo}

We now apply the PGA specified in Section \ref{sec:algo} to the previous 2D examples. We use the same model and data $\A,\y$ as in Section~\ref{sec:2D-exampls} and the same parameters  $\lambda_0,\lambda_2$. For each application, PGA on both $J_0$ and the relaxed functionals $J_\Psi$ is initialized with the same point and uses the same fixed step size $0<\rho< 1/L$ where  the Lipschitz constants are given by $L=\|\A\|^2, \; L=(1/4)\|\A\|^2+\lambda_2 \; \text{and} \; L= b^{-2}\|\A\|^2 $ for LS, LR, and KL problems, respectively. We used as a stopping criterion the norm of the relative error (i.e., $\|\x^{k+1}-\x^{k}\|< \varepsilon$) and a maximum number of iterations, so the algorithm stops when either of these criteria is met first. The results are shown in Figure~\ref{fig:iter-BPGA} for different choices of the data terms $F_\y$ and  generating function $\psi$ with always $\bm{\gamma} = \bm{\gamma}_\mathrm{thr}$.  For the LS and LR data terms the generating functions are the $p$-power functions for different values of $p\in\left\{4/3,3/2, 2\right\}$. We observe that for $p=2$ we achieve the global minimizer for both LS and LR, whereas convergence to a local minimizer is observed for $p\in\{4/3,3/2\}$. For the KL case, we consider three different choices of generating functions corresponding to the $2$-power function, the Shannon Entropy and the KL divergence. We observe that the global minimizer is attained when the $2$-power function or KL divergence are used to generate the \BR{} penalty, while convergence to local minimizers is observed when \BR{} is generated with the Shannon entropy. 


\begin{figure}
    \centering
    \begin{subfigure}{0.3\textwidth}
        \centering
        \includegraphics[width=\linewidth]{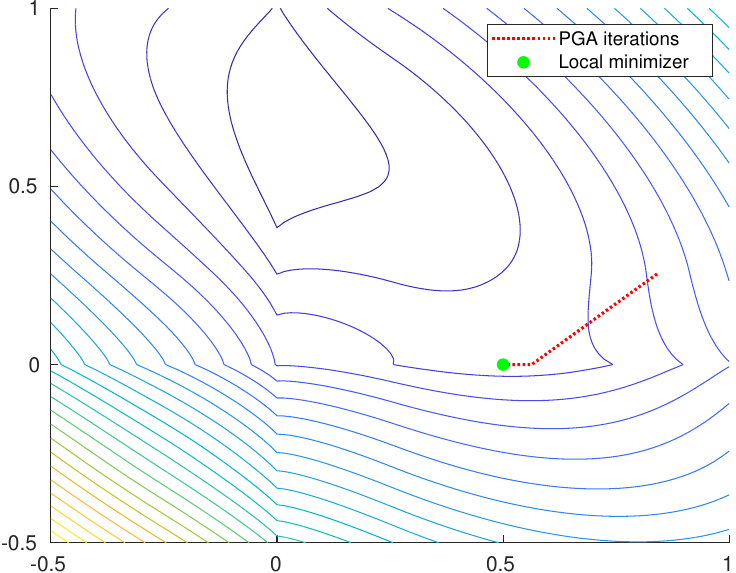}
        \caption{\footnotesize $F_\y$=LS, $\frac{4}{3}$-power function.}
        \label{subfig:pc-ls-1.75}
    \end{subfigure}
       \hfill
    \begin{subfigure}{0.3\textwidth}
        \centering
        \includegraphics[width=\linewidth]{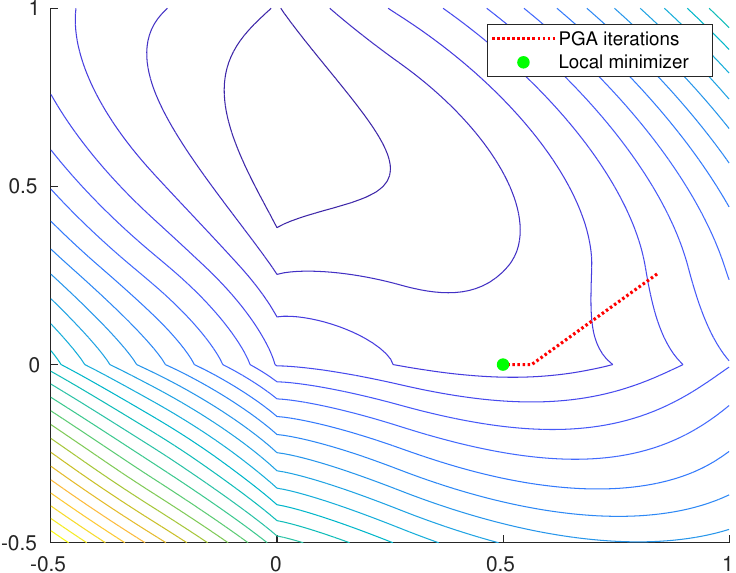}
        \caption{\footnotesize $F_\y$=LS, $\frac{3}{2}$-power function.}
        \label{subfig:pc-ls-1.5}
    \end{subfigure}
    \hfill
    \begin{subfigure}{0.3\textwidth}
        \centering
        \includegraphics[width=\linewidth]{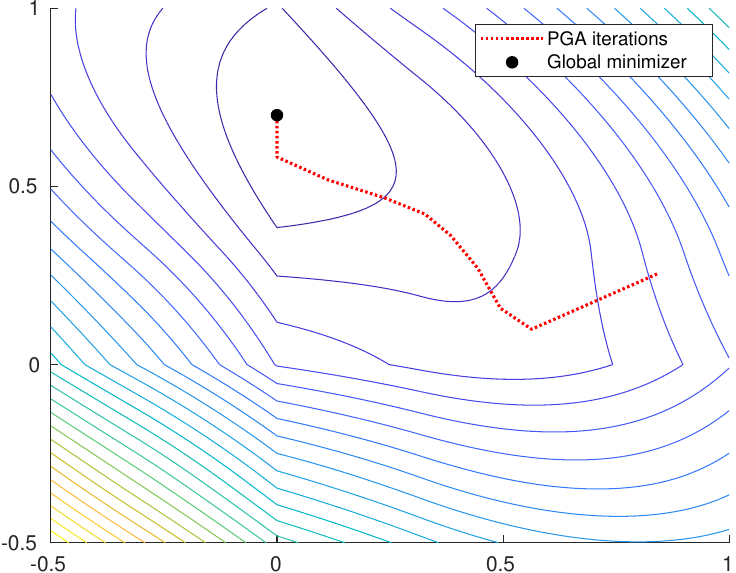}
        \caption{\footnotesize $F_\y$=LS, $2$-power function.}
        \label{subfig:pc-ls-2}
    \end{subfigure}
    \hfill
       \begin{subfigure}{0.3\textwidth}
        \centering
        \includegraphics[width=\linewidth]{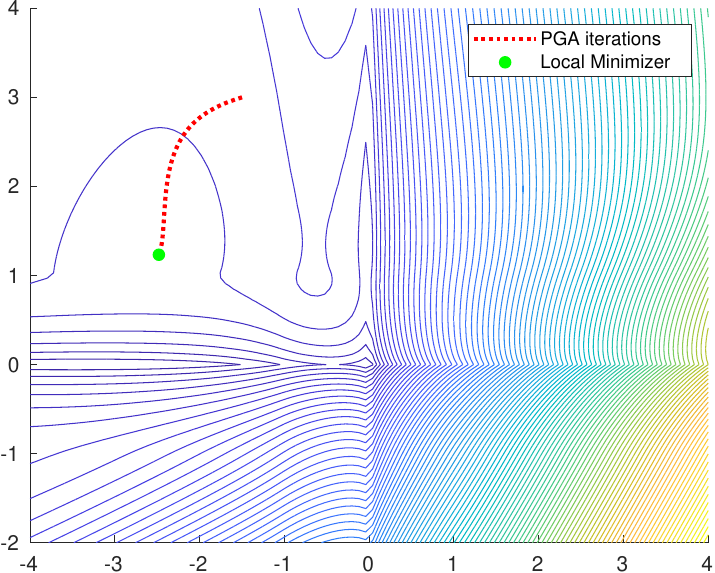}
        \caption{\footnotesize $F_\y$=LR, $\frac{4}{3}$-power function.}
        \label{subfig:pc-lr-1.75}
    \end{subfigure}
       \hfill
     \begin{subfigure}{0.3\textwidth}
        \centering
        \includegraphics[width=\linewidth]{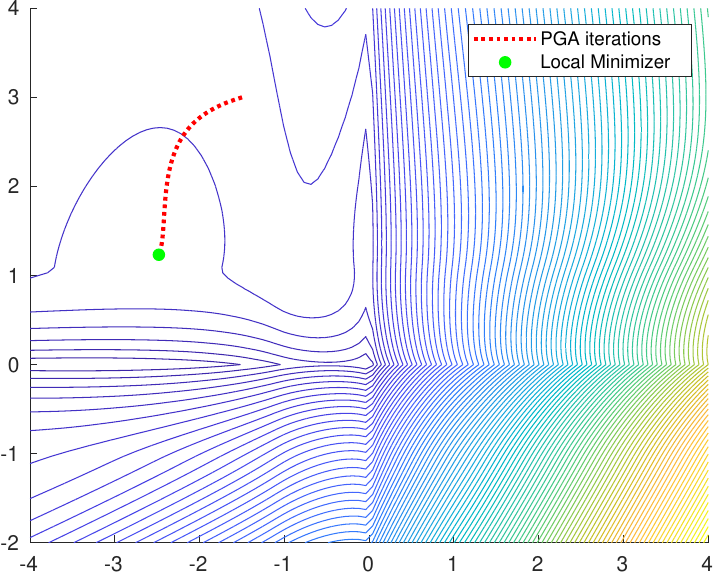}
        \caption{\footnotesize $F_\y$=LR, $\frac{3}{2}$-power function.}
        \label{subfig:pc-lr-1.5}
    \end{subfigure}
    \hfill
    \begin{subfigure}{0.3\textwidth}
        \centering
        \includegraphics[width=\linewidth]{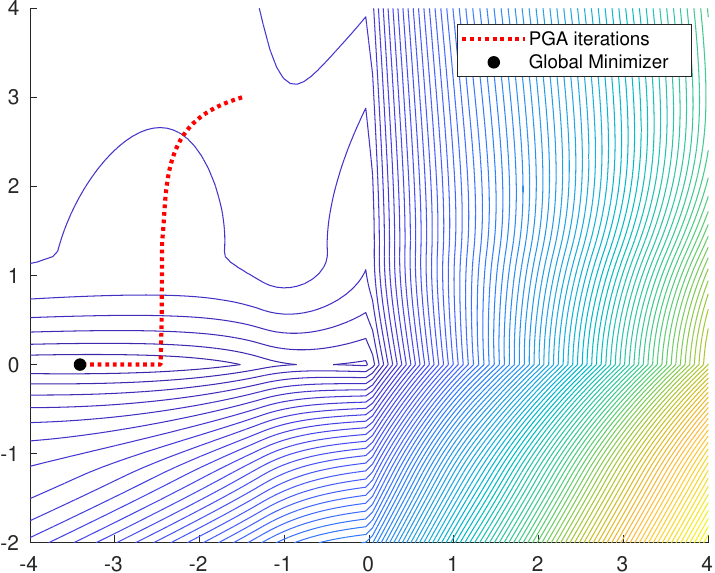}
        \caption{\footnotesize $F_\y$=LR, $2$-power function.}
        \label{subfig:pc-kl-energy}
    \end{subfigure}
    \medskip
        \begin{subfigure}{0.3\textwidth}
        \centering
        \includegraphics[width=\linewidth]{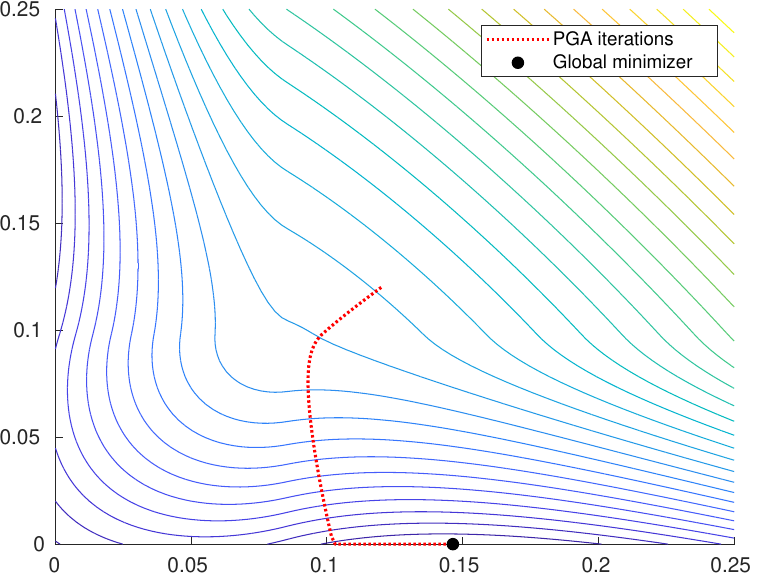}
        \caption{\footnotesize $F_\y$=KL, $2$-power function.}
        \label{subfig:pc-kl-energy}
    \end{subfigure}
       \hfill
      \begin{subfigure}{0.3\textwidth}
        \centering
        \includegraphics[width=\linewidth]{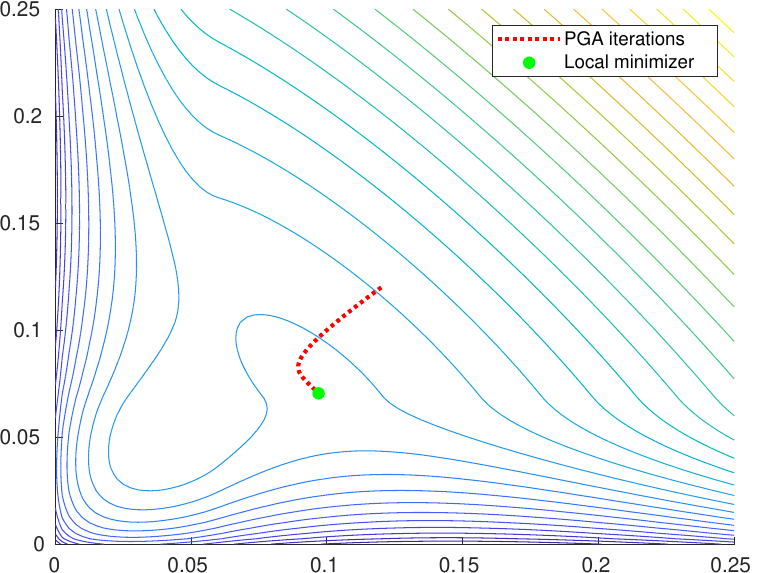}
        \caption{\footnotesize $F_\y$=KL, Shannon entropy.}
        \label{subfig:pc-kl-entropy}
    \end{subfigure}
    \hfill
    \begin{subfigure}{0.3\textwidth}
        \centering
        \includegraphics[width=\linewidth]{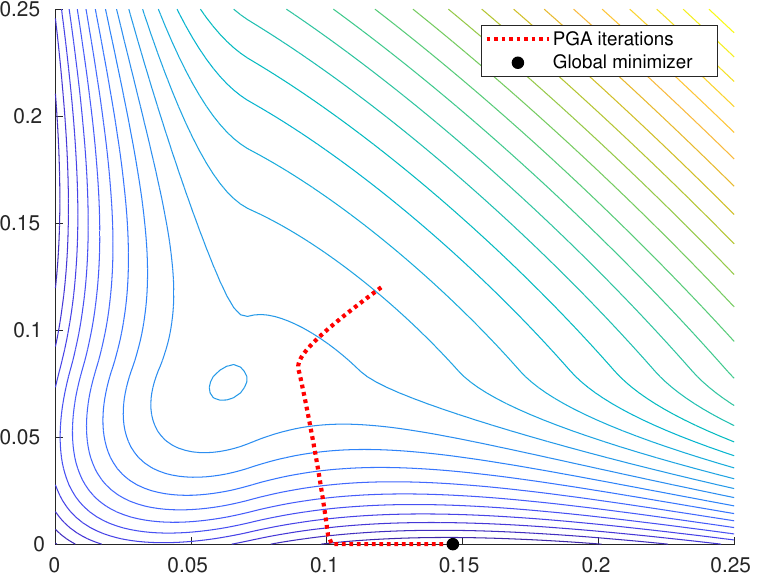}
        \caption{\footnotesize $F_\y$=KL, KL function.}
        \label{subfig:pc-kl-kl}
    \end{subfigure}
    \caption{Trajectory of the iterates $(\x^k)_k$ of PGA used to minimize $J_\Psi$ under different choices of $F_\y\in\left\{\text{LS}, \text{LR}, \text{KL} \right\}$ and generating functions $\psi$.}
    \label{fig:iter-BPGA}
\end{figure}

\subsection{Examples in higher dimensions}\label{sec:high_dim_ex}

We now consider an experiment in higher dimensions. We generate synthetic data with size $M=500$ and $N=1500$. 
The parameter $\lambda_2$ is fixed as $\lambda_2=0$ for both LS and KL problems, and $\lambda_2=0.01$ for LR to guarantee existence, see Theorem~\ref{th:existence-minimizers}.  We fix $\lambda_0 = \alpha F_\y(\mathbf{0})$ with $\alpha \in(0,1)$ manually adjusted such that, on average (among all instances and all tested methods),  the sparsity of the computed solutions corresponds to the one used to generate the data. The parameter $\bm{\gamma}$ verifying \eqref{eq:sup_inf_cond} is here kept fixed to $\bm{\gamma}=\bm{\gamma_{\text{thr}}}$.
For each choice of the data fidelity and the relaxation functional considered, we take $I=100$ realizations of data matrix $\A$ and noisy observation $\y$. Then, we solve these $I$ problems by running the proximal gradient algorithm \eqref{eq:prox-grad-algo} with backtracking line search both on $J_0$ (i.e., a gradient-step on $F_\y$ followed by a hard-thresholding) as well as on the considered relaxed functional $J_\Psi$ using the expressions of proximal points provided by Proposition \ref{prop:prox_formula} and specified for each instance of $B_{\Psi}$. The use of backtracking in these tests has been proved effective due to the potential over-estimation of the Lipschitz constants $L$ of the gradients of the data terms considered, which may result in the underestimation of a constant step-size badly affecting numerical convergence. The same initial point $\x^0=\mathbf{0} \in \Cc^N$ was considered as initialization. As stopping criteria we used a maximal number of 5000 iterations (never attained in our experiments) and a relative tolerance on the difference between two successive iterates of $10^{-6}$. 

\paragraph{Data generation} The realizations $i \in [I]$ are generated according to the data term as follows:
\begin{itemize}
    \item \textit{Least-Squares (LS)}: The rows of $\A_i$ are drawn from a zero mean multivariate normal distribution with covariance matrix $\bm{\Sigma}_\eta$, $\eta \in [0,1)$ given by $[\bm{\Sigma}_\eta]_{m,n} = \eta^{|m-n|}$. The larger $\eta$, the more correlated the columns of $\A_i$ and thus the more difficult the problem. In our experiments we fixed $\eta =0.9$. Normalization is then performed on columns $\a_n$. For each  $i\in[I]$, the observation vector $\y_i$ is  generated as 
    \[
    \y_i=\A_i\bar{\x}_i+\boldsymbol{\upvarepsilon}_i,
    \]
    where $\bar{\x}_i$ is a sparse vector of support $\# \sigma(\bar{\x}_i)=50$ and whose non-zero elements have an absolute amplitude of $1$, with their signs determined by normally distributed random values. The additive white Gaussian noise component $\boldsymbol{\upvarepsilon}_i$ is distributed as $\boldsymbol{\upvarepsilon}_i \sim \mathcal{N}(\mathbf{0}, \sqrt{\varsigma}    \mathbf{I})$, where $\varsigma = \|\A_i\bar{\x}_i\|^2 10^{-\tau/10}$, and $\tau = 8$.
    \item  \textit{Logistic Regression (LR)}: We generated synthetic data following the methodology detailed in~\cite{LRhazim}. For that, we first generated a matrix $\A_i$ as in the LS case. Next, we generated a sparse  vector $\bar{\x}_i$ with $\sharp\sigma(\bar{\x}_i) = 50$ non-zero entries that are equi-spaced and equal to one, representing the true features. Then, the coordinates of the outcome vector $\y_i[m] \in \{0, 1\}$, where $\y_i[m]=1$ is determined by the probability $P(\y_i[m]=1 \, | \,\a_m) = \left( 1 + \e^{-s \left<\a_m,\bar{\x}_i\right>} \right)^{-1}$ with $\a_m$  the m-th row of $\A$ and $s$ fixed to $0.5$ that controls the signal-to-noise ratio. The binary outcomes are then sampled from a Bernoulli distribution based on this probability.
    \item  \textit{Kullback-Leibler (KL)}: The entries of matrix $\A_i=(a_{m,n})_{m,n}$ are independent and identically distributed random variables drawn from a half-normal distribution, i.e., $a_{m,n} = |d_{m,n}|$, where $d_{m,n} \sim \mathcal{N}(0,1)$. For each  $i\in[I]$, the observation vector $\y_i$ is generated by 
    \[
    \y_i = \operatorname{Poisson}(\alpha(\A_i{\bar{\x}_i}+\mathbf{b}))/\alpha,
    \] 
    where ${\bar{\x}_i}$ is a sparse vector with support $\# \sigma({\bar{\x}_i})= 20$ whose nonzero elements are drawn from a uniform distribution, $\mathbf{b}= b\mathbf{e}$ with $b =0.1$ and $\mathbf{e}$ a constant vector of  ones, and $\alpha = 50$ is a gain factor.
\end{itemize}

\paragraph{Relaxations and metrics}
The experiments performed aim to compare the performance of the proximal-gradient algorithm \eqref{eq:prox-grad-algo} in minimizing the $J_0$ functional \eqref{eq:problem_setting} \emph{directly} or \emph{indirectly} through the minimization of the proposed relaxations $J_\Psi$. A good measure to quantify the quality of the minimizer computed is, therefore, the computation of the value of the objective function $J_0$ at convergence of the algorithm, for both problems: the lower, the better. In more detail, for each instance $i \in [I]$, we sort (with possible equality) the objective values $J_0(\hat{\x}^i_{J_0})$ and $J_0(\hat{\x}^i_{J_\Psi})$ (for three choices of the family $\Psi$),  where $\hat{\x}^i_{J_0}$ and $\hat{\x}^i_{J_\Psi}$ are the critical points of $J_0$ and $J_\Psi$ obtained by running PGA  on the $i$-th instance of the problem considered. The functional for which PGA converged to the lowest objective value is ranked 1st (best), the one for which PGA converged to the second lowest objective value is ranked 2nd, etc. For each  functional ($J_0$ and the three instances of $J_\Psi$) we could thus collect the total number of occurrences of being ranked 1st, 2nd, ... among the $I$ generated problems for the three specific choices of the data terms considered. For comparison, we further collected all computing times.

As far as the particular choice of the relaxation functionals $J_\Psi$ is concerned with respect to the family $\Psi$, we considered the following ones:
\begin{itemize}
    \item for the LS and LR data terms: $p$-power generating functions with $p\in\{4/3,3/2,2\}$;
    \item for the KL data term:  $p$-power generating functions with $p\in\{3/2,2\}$ and KL generating function.
\end{itemize}
\change{We recall that \BR{} with the 2-power generating function is nothing but the CEL0~\cite{Soubies2015} penalty,  a specific instance of MCP~\cite{Cun-Hui}.}

\medskip

\begin{remark}
The PGA algorithm is guaranteed to converge only to critical points of the objective functional without the need of them being local minimizers (cf. Proposition~\ref{prop:critical_point_Jpsi} and Corollary~\ref{coro:critical_poin_to_loc_min}). We could 
exploit Corollary~\ref{coro:critical_poin_to_loc_min} to deploy an outer loop around PGA iterates to ensure the convergence to a local minimizer of $J_\Psi$, similarly to the macro algorithm proposed in~\cite{Soubies2015} in the least squares case.
Yet, such critical points that are not local minimizers are typically unstable and very rarely attained. As such, and to keep our illustrative experiments simple, we evaluate our metric on the critical points attained at convergence of the algorithm, without resorting to the macro algorithm strategy. Note that the comparison still remains fair as the use of an outer loop ensuring convergence to local minimizers of $J_\Psi$ can only increase the gap between the performance of PGA on $J_\Psi$ and  $J_0$. Indeed, let $\hat{\x}_{J_\Psi}$ be a critical point of $J_\Psi$ which is not a local minimizer, then the use of an outer loop will necessarily lead to the computation of a local minimizer $\tilde{\x}_{J_\Psi}$ of $J_\Psi$ (and thus of $J_0$) such that $J_0(\tilde{\x}_{J_\Psi}) = J_\Psi(\tilde{\x}_{J_\Psi}) < J_\Psi(\hat{\x}_{J_\Psi}) \leq J_0(\hat{\x}_{J_\Psi})$.
\end{remark}

\paragraph{Results and discussion}
We report in Figure~\ref{fig:histo-all} the histograms of the rankings computed as above.  We first observe that PGA globally performs better in minimizing any exact relaxation $J_\Psi$ rather than the original functional~$J_0$. Yet, this performance varies with the choice of the generating functions used to define \BR{}. 
In particular, among the $p$-power functions, $p=2$ presents the best performance independently on the data term considered. This can be explained by the fact that, for $p \in (1,2]$, the  associated \BR{} admits an interval $[\alpha_n^-,\alpha^+_n]$ (respectively, $[\ell_n^-,\ell_n^+]$) which gets larger (respectively, smaller) as $p$ increases. As such, the larger $p \in (1,2]$, the more are the local minimizers of $J_0$ eliminated by the exact relaxation (cf. Proposition~\ref{Local_minimizers_of_J0_preserved_by_JPsi}). 
Note, however, that using a KL generating function  for problems involving a KL data term improves the performance of PGA, which illustrates the interest of adapting the geometry of the relaxation (that is, choosing a good family $\Psi$) to the geometry of the data term. This is also clearly visible by observing the performance of 2-power generating functions for the LS data term.  In this case,  the concavity condition~\eqref{eq:general-exact-condition} (which for such specific data term is indeed equivalent to~\eqref{eq:sup_inf_cond}) is exploited tightly. In contrast, as far as both the LR and the KL data-term is concerned, all considered \BR{} are defined in terms of the easier (but coarser) concavity condition~\eqref{eq:sup_inf_cond} which could potentially explain the sub-optimality of the results. We believe in fact that tighter relaxations by defining \BR{} relaxations exploiting tailored to the problem at hand and exploiting directly~\eqref{eq:general-exact-condition}, a question that we shall address  in future work.

\medskip

Our experiment shows that no relaxation systematically outperforms all the others. This can be observed even for a few instances of LS problems where PGA reached a better local minimizer by minimizing $J_{\ell_{1.5}}$ than by minimizing $J_{\ell_2}$
. This highlights that the success of a relaxation over another depends not only on its ability to eliminate a higher number of local minimizers but also on the ability of the specific algorithm considered to generate 'optimal' trajectories from a given initial point. In this regard, the development of dedicated solvers (beyond off-the-shelf ones such as PGA) to minimize relaxations $J_\Psi$ which better exploit distinctive properties of B-rex (e.g., Proposition~\ref{Local_minimizers_of_J0_preserved_by_JPsi}) or the generalized duality properties (Proposition \ref{prop:Bregman-composition-of-S}) should be certainly investigated in future work.

\medskip

\begin{remark}
    Note that for a given matrix $\A$, $J_0$ admits less local minimizers when used with a KL data term than with an LS or LR data term, due to the need of the non-negativity constraint. Consequently, in this case the proportion of local minimizers removed by exact relaxation is lower than for the other two data terms. This explains the different ranking histograms observed in Figure~\ref{fig:histo-all} for the KL case.
\end{remark}

\medskip

To conclude, we report in Table~\ref{tab:avg_times} the average computational times required to solve both the original and the relaxed problems discussed above. It is worth commenting that efficiency is largely affected by the computation of proximal points, which depends on the particular choice of the generating functions $\psi_n$. The large standard deviations are due to the number of performed iterations required to achieve the desired tolerance, which varies from an instance of problem to another. 


\begin{figure}
    \centering
    \begin{tikzpicture}
        \node[anchor = north west] (LS) at (0,0) {\includegraphics[width=0.49\textwidth,trim={0cm 0cm 0cm 0.2cm},clip]{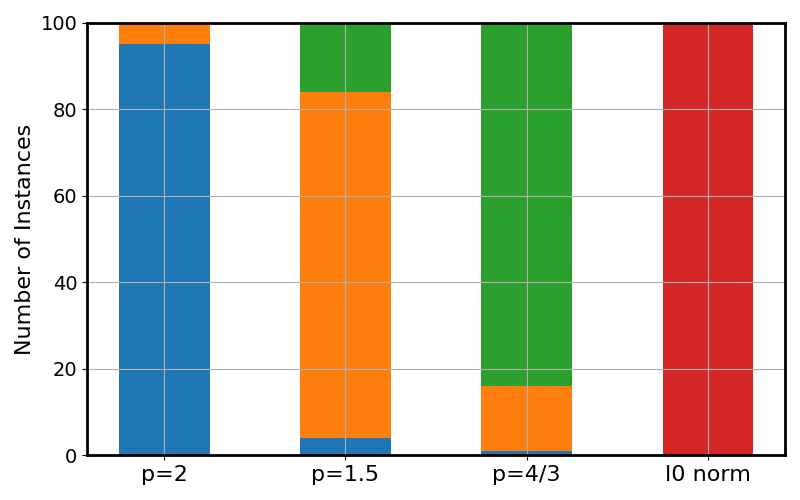}};
        \node at (LS.north) {LS data term};
        \node[anchor = north west] (LR) at (0.5\textwidth,0) {\includegraphics[width=0.49\textwidth,trim={0cm 0cm 0cm 0.2cm},clip]{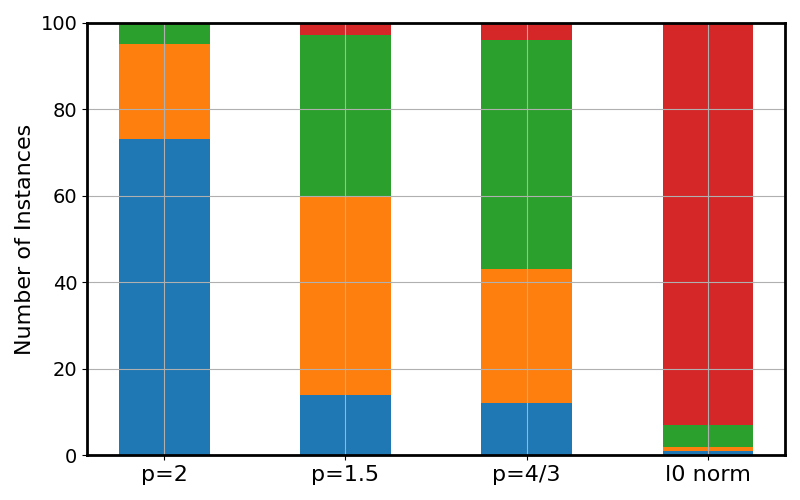}};
        \node at (LR.north) {LR data term};
         \node[anchor = north west] (KL) at (0,-0.35\textwidth) {\includegraphics[width=0.49\textwidth,trim={0cm 0cm 0cm 0.1cm},clip]{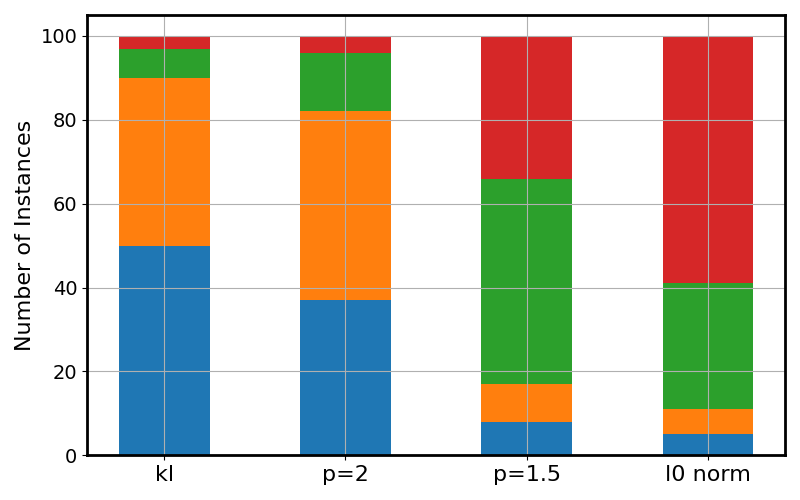}};
        \node at (KL.north) {KL data term};
        \node[anchor = north] (KL) at (0.75\textwidth,-0.33\textwidth) {\includegraphics[width=0.23\textwidth]{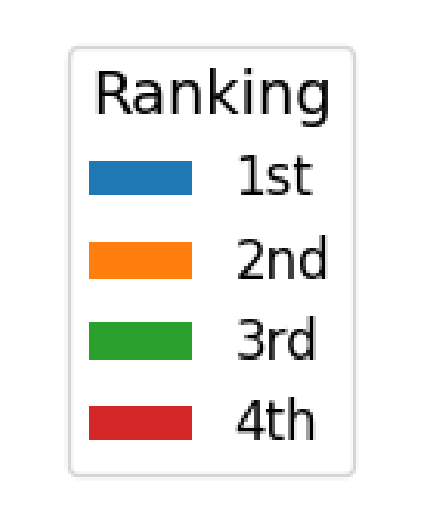}};
    \end{tikzpicture}
    \caption{Comparative ranking distribution of FBS algorithm with backtracking line search applied to the original function $J_0$ and different exact relaxations $J_\Psi$ across $I=100$ generation of problems. The hyper-parameter $\lambda_0$ is such that $\lambda_0 = \alpha F_\y (\mathbf{0})$ with $\alpha = 4\times 10^{-3}$ (for LS data term), $\alpha= 3.8 \times 10^{-3}$ (for LR data term) and $\alpha = 5 \times 10^{-4}$ (for KL data term). The ranking is based on the final objective function values at convergence, as described in the \textit{Metrics} paragraph.}
    \label{fig:histo-all}
\end{figure}
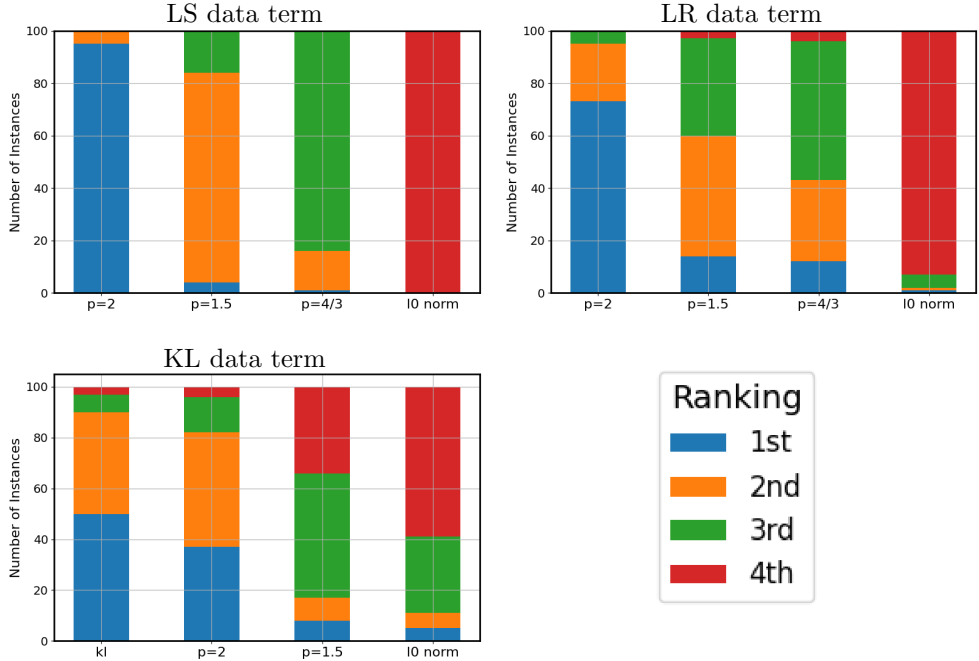

\begin{table}[h!]
\renewcommand{\arraystretch}{1.3}  
\centering
\begin{tabular}{|c|c|c|c|}
\hline
\cellcolor[HTML]{FFFFC7} \textbf{Objective function} & \cellcolor[HTML]{FFFFC7} \textbf{LS (mean $\pm$ std)} & \cellcolor[HTML]{FFFFC7} \textbf{LR (mean $\pm$ std)} & \cellcolor[HTML]{FFFFC7}  \textbf{KL (mean $\pm$ std)} \\ \hline
 $\ell_0$ &  $1.71 \pm 1.29$ & $0.71\pm 1.08$ & $0.21 \pm 0.11$ \\ \hline
$\psi_n = \ell_2$ &  $2.33 \pm 2.41$ & $5.57  \pm 9.01$ & $0.29 \pm 0.05$ \\ \hline
 $\psi_n = \ell_\frac32$ &  $2.10 \pm 1.63$ & $4.06 \pm 3.39$ & $0.88 \pm 0.21$ \\ \hline
 $\psi_n=\ell_\frac43$ or $\psi_n=\mathrm{KL}$ &  $3.68 \pm 2.43$ & $8.00 \pm 13.48$ & $0.49 \pm 0.08$ \\ \hline
\end{tabular}

\caption{Average times (in seconds) over $I=100$ problem instances. The last row refers to the use of generating function $\psi_n$ set to $p = 4/3$ power function for both LS and LR date terms, and  to the KL generating function for KL data term.}
\label{tab:avg_times}
\end{table}

\change{
\paragraph{On recovery errors}

The experiments reported above demonstrate that off-the-shelf algorithms such as PGD are more effective at optimizing the proposed \textit{equivalent} exact relaxations than the original $\ell_0$-regularized problem, thanks to their more favorable optimization landscape. While this constitutes our main motivation stemming from an optimization perspective, we further analyzed whether the minimizers of the relaxed criteria 
are also better in terms of other popular metrics such that F1-score (reflecting the quality of support estimation) and root mean square error (RMSE).

For all LS, LR and KL problems, we generated a matrix $\mathbf{A} \in \R^{500 \times 1000}$ and a ground-truth vector $\bar{\mathbf{x}}$ as described in the previous experiments (with $\eta = 0.8$, $\tau =5$ and $\sharp \|\bar{\x}\|_0 = 50$ for LS, $\eta =0.8$ $s=1$ and $\sharp \|\bar{\x}\|_0 = 25$ for LR. Finaly, for KL, the matrix $\A$ is generated with $\eta=0.1$, using absolute value of each its element. With $\alpha = 50$ and $\sharp \|\bar{\x}\|_0 = 25$. We recall that $\lambda_2 = 0$ for both LS and KL, where for LR, it is fixed to $2 \times10^{-2}$). Then 20 data realizations $\mathbf{y}$ were obtained through different noise generation. For each realization, we solved the considered optimization problems over a grid of regularization strengths $\lambda_0$ and selected the solution yielding the best F1-score. Tables~\ref{tab:F1-score-LS} to~\ref{tab:F1-score-KL} report the average F1-scores and RMSE values (computed over the 20 realizations) for LS, LR, and KL, respectively.

We observe that PGD performs better when applied to the exact relaxations than to the original $\ell_0$-based problem. Furthermore, tailoring the \BR{} relaxation to the geometry of the data fidelity term yields improved results. These results are in agreement with the exact relaxation ones discussed before, which focus on the objective values only.

\begin{table}
\renewcommand{\arraystretch}{1.3}  
\setlength{\tabcolsep}{6pt}
\centering
\begin{tabular}{|c|c|c|c|c|}
\hline
\cellcolor[HTML]{FFFFC7}
 & \cellcolor[HTML]{FFFFC7} $\ell_0$ 
 &\cellcolor[HTML]{FFFFC7}  $\cellcolor[HTML]{FFFFC7} \psi_n = \ell_2$ 
 & \cellcolor[HTML]{FFFFC7} $\psi_n = \ell_{\frac32}$ 
 & \cellcolor[HTML]{FFFFC7}  $\psi_n = \ell_\frac43$ \\
\hline
  F1-score $\pm$ std & $0.820 \pm 0.050  $ & $\mathbf{0.867 \pm 0.026}$& $\underline{0.840 \pm 0.046}$ &$\underline{0.840 \pm 0.041}$ \\
 \hline
 RMSE $\pm$ std & $ 0.479 \pm 0.083$ & $\mathbf{0.451 \pm 0.070}$ & $0.471 \pm 0.073$ & $\underline{0.456 \pm 0.081}$\\
\hline
\end{tabular}
\caption{Least-squares. Best score in bold and second best underlined.}\label{tab:F1-score-LS}
\end{table}

\begin{table}
\renewcommand{\arraystretch}{1.3}  
\setlength{\tabcolsep}{6pt}
\centering
\begin{tabular}{|c|c|c|c|c|}
\hline
\cellcolor[HTML]{FFFFC7}
 &\cellcolor[HTML]{FFFFC7}  $\ell_0$ 
 &\cellcolor[HTML]{FFFFC7}  $\psi_n = \ell_2$ &\cellcolor[HTML]{FFFFC7}  $\psi_n = \ell_{\frac32}$ &\cellcolor[HTML]{FFFFC7}  $\psi_n = \ell_\frac43$ \\
\hline
 F1-score $\pm$ std & $ 0.802 \pm 0.069   $ & $\mathbf{ 0.831 \pm 0.066}$& $\underline{0.820 \pm 0.057}$& $0.823 \pm 0.063  $ \\
 \hline
RMSE $\pm$ std & $ 0.768 \pm 0.361$ & $\mathbf{0.621 \pm 0.185}$ & ${0.671 \pm 0.195}$ & $\underline{0.653 \pm 0.154}$\\
\hline
\end{tabular}
\caption{Logistic regression. Best score in bold and second best underlined.}\label{tab:F1-score-LR}
\end{table}

\begin{table}
\renewcommand{\arraystretch}{1.3}  
\setlength{\tabcolsep}{6pt}
\centering
\begin{tabular}{|c|c|c|c|c|}
\hline
\cellcolor[HTML]{FFFFC7}
 &\cellcolor[HTML]{FFFFC7} $\ell_0$ 
 &\cellcolor[HTML]{FFFFC7} $\psi_n = \ell_2$ 
 &\cellcolor[HTML]{FFFFC7} $\psi_n = \ell_{\frac32}$ 
 &\cellcolor[HTML]{FFFFC7} $\psi_n = \mathrm{KL}$ \\
\hline
 F1-score $\pm$ std & $ 0.883 \pm 0.051   $ & $\mathbf{0.924 \pm 0.036}$& $0.914 \pm 0.038$& $\underline{0.917 \pm 0.031}$ \\
 \hline
RMSE $\pm$ std & $0.420 \pm 0.079$ & $\mathbf{0.391 \pm 0.089}$ & $0.402 \pm 0.081$ & $\underline{0.398 \pm 0.073}$\\
\hline
\end{tabular}
\caption{Kullback Leibler. Best score in bold and second best underlined.}\label{tab:F1-score-KL}
\end{table}

}

\section{Conclusions and outlook}


In this work, we introduced the $\ell_0$-Bregman relaxation (\BR{}), a class of continuous (non-convex) approximations of the $\ell_0$ pseudo-norm leading to exact continuous relaxations of $\ell_0$-based criteria with general (i.e.~non-quadratic) data terms. Our analysis guarantees that replacing the $\ell_0$ term by \BR{} leads to an equivalent optimization problem (same global minimizers) that is continuous and admits less local minimizers. Such relaxed  problems are thus better suited  than the initial one  to be minimized by standard non-convex optimization algorithms, such as proximal gradient algorithm, as illustrated by several numerical experiments. \change{Note that any algorithm (PGD, but also, e.g., ADMM) suited to solve problems in the form~\eqref{eq:problem_setting} via the proximal operator of $\ell_0$ could be readily adapted to minimize the corresponding \BR{} formulation, thanks to the closed-form expression of its proximal operator. Doing so exhibits the advantage of avoiding to be trapped in spurious local minimizers of the original problem, which are effectively eliminated upon relaxation. 

Nonetheless, the proposed exact relaxations remain non-convex and therefore challenging to optimize. In general, the aforementioned off-the-shelf algorithms are not guaranteed -- and typically fail -- to converge to global minimizers. As highlighted in our numerical experiments, the effectiveness of one relaxation over another depends not only on its ability to eliminate spurious local minima, but also on the capacity of the chosen algorithm to follow favorable optimization trajectories given a suitable initialization. An appealing direction for future research is therefore the development of tailored optimization strategies that leverage the favorable properties of \BR{} such as, for instance, the dependency of the set of local minimizers of $J_0$ eliminated by $J_\Psi$ on $\lambda_0$. This could be indeed exploited to enable the use of warm-start strategies in a tree-search manner to efficiently estimate the $\ell_0$-path (i.e., the path of solutions across a range of sparsity levels).
}

\change{

}

\section*{Data availability}
All generated data and code related to this work will be made available online in a repository after publication of the paper.

\section*{Acknowledgments}
The authors are thankful to C\'edric Herzet for helpful comments and suggestions. The authors acknowledge the support received by the projects ANR
MICROBLIND (ANR-21-CE48-0008), ANR JCJC EROSION (ANR-22-CE48-0004) 
ANR JCJC TASKABILE (ANR-22-CE48-0010) and by the GdR ISIS project SPLIN.
The work of LC was supported by the funding received from the European Research Council (ERC) Starting project MALIN under the European Union’s Horizon Europe programme. This output reflects only the views of the authors. The European Commission and the other organizations are not responsible for any use that may be made of the information it contains.

\begin{appendices}
\section{Existence and characterization of minimizers of $J_0$}\label{appendix:existence}

\subsection{Existence of solutions}

The objective of this appendix is to prove the existence of solutions to problems of the form
\begin{equation}\label{eq:jphi}
    \hat \x \in \argmin_{\x \in \Cc^N} \; J_\Phi(\x) := F_\y(\A\x) + \Phi(\x),
\end{equation}
where $F_\y$ is coercive, $\Cc \in \{\R, \R_{\geq 0}\}$ and  $\Phi(\x) = \sum_{n=1}^N \phi_n(x_n)$ where the functions $\phi_n:\Cc\to\R$ are lower semi-continuous and satisfy
\begin{equation}\label{eq:cond_phi}
\begin{cases}
    \phi_n(0)=0, \\
     \phi_n(x) \in [0,\lambda], \quad &\forall x \in (\alpha_n^-,\alpha_n^+)\\
    \phi_n(x) = \lambda, \quad &\forall x \in \R \backslash (\alpha_n^-,\alpha_n^+),
\end{cases}
\end{equation}
for $\lambda > 0$, $\alpha_n^- \leq 0$ and $\alpha_n^+ \geq 0$. Such choice includes the $\ell_0$ pseudo-norm, as well as folded concave penalties such as SCAD~\cite{Fan2001VariableSV} MCP~\cite{Cun-Hui} or the proposed \BR{} \eqref{eq:prop1_equa2}. As $\Phi$ is bounded and $F_\y$ is coercive 
, the existence of global minimizers for $J_\Phi$ is trivial when $\mathrm{ker}(\A) = \{\bm{0}\}$.

The case $\mathrm{ker}(\A) \neq \{\bm{0}\}$ is significantly more involved. Indeed, despite the coercivity of $F_\y$, the composition $F_\y(\A\cdot)$ is not anymore coercive as, for any $\v \in \mathrm{ker}(\A) \setminus \{\bm{0}\}$, we have $F_\y(\A(\eta \v)) \to 0$ when $\eta \to \infty$. Following~\cite{nikolova2013description} where the author treated the $\ell_0$-regularized least-squares problem, we exploit here the notion of asymptotically stable functions to prove the existence of solutions to~\eqref{eq:jphi} in the general case.

\begin{definition}[Asymptotically stable functions~\cite{AuslenderTeboulle}]\label{def:asymptotically-level-stable}
    A l.s.c.~and proper function $f : \R^N \to \R \cup \left\lbrace + \infty \right\rbrace$  is said to be  asymptotically level stable if for each~$\rho >0$, each bounded sequence~$\left\lbrace \beta^{(k)} \right\rbrace  \subset \R  $ and each sequence $\left\lbrace \x^{(k)}\right\rbrace \subset \R^N$ satisfying 
    \begin{equation}\label{eq:seqences_level_stable}
        \x^{(k)} \in \operatorname{lev}\left( f,\beta^{(k)} \right), \quad \|\x^{(k)}\| \to +\infty , \quad \x^{(k)} \|\x^{(k)}\|^{-1} \to \hat{
\x} \in \ker\left(f_{\infty} \right)
,
    \end{equation}
    where $\operatorname{lev}\left( f,\beta^{(k)} \right)$ denotes the $\beta^{(k)}$ sublevel set of $f$ and $f_{\infty}$ denotes the asymptotic (or recession) function\footnote{From~\cite[Theorem 2.5.1]{AuslenderTeboulle}, we have that, for a proper, l.s.c. function $f : \R^N \to \R \cup \{+\infty\}$, for all  $\d \in \R^N$,  $f_{\infty}(\d)=\liminf _{\substack{\d^{\prime} \rightarrow \d \\ t \rightarrow+\infty}} \frac{f\left(t \d^{\prime}\right)}{t}
.$} of $f$,
    there exists~$k_0\in\mathbb{N}$ such that 
    \begin{equation}\label{eq:seqences_level_stable-2}
        (\x^{(k)} -\rho\hat{\x}) \in \operatorname{lev}\left( f, \beta^{(k)} \right) \quad \forall k \geq k_0
        .
    \end{equation}
\end{definition}

A coercive function $f$ satisfies $f_{\infty}(\d) > 0$ for all $\d \neq \mathbf{0}$~\cite[Definition 3.1.1]{AuslenderTeboulle} and thus $\ker\left(f_{\infty} \right) = \{\bm{0}\}$. Hence, it is asymptotically level stable since for each bounded sequence $\left\lbrace \beta^{(k)} \right\rbrace$ there does not exist any sequence $\left\lbrace \x^{(k)}\right\rbrace$ satisfying \eqref{eq:seqences_level_stable}. Let us now provide an intuition on why this notion of  asymptotically level stable functions allows to ensure existence of minimizers for non-coercive functions. First, let us observe that if $f$ is not coercive, then $\ker\left(f_{\infty} \right) \neq  \{\bm{0}\}$ and the vectors $\d \in \ker\left(f_{\infty} \right) \setminus  \{\bm{0}\}$ can be interpreted as the directions along which $f$ lacks coercivity. Hence, a sequence $\left\lbrace \x^{(k)}\right\rbrace$ satisfying~\eqref{eq:seqences_level_stable} belongs to a vector subspace generated by such a direction where $f$ is not coercive. Moreover, this sequence is such that $\{f(\x^{(k)})\}$ is bounded by definition. Taking $\beta^{(k)} = f(\x^{(k)})$ we get from~\eqref{eq:seqences_level_stable-2} that there exists $k_0$ such that for all $k \geq k_0$,
\begin{equation}
    f(\x^{(k)} -\rho\hat{\x}) \leq f(\x^{(k)})
\end{equation}
showing that $-\hat{\x}$ is a descent direction of $f$ at $\x^{(k)}$.  Combining this with the fact that $\|\x^{(k)}\| \to + \infty$ (by definition) gives the intuition that $\rho \mapsto  f(\rho\hat{\x})$, besides not being coercive, should admit a global minimizer. 

The following Proposition~\ref{prop:J0_ALS} shows that the objective function $J_\Psi$ of  Problem~\eqref{eq:jphi} is asymptotically level stable, which indeed  allows us to show desired existence result in Theorem \ref{th:existence-general}.

\begin{proposition}\label{prop:J0_ALS}
    Let $F_\y$ be coercive function. Then under Assumption~\ref{assump} the functional $J_\Phi$ in \eqref{eq:jphi} is asymptotically level stable.
\end{proposition} 
\begin{proof}
First, let us prove that $\ker\left( (J_\Phi)_{\infty } \right)= \ker\left( \A \right) $. 
 We have from~\cite[Theorem 2.5.1]{AuslenderTeboulle}
\begin{align*}
(J_\Phi)_{\infty}(\x) & = \liminf_{\substack{\mathbf{x}' \rightarrow \x \\ t \rightarrow +\infty}}~ \frac{J_\Phi(t\A \x')}{t}  = \liminf_{\substack{\x' \rightarrow \x \\ t \rightarrow +\infty}} ~\frac{F_\y(t \A \x') + \Phi(t\x')}{t} \\
& = \liminf_{\substack{\x' \rightarrow \x \\ t \rightarrow +\infty}} ~\frac{F_\y(t \A \x')}{t} + \frac{\sharp\sigma(\x) \lambda}{t} =\liminf_{\substack{\x' \rightarrow \x \\ t \rightarrow +\infty}}~ \frac{F_\y(t \A \x')}{t} \\
& = \left( F_\y \right)_{\infty} \left( \A\x\right).
\end{align*}
where we used the fact that $\Phi$ is bounded and $\Phi(t\x)=\sharp\sigma(\x) \lambda$ for $t$ sufficiently large (from conditions~\eqref{eq:cond_phi}). Moreover, given that $F_\y$ is coercive 
, we have that $f_{\infty}(\d) > 0$ for all $\d \neq \mathbf{0}$~\cite[Definition 3.1.1]{AuslenderTeboulle}. It then follows that $(J_\Phi)_{\infty}(\x) = \left( F_\y \right)_{\infty} \left( \A\x\right) \neq 0$ if and only if $\x \notin \mathrm{ker}(\A)$. This means that $\ker((J_\Phi))_\infty) = \ker(\A)$. 

 Let now $\left\lbrace \beta^{(k)} \right\rbrace \subset \R $ be a bounded sequence, and let $\left\lbrace\x^{(k)}  \right\rbrace \subset \R^N$ satisfy~\eqref{eq:seqences_level_stable}  with $\x^{(k)}  \|\x^{(k)} \|^{-1} \to \hat{\x} \in \ker\left( (J_ \Phi)_{\infty} \right)$. For $\rho>0$, let us compare $\Phi\left(\x^{(k)}  -\rho \hat{\x}\right) $ and $\Phi\left(\x^{(k)}\right)$:
\begin{itemize}
    \item If $n \in \sigma\left( \hat{\x}\right)$, then $\hat{x}_n= \lim_{\substack{k \to +\infty}} {x^{(k)}_n}{\|\x^{(k)}\|^{-1}} \neq 0$. It then follows that $x^{(k)}_n$ itself must be growing unbounded. Thus there exists $k_n \in \N$ such that $\forall k \geq k_n$
\begin{equation}\label{eq:compare1}
        |x^{(k)}_n| > \max\left\lbrace |\alpha^-_n|,\alpha^+_n\right\rbrace \quad  
        \Longrightarrow \quad  \phi\left(x^{(k)}_n-\rho \hat{x}_n\right) \underset{\eqref{eq:cond_phi}}{\leq} \lambda=\phi\left(x^{(k)}_n\right).
    \end{equation}
    \item If $n \in \sigma^c \left( \hat{\x}\right)$, then  
 \begin{equation}\label{eq:compare2}
       x^{(k)}_n -\rho \hat{x}_n =  x^{(k)}_n  \quad  
        \Longrightarrow \quad  \phi\left(x^{(k)}_n-\rho \hat{x}_n\right)=\phi\left(x^{(k)}_n\right).
   \end{equation}
\end{itemize}
Defining $k_0:=\max_{ n \in \sigma(\hat \x)} k_n $, we get from~\eqref{eq:compare1} and~\eqref{eq:compare2} that 
\begin{equation}\label{eq:norm0_level_stable}
    \Phi\left(\x^{(k)}-\rho \hat{\x} \right) \leq 
  \Phi\left(\x^{(k)}\right), \quad \forall k \geq k_0.
\end{equation}
Moreover, using the fact that $\hat{\x} \in  \ker\left( (J_\Phi )_{\infty }  \right) = \ker\left( \A \right)$, we have
\begin{equation}\label{eq:level_sstable_F}
    F_\y \left( \A(\x^{(k)}-\rho\hat{\x}) \right) =   F_\y \left( \A\x^{(k)}\right)
    .
\end{equation}
Finally, combining~\eqref{eq:norm0_level_stable} and~\eqref{eq:level_sstable_F}, we get  
\begin{equation}
    J_\Phi\left(\x^{(k)}-\rho \hat{\x}\right) \leq   J_\Phi\left(\x^{(k)}\right) \leq \beta^{(k)} \quad \forall k \geq k_0
    ,
\end{equation}
which completes the proof.
\end{proof}

\begin{theorem}[Existence of solutions to Problem~\eqref{eq:jphi}] \label{th:existence-general}
   Let $F_\y$ be a coercive function. Then under Assumption~\ref{assump}, the solution set of Problem~\eqref{eq:jphi}
   is nonempty.
\end{theorem}
\begin{proof}
Under Assumption~\ref{assump},  $\inf J_\Phi > -\infty$ and thus  $(J_\Phi)_{\infty}(\z) \geq 0 \; \forall \z \in \R^N \setminus \{\bm{0}\}$~\cite[p. 97]{AuslenderTeboulle}.
   Moreover, from Proposition~\ref{prop:J0_ALS}, we  have that $J_\Phi$ is asymptotically level stable. Hence, all the conditions of~\cite[Corollary  3.4.3]{AuslenderTeboulle} are satisfied which proves that the solution set  of Problem~\eqref{eq:jphi} is nonempty.
\end{proof}

\subsection{Proof of Proposition~\ref{prop:local_min_J0}}\label{appendix:characterization}
We observe that for all $\x\in\mathcal{C}^N$ 
\begin{equation}\label{eq:def_H_sig}
  H(\x):=F_\y(\A\x) +\frac{\lambda_2}{2}\|\x\|^2 
=F_\y(\A_\sigma\x_\sigma) +\frac{\lambda_2}{2}\|\x_\sigma\|^2= H_\sigma(\x_\sigma),
\end{equation}
where $\sigma$ denotes the support of $\x$, $H : \Cc^N \to \R$ and $H_\sigma : \R^{\sharp \sigma} \to \R$.
\begin{enumerate}
        \item[$\Longrightarrow$] Let $\hat{\x} \in \Cc^N$ be a (local) minimizer of~$J_0$  and assume that~$\hat{\x}_{\hat{\sigma}}$ does not solve~\eqref{eq:restriction_j0}. Then, from the convexity of~$F_\y$, for all neighborhoods~$\mathcal{N} \subseteq \Cc^N$ of~$\hat{\x}$, there exists~$\x^* \in \mathcal{N}$ such that
        \begin{equation*}
        \left\lbrace
             \begin{array}{l}
             \sigma^* = \sigma \left( \x^* \right) \subseteq \hat \sigma =\sigma(\hat{\x} ) \\
             H_{\hat \sigma}(\x_{\sigma^*}^*) < H_{\hat \sigma} (\hat{\x}_{\hat \sigma})
            \end{array} \right.\; \Longrightarrow \; 
             \left\lbrace
             \begin{array}{l}
                \|\x^*\|_0 \leq \|\hat{\x}\|_0, \\
            H_\y(\x^*)= H_{\hat \sigma}(\x_{\sigma^*}^*) < H_{\hat \sigma} (\hat{\x}_{\hat \sigma})=H_\y (\hat{\x}).
             \end{array} \right.
        \end{equation*}
        We thus easily observe that~$J_0({\x}^*) < J_0(\hat{\x})$, which contradicts the fact that~$\hat{\x}$ is a (local) minimizer of $J_0$, hence the claim follows.
        \item[$\Longleftarrow$] Let~$\hat{\x}_{\hat{\sigma}} \in \Cc^{\sharp\hat{\sigma}} $ be a solution of problem~\eqref{eq:restriction_j0}, $\lambda_0 > 0$, and define  $ \rho_1: = \min_{\substack{n \in \sigma(\hat{\x})}} |\hat{x}_n|$. Then, there exists an open neighborhood~$\mathcal{N}_1\subset \mathcal{B}(\hat \x, \rho_1)$ of $\hat{\x}$ such that for all $\x \in \mathcal{N}_1$, we have 
    \begin{equation}\label{eq:local_min_sigma}
            \sigma\left(\hat{\x} \right) \subseteq \sigma\left( \x \right).
    \end{equation}      
        Furthermore, since $F_\y$ is continuous at~$\hat{\x}$ and hence $H_\y$, there exists $\rho_2 > 0$ such that for all open neighborhoods $\mathcal{N}_2 \subset \mathcal{B}(\hat \x, \rho_2) $ of $\hat \x$ we have ,\begin{equation}\label{eq:continuity_Fy}
            \left|H_\y \left(\x\right) - H_\y \left(\hat{\x}\right) \right| < \lambda_0, \quad \forall   \x \in \mathcal{N}_2.
        \end{equation}
   Set now $\rho := \min\left\lbrace \rho_1 , \rho_2\right\rbrace$. Then,  there exists an open  neighborhood $\mathcal{N} \subset \mathcal{B}(\hat \x, \rho)$ such that both~\eqref{eq:local_min_sigma} and~\eqref{eq:continuity_Fy} hold. Moreover, as $\hat{\x}_{\hat{\sigma}} \in \Cc^{\sharp\hat{\sigma}} $ solves~\eqref{eq:restriction_j0}, we have that
   \begin{equation}\label{eq:ineq_Fy_rest}
       \forall \x \in \mathcal{N} \cap K_{\hat{\sigma}}, \; H_\y \left(\hat{\x}\right) =  H_{\hat \sigma} \left(\hat{\x}_{\hat{\sigma}}\right) \leq H_{\hat \sigma} \left(\x_{\hat \sigma}\right)=H_\y \left(\x\right)
       ,
   \end{equation}
   where $K_{\hat{\sigma}}=\left\lbrace \x \in \mathbb{R}^N~:~x_n =0~ \forall n \notin \hat \sigma  \right\rbrace $. We can distinguish two cases for all $\x \in \mathcal{N} \cap \Cc^N$:
    \begin{itemize}
        \item If~$\sigma \left( \x\right) = \sigma \left( \hat{\x}\right)$, then we get from~\eqref{eq:ineq_Fy_rest}  that~$ H_\y \left(\hat{\x} \right)+ \lambda_0 \sharp \hat{\sigma} \leq H_\y(\x)+\lambda_0 \sharp \hat{\sigma}$.
        \item  If $\sigma(\x) \supset \sigma(\hat{\x}) = \hat{\sigma}$, then $\|\hat \x\|_0 \leq \|\x\|_0 -1$, 
        which combined with~\eqref{eq:continuity_Fy}  entails
        \begin{align*}
            H_\y(\hat{\x})-\lambda_0 < H_\y(\x)  \Longleftrightarrow & \, H_\y(\hat{\x})+\lambda_0(\|\x\|_0 -1)  < H_\y(\x)+\lambda_0\|\x\|_0  \\
             \Longrightarrow & \,  H_\y(\hat{\x})+\lambda_0 \|\hat \x\|_0 < H_\y(\x)+\lambda_0 \|\x\|_0
             .
             \end{align*}
    \end{itemize}
    Therefore, we have shown that for all $\x \in \mathcal{N} \cap \Cc^N$, $J_0\left(\hat{\x} \right) \leq J_0 \left(
        \x\right)$, as required.
    \end{enumerate}
    
    \subsection{Proof of Lemma~\ref{lemma:0-is-strict-min}}\label{appendox:0strict}

Since $F_\mathbf{y}(\A\cdot)$ is convex at $\mathbf{0}$, then it is locally Lipschitz at $\mathbf{0}$ and thus is calm\footnote{A function $g: \R^N \to \R$ is said to be calm at $\hat{\x}$, if there exists constants ${l} \geq 0$ and ${\varepsilon} > 0$ such that $g(\x) \geq g(\hat{\x})-{l} \|\x-\hat{\x}\| \; \forall \hat{\x} \in \Ball(\hat{\x};{\varepsilon})$ \cite[Chapter 8, Section F]{rockafellar2009variational}. } at $\mathbf{0}$: there exists $l \geq 0$, $r > 0$, and a neighborhood $\mathcal{N} \supset \Ball(\mathbf{0}; r)$ of $\mathbf{0}$ such that
\[
    F_\mathbf{y}(\mathbf{A}\mathbf{v}) \geq F_\mathbf{y}(\mathbf{0}) - l \|\mathbf{v}\|, \quad \forall \mathbf{v} \in \Ball(\mathbf{0}; r) \cap \Cc^N.
\]
Now, let $\rho := \min\left\lbrace r, \frac{\lambda_0}{l+1} \right\rbrace$. Therefore, for all $\mathbf{v} \in \Ball(\mathbf{0}; \rho) \cap \Cc^N \setminus \{\mathbf{0}\}$, we have  $\lambda_0 \|\mathbf{v}\|_0 \geq \lambda_0 > 0$. Since $\v \neq \mathbf{0}$ and $\lambda_0 - l\|\mathbf{v}\| > \|\mathbf{v}\| > 0$, we have:
\begin{align*}
    J_0(\mathbf{v}) &= F_\mathbf{y}(\A\mathbf{v}) + \lambda_0 \|\mathbf{v}\|_0 + \frac{\lambda_2}{2} \|\v\|^2 \\
    & \geq F_\mathbf{y}(\mathbf{0}) - l\|\mathbf{v}\| + \lambda_0  + \frac{\lambda_2}{2} \|\v\|^2 \\
    & > F_\mathbf{y}(\mathbf{0}) = J_0(\mathbf{0}),
\end{align*}
as required.

\subsection{Proof of Theorem~\ref{thr:strict-minimizer}}\label{appendix:strict-min}
 The proof of Theorem~\ref{thr:strict-minimizer} is a direct consequence of Lemma~\ref{lem:strict-min} below which extends \cite[Theorem 3.2 (i) and (ii)]{nikolova2013description} to problems of the form~\eqref{eq:problem_setting}. Indeed, the objective function of subproblem~\eqref{eq:restriction_j0} is always strictly convex and coercive when $\lambda_2>0$ and thus admits a unique global minimizer. When $\lambda_2=0$ we get the strict convexity and coercivity of the objective function of subproblem~\eqref{eq:restriction_j0} only if $\operatorname{rank}(\A_{\hat{\sigma}})=\sharp \hat{\sigma}$ under Assumption~\ref{assump}.
 
\begin{lemma}\label{lem:strict-min}
     Let $\hat{\x}$ be a (local) minimizer of $J_0$. Define $\hat{\sigma}=\sigma(\hat{\x})$. Then $\hat{\x}$ is a strict local minimizer of $J_0$ if and only if $\hat{\x}_{\hat{\sigma}}$ is the unique solution of the subproblem~\eqref{eq:restriction_j0}.
\end{lemma}

\begin{proof}
Let us define 
\begin{equation}\label{eq:kspace}
    {K}_{\hat{\sigma}} := \lbrace \z \in \Cc^N \; ; \: z_n=0, \; \forall n \in \hat{\sigma}^c\rbrace
    .
\end{equation}
Recalling Lemma 1.2(i) in~\cite{nikolova2013description}, for $\hat{\x} \in \R^N \setminus \{\mathbf{0}\}$, by setting $\rho:=\min_{\substack{n \in \hat{\sigma}}} |\hat{x}_n|$, we have
\begin{equation}\label{eq:lemma1.2-nicolas}
    \|\hat{\x}+\z\|_0=\sum_{n \in \hat{\sigma}}|\hat{x}_n|_0 + \sum_{n \in \sigma^c} |z_n|_0, \quad \forall \z \in \Binf(\mathbf{0};\rho)
    .
\end{equation}

\noindent We proceed by proving both implications
\begin{enumerate}
    \item[$\Longrightarrow$] Let $\hat{\mathbf{x}} \neq \mathbf{0}$ be a strict (local) minimizer of $J_0$. Suppose that $\hat{\mathbf{x}}_{\hat \sigma}$ is not the unique solution of~\eqref{eq:restriction_j0}. Hence we can find $\z$ such that
 \begin{equation}
     \left\lbrace \begin{array}{l}
        \z \in \Binf(\mathbf{0};\rho) \cap {K}_{\hat{\sigma}} \\
        H_{\hat \sigma}(\hat{\mathbf{x}}_{\hat \sigma} + \z_{\hat{\sigma}}) = H_{\hat \sigma}(\hat{\mathbf{x}}_{\hat \sigma})
     \end{array}\right.
 \end{equation}
 where $H_{\hat \sigma}$ is defined as in~\eqref{eq:def_H_sig}.
 It follows from~\eqref{eq:lemma1.2-nicolas} and the fact that (by definition of $\z$) $\mathbf{A}\mathbf{z} = \mathbf{A}_{\hat{\sigma}} \mathbf{z}_{\hat{\sigma}}$
 \begin{align*}
  J_0(\hat{\mathbf{x}}+\mathbf{z}) &= H_{\hat \sigma}\left( \hat{\mathbf{x}}_{\hat{\sigma}}+\mathbf{z}_{\hat{\sigma}}\right) + \lambda_0 \|\hat{\x}+\z\|_0 \\
&= H_{\hat \sigma}( \hat{\mathbf{x}}_{\hat{\sigma}}) + \lambda_0 \sum_{n \in \hat{\sigma}}|\hat{x}_n|_0 + \lambda_0 \sum_{n \in \sigma^c} |z_n|_0 \\
&= F_\mathbf{y}(\mathbf{A}\hat{\mathbf{x}}) + \frac{\lambda_2}{2} \|\hat{\x}\|^2+ \lambda_0 \sum_{n \in \hat{\sigma}}|\hat{x}_n|_0 = J_0(\hat{\mathbf{x}}),
\end{align*}
which contradicts the fact that $\hat{\x}$ is a strict local minimizer of $J_0$. Hence $\hat{\mathbf{x}}_{\hat \sigma}$ is  the unique solution of~\eqref{eq:restriction_j0}.

\item[$\Longleftarrow$]  Let $\hat{\x}$ be a (local) minimizer of $J_0$ such that  $\hat{\mathbf{x}}_{\hat \sigma}$ is  the unique solution of~\eqref{eq:restriction_j0}.  Hence,  for all $\z \in {K}_{\hat{\sigma}}$
\begin{equation}\label{eq:proof_strict_min_recip-1}
H_\y(\hat{\x}+\z)=H_{\hat \sigma}(\hat{\x}_{\hat{\sigma}}+\z_{\hat{\sigma}}) >  H_{\hat \sigma}(\hat{\x}_{\hat{\sigma}})=H_\y(\hat{\x}),
\end{equation}
where again $H_\y$ and  $H_{\hat {\sigma}}$ are defined as in~\eqref{eq:def_H_sig}. Now, defining  $\rho := \min_{\substack{n \in \hat{\sigma}}} |\hat{x}_n|$, we get from~\eqref{eq:lemma1.2-nicolas} and~\eqref{eq:proof_strict_min_recip-1} that for all  $\z \in ({K}_{\hat{\sigma}} \cap \Binf(\mathbf{0};\rho) )\setminus \{\bm{0}\} $
\begin{align*}
  J_0(\hat{\mathbf{x}}+\mathbf{z}) &= H_\y\left( \hat{\mathbf{x}}+\mathbf{z}\right) + \lambda_0 \|\hat{\x}+\z\|_0 \\
&>H_{\y}( \hat{\mathbf{x}}) + \lambda_0 \sum_{n \in \hat{\sigma}}|\hat{x}_n|_0 + \lambda_0 \sum_{n \in \sigma^c} |z_n|_0 \\
&= F_\mathbf{y}(\mathbf{A}\hat{\mathbf{x}}) + \frac{\lambda_2}{2} \|\hat{\x}\|^2+ \lambda_0 \sum_{n \in \hat{\sigma}}|\hat{x}_n|_0 = J_0(\hat{\mathbf{x}}),
\end{align*}
On the other hand, for $\z \notin {K}_{\hat{\sigma}}$, we use the fact that  $H_\y(\cdot)$ is convex and hence calm at $\hat{\x}$. This implies that it exists 
 $r>0$ and $l \geq 0$ such that 
\begin{equation}
   H_\y(\hat{\x}+\z) \geq H_\y(\hat{\x})-l\|\z\|, \quad \forall \z \in \Binf(\mathbf{0};r) \cap \Cc^N 
\end{equation}
Defining $\tilde \rho:=\min \lbrace \rho, r, \frac{\lambda_0}{l+1} \rbrace$. Using~\eqref{eq:lemma1.2-nicolas} and noticing that $\sum_{n \in \sigma^c} |z_n|_0 \geq 1$ for $\z \notin {K}_{\hat{\sigma}}$, we get that for all $\z \in \Binf(\mathbf{0};\tilde \rho) \cap \Cc^N \setminus {K}_{\hat{\sigma}} $
\begin{align*}
   J_0(\hat{\x}+\z)&=H_\y(\hat{\x}+\z)+\lambda_0 \sum_{n \in \hat{\sigma}}|\hat{x}_n|_0 + \lambda_0 \sum_{n \in \sigma^c} |z_n|_0\\
   & \geq H_\y(\hat{\x})-l\|\z\| +\lambda_0 \|\hat{\x}\|_0+\lambda_0 \sum_{n \in \sigma^c} |z_n|_0 \\
   & =J_0(\hat{\x})-l\|\z\|+\lambda_0  >J_0(\hat{\x}).
\end{align*}
We deduce from all these derivations that $J_0(\hat{\x}+\z) > J(\hat{\x}), \; \forall \z \in \Binf(\mathbf{0};\tilde \rho) \cap \Cc^N$, hence $\hat{\x}$ is a strict minimizer. 
\end{enumerate}

\end{proof}

\subsection{Proof of Theorem~\ref{th:glob_min_J0_strict}}\label{sec:proof_glob_min_J0_strict}

   Let $\hat{\x} \in \Cc^N$ be a global minimizer of $J_0$. Define $\hat{\sigma}=\sigma(\hat{\x})$. If $\hat{\x}=\mathbf{0}$, then $\hat{\x}$ is strict minimizer by Lemma~\ref{lemma:0-is-strict-min}. Now, let $\hat{\x} \neq \mathbf{0}$. We suppose that $\hat{\x}$ is a non-strict minimizer. Therefore, Theorem~\ref{thr:strict-minimizer} fails, meaning that $\lambda_2=0$ and $\dim \ker \A_{\hat{\sigma}} \geq 1$.
    Let us take  $\z \in \R^\N$  such that $\z_{\hat{\sigma}} \in \ker \A_{\hat{\sigma}}$,   $\z_{\hat{\sigma}^c} = \bm{0}$, and $[\z_{\hat{\sigma}}]_k >0$ for some $k \in [\sharp \hat{ \sigma}]$.
    Then, there exists $\eta >0$ such that 
    \begin{equation}\label{eq:proof_min_glob_strict-1}
        \Bar{\x}=\hat{\x}-\eta \z \in \Cc^N \; \text{and } \|\Bar{\x}\|_0 \leq \|\hat{\x}\|_0-1 .
    \end{equation}

To prove the existence of such an $\eta$, we distinguish two cases:
\begin{enumerate}
    \item If $\Cc = \R^N$, then taking $\eta = [\hat{\x}_{\hat{\sigma}}]_k / [\hat{\z}_{\hat{\sigma}}]_k >0$ is a valide choice.
    \item If $\Cc = \R_{\geq 0}^N$, then we define the point ${\x}_{\mathrm{ext}} \in \R^{\sharp \hat \sigma}_{\geq 0}$ as
\begin{equation}
    {\x}_{\mathrm{ext}} = \hat{\x}_{\hat{\sigma}} -  \eta_1 \z_{\hat \sigma},
\end{equation}
with $\eta_1 >[\hat{\x}_{\hat{\sigma}}]_k / [\hat{\z}_{\hat{\sigma}}]_k >0$.
Hence we have that ${\x}_{\mathrm{ext}} \in E:= \mathrm{ext}(\R_{\geq 0}^{\sharp \hat{\sigma}})$, the exterior of $\R_{\geq 0}^{\sharp \hat{\sigma}}$. Moreover, (by definition) $\hat{\x}_{\hat{\sigma}} \in I:= \mathrm{int}(\R_{\geq 0}^{\sharp \hat{\sigma}})$, the interior of $\R_{\geq 0}^{\sharp \hat{\sigma}}$. Then, there exists $t^* \in (0,1)$ such that $\mu(t^*) \in B := \partial (\R_{\geq 0}^{\sharp \hat{\sigma}})$ (the boundary of $\R_{\geq 0}^{\sharp \hat{\sigma}}$) with $\mu$ the path defined by
\begin{equation}
    \mu(t) = t \hat{\x}_{\hat{\sigma}} + (1-t) \bar {\x}_{\mathrm{ext}} .
\end{equation}
This is due to the fact that $\{I, B, E\}$ forms a partition of $\R^{\sharp \hat{ \sigma}}$. Injecting the expression of ${\x}_{\mathrm{ext}}$ in the last equation, we have
\begin{equation}
     \mu(t^*) =  \hat{\x}_{\hat{\sigma}} - \eta_1(1-t^*) \z_{\hat{\sigma}} \;\text{ with }\;  \|\mu(t^*)\|_0 < \|\hat{\x}_{\hat{\sigma}} \|_0 = \sharp \hat{\sigma},
\end{equation}
 This shows that $\eta =\eta_1(1-t^*) >0$ is a valid choice for~\eqref{eq:proof_min_glob_strict-1}.
\end{enumerate}

     Finally, from  the fact that $\eta \z_{\hat{\sigma}} \in \ker \A_{\hat{\sigma}}$, we have $\A\hat{\x}=\A_{\hat{\sigma}}\hat{\x}_{\hat{\sigma}}=\A_{\Bar{\sigma}}\Bar{\x}_{\Bar{\sigma}}=\A\Bar{\x}$.
     Therefore, $F_\y(\A\hat{\x})=F_\y(\A\Bar{\x})$ and we get 
     \begin{align*}
         J_0(\Bar{\x})&=F_\y(\A\Bar{\x})+\lambda_0 \|\bar \x\|_0\\
         & \leq F_\y(\A\hat{\x})+\lambda_0( \|\hat \x\|_0 -1 ) < J_0(\hat{\x})
         ,
     \end{align*}
     which contradicts the fact $\hat{\x}$ is a global minimizer of $J_0$. Hence $\hat{\x}$ is a strict minimizer of $J_0$.

\section{\BR{} penalty and exact relaxation results}   
\subsection{Proof of Proposition~\ref{prop:Breg_l0}}\label{appendix:structure_bcel0}
The separability of $B_\Psi$ comes directly from the separability of both the Bregman distance $D_\Psi$ and the $\ell_0$ pseudo-norm.

 We thus focus on the proof of~\eqref{eq:prop1_equa2}  (1-dimensional functional). 
For $z\in\Cc$, we have that $\alpha -  d_{\psi_n} (\cdot,z) \leq  \lambda_0 | \cdot|_0 $ if and only if $ \alpha \leq   d_{\psi_n}(\cdot, z) + \lambda_0 | \cdot |_0 $. As such,  the supremum with respect to $\alpha$ for $z\in \Cc$ in~\eqref{eq:prop:def_bergman1} is given by,
\begin{equation}\label{eq:proof_l0_Breg-1}
    \alpha= \inf_{x \in \Cc} \; \lambda_0 |x|_0 +  d_{\psi_n}(x,z).
\end{equation}
Since $\psi_n$ is strictly convex (and so is $d_{\psi_n}(\cdot,z)$ for all $z\in\mathcal{C}$), we have that for all $z\in \Cc$, $x \mapsto \lambda_0 |x|_0 +  d_{\psi_n}(x,z)$ 
admits two local minimizers at $x=0$ with value $d_{\psi_n}(0,z)$ and $x=z$  with value $\lambda_0$ (property of Bregman divergences). Combining with~\eqref{eq:proof_l0_Breg-1}, we get 
\begin{equation}\label{eq:proof_prop1_eq2}
    \alpha =  \min\left(\lambda_0,  d_{\psi_n}(0,z) \right) =
    \left\lbrace
    \begin{array}{ll}
           d_{\psi_n}(0,z) & \text{ if } z \in [\alpha_n^-,\alpha_n^+],\\
         \lambda_0 & \text{ otherwise}
    \end{array}\right.
    ,
\end{equation}
where the interval $[\alpha_n^-,\alpha_n^+] \ni 0$ defines the $\lambda_0$-sublevel set of $z \mapsto d_{\psi_n}(0,z) =  \psi_n'(z)z - \psi_n(z) + \psi_n(0)$. Note that such a bounded interval exists thanks to the assumption that $z \mapsto \psi_n'(z)z - \psi_n(z)$ is coercive in Definition~\ref{def:Breg_L0}.
Injecting this optimal value for $\alpha$ in~\eqref{eq:prop:def_bergman1}, it remains to compute the supremum with respect to $z$
\begin{equation}\label{eq:proof_A3}
    \beta_{\psi_n}(x) = \sup_{z \in \Cc} \; \min\left(\lambda_0,  d_{\psi_n}(0,z) \right) -  d_{\psi_n}(x,z).
\end{equation}

To derive its analytical expression, we distinguish the following cases:\\
\underline{If $x \in [\alpha_n^-, \alpha_n^+]$}, we compute the sup on $\Cc \setminus [\alpha_n^-, \alpha_n^+]$ and $[\alpha_n^-, \alpha_n^+]$, and then combine the results.
\begin{itemize}
   \item 
  When \( z \in \Cc \setminus [\alpha_n^-, \alpha_n^+] \), we have \(\min\left(\lambda_0, d_{\psi_n}(0,z)\right) = \lambda_0\). Therefore, 
\[
\sup_{z \in \Cc \setminus [\alpha_n^-, \alpha_n^+]} \left(\lambda_0 - d_{\psi_n}(x,z)\right) = \begin{cases}
    \lambda_0 - d_{\psi_n}(x, \alpha_n^-),\; \text{if } x \in [\alpha_n^-, 0],\\
    \lambda_0 - d_{\psi_n}(x, \alpha_n^+), \; \text{if } x \in [0, \alpha_n^+]
\end{cases}
\]
The reason is that, for $x \in [0, \alpha_n^+]$, the function \( g(z) :=\lambda_0 -d_{\psi_n}(x,z) =\lambda_0 -\psi_n(x) + \psi_n(z) + \psi_n'(z)(x - z) \) is non-increasing for all \( z \geq \alpha_n^+ \). Indeed, \( g'(z) = (x - z)\psi_n''(z) \leq 0\), as \(\psi_n\) is convex and \( x < z \)). Hence, the supremum of \( g \) is attained at \( z = \alpha_n^+ \), and  similarly,  at \( z = \alpha_n^- \) for $x \in [\alpha_n^-, 0]$.
\item When $z \in [\alpha_n^-, \alpha_n^+]$, we have \(\min\left(\lambda_0, d_{\psi_n}(0,z)\right) = d_{\psi_n}(0,z)\).
Hence,
   \begin{align*}
        &  \psi_n(0) - \psi_n(x)  +  \sup\limits_{\substack{{z \in [\alpha_n^-,\alpha_n^+]}}}\; - \psi_n'(z) (0-z) + \psi_n'(z)(x-z) \\
         = \;  & \psi_n(0) - \psi_n(x)  + \sup\limits_{\substack{{z \in [\alpha_n^-,\alpha_n^+]}}}\;  \psi_n'(z)x.
    \end{align*}
    Since $\psi_n$ is strictly convex, $\psi_n'$ is increasing and the last sup is thus attained at 
    \begin{equation}
        z \in \left\lbrace
        \begin{array}{ll}
         \{ \alpha_n^- \}   & \text{ if } x <0   \\
           \{ \alpha_n^+ \}  & \text{ if } x > 0 \\
           \left[\alpha_n^-,\alpha_n^+\right] & \text{ if } x = 0
        \end{array}\right. \; \text{with the value } \;
         \left\lbrace
        \begin{array}{ll}
        \psi_n(0) - \psi_n(x) + \psi_n'(\alpha_n^-) x  & \text{ if } x \leq 0   \\
          \psi_n(0) - \psi_n(x) + \psi_n'( \alpha_n^+)x  & \text{ if } x \geq 0
        \end{array}\right. .
    \end{equation}
\end{itemize}
From the definition of $\alpha_n^\pm$, we have that $d_{\psi_n}(0,\alpha_n^\pm) = \lambda_0$, therefore $\lambda_0 - d_{\psi_n}(x, \alpha_n^\pm) = d_{\psi_n}(0,\alpha_n^\pm) - d_{\psi_n}(x, \alpha_n^\pm) = \psi_n(0)-\psi_n(x) + \psi_n'(\alpha_n^\pm) x$. As such, the two partial sup computed in the above two bullets have the same expression and we get that
\begin{align}
\forall x \in [\alpha_n^-, \alpha_n^+], \quad \beta_{\psi_n}(x) = \psi_n(0) - \psi_n(x) + \psi_n'(\alpha_n^\pm) x. \label{eq:case1-of-x}
\end{align}
\underline{If $x \in \Cc \setminus [\alpha_n^-, \alpha_n^+]$}, we proceed similarly,
\begin{itemize}
    \item When $z \in [\alpha_n^-, \alpha_n^+]$, the computations are similar as previous case of $x$.
    \item  When $z \in \Cc \setminus [\alpha_n^-, \alpha_n^+]$, we have $$\sup\limits_{\substack{z \in \Cc \setminus [\alpha_n^-, \alpha_n^+]}} \; \lambda_0 -  d_{\psi_n}(x,z) = \lambda_0 + \inf\limits_{\substack{z \in \Cc \setminus [\alpha_n^-, \alpha_n^+]}} \;  d_{\psi_n}(x,z) = \lambda_0.$$ Indeed by definition of $d_{\psi_n}$, the supremum is attained for $z=x$.
    
\end{itemize}
Combining these two partial sup we have
\begin{align}
\forall x  \in \Cc \setminus [\alpha_n^-, \alpha_n^+], \quad \beta_{\psi_n}(x) &= \max \left\{ \psi_n(0) - \psi_n(x) + \psi_n'(\alpha_n^\pm) x, \lambda_0\right\} \notag \\
&=  \lambda_0 \label{eq:case2-of-x}
\end{align}
due to the fact that  $\psi_n(0) - \psi_n(x) + \psi_n'(\alpha_n^\pm) x = \lambda_0 - d_{\psi_n}(x, \alpha_n^\pm) \leq \lambda_0$ for all  $x  \in \Cc$.
Finally, the proof is completed through the combination of~\eqref{eq:case1-of-x} and~\eqref{eq:case2-of-x}.

\subsection{Proof of Proposition~\ref{prop:Bregman-composition-of-S}}\label{appendix:Bregman-composition-of-S}
The fact that $S_\Psi$ takes values in $\R_{\leq 0}$ easily follows from equation~\eqref{eq:definition:s_transfrom} and the fact that $\forall (\x,\z), \; D_\Psi(\x,\z) \geq0$, which also shows that $S_\Psi$ never attains ~$-\infty$. Using equation~\eqref{eq:definition:s_transfrom}, we can see that $S_\Psi$ is the sum of a l.s.c. function and a continuous function. Therefore, the continuity statements can be derived from the standard properties of l.s.c. convex functional (see, e.g., \cite[Proposition 3.2]{Carlsson2019}). 

 The proof of~\eqref{eq:B_Psi_composition} follows~\cite[Proposition 3.1]{Carlsson2019}. First of all, we have that $\alpha - D_{\Psi}(\cdot,\z) \leq \lambda_0 \|\cdot\|_0$ if and only if  $\alpha \leq \lambda_0 \|\cdot\|_0 +  D_{\Psi}(\cdot,\z)$. As such, the maximal $\alpha$ for a fixed $\z$ is given by
\begin{equation}
    \alpha = - S_{\Psi}(\z).
\end{equation}
It then follows from~\eqref{eq:prop:def_bergman1} that 
\begin{align}
     B_\Psi(\x)= \sup_{\z \in \Cc^N } \; - S_{\Psi}(\z) - D_{\Psi}(\x,\z) =  S_{\Psi} \circ S_{\Psi} (\x) 
\end{align}
which also shows that $B_\Psi$ is lower semi-continuous and takes values in $\R_{\geq 0}$.

\subsection{Clarke subdifferential of \BR{}}\label{appendix:formula_clark}
We start this section by recalling the definition of Clarke's subdifferential, which is a generalization of subdifferential for non-smooth and non-convex functions.
\begin{definition}[Generalized gradient~\cite{Frank99}] \label{def:def_GG} Let $f: \mathbb{R}^N\to \mathbb{R} $ be a locally Lipschtiz function and let $\x \in \mathbb{R}^N$. Then Clarke’s generalized gradient at $\x$, denoted by $\partial f(\x)$ is defined by :
\begin{equation}\label{eq:def_GG}
  \partial f(\mathbf{x}) = \left\lbrace \bm{\xi} \in \mathbb{R}^N \,:\, f^\circ(\mathbf{x}, \mathbf{v}) \geq \langle \mathbf{v}, \bm{\xi} \rangle,~ \forall \mathbf{v} \in \mathbb{R}^N \right\rbrace,
\end{equation}
where $f^\circ(\x,\v)$ stands for the Clarke's generalized directional derivative of f at $\x$ in the direction $\v \in \mathbb{R}^N$, that is
\begin{equation}
  f^\circ (\x,\v)= \limsup _{\substack{\y \rightarrow \x  \\
\eta \downarrow 0}} \frac{f(\y+\eta \v)-f(\y)}{\eta}
.
\end{equation}
\end{definition}

The following Lemma shows that the 1D relaxations $\beta_{\psi_n}$ are locally Lipschitz continuous functions, thus they admit a Clarke's subdifferential.

\begin{lemma}\label{lemma:B-locally-lip}
Let $n \in [N]$. For all $x \in \Cc$, there exists a neighborhood $\mathcal{N}$ of $x$, such that $\beta_{\psi_n}$ is Lipschitz in $\mathcal{N}$. 
\end{lemma}
\begin{proof}
Let $n\in [N]$. For all $x \in \Cc \backslash (\alpha_n^-,\alpha_n^+)$ $\beta_{\psi_n}$ is constant, thus the proof is trivial. Now, let $x \in (0,\alpha_n^+)$ and let $\mathcal{N}_1$ be a neighborhood of $x$.
 Since $\psi_n$ is convex, it is locally Lipschitz at $x$, hence there exists $L>0$ such that $\psi_n$ is  $L$-Lipschitz on a neighborhood $\mathcal{N}_2$ of $x$.
 Thus for all $x' \in \mathcal{N}_1 \cap \mathcal{N}_2$, we have 
\begin{align*}
    \left|\beta_{\psi_n}(x) - \beta_{\psi_n}(x')\right| &= \left|\psi_n(0) - \psi_n(x) + \psi_n'(\alpha_n^+)x - \psi_n(0) + \psi_n(x') - \psi_n'(\alpha_n^+)x'\right|\\
    &= \left|- \psi_n(x) + \psi_n(x') + \psi_n'(\alpha_n^+)(x - x')\right|\\
    &\leq (L+|\psi_n'(\alpha_n^+)|)|x-x'|
\end{align*}
which proves that $\beta_{\psi_n}$  is locally Lipschitz at $x$. Proceeding similarly for $x\in (\alpha_n^-,0)$ completes the proof.
\end{proof}


From Definition~\ref{def:def_GG}, we can thus compute~$\partial \beta_{\psi_n}$, where~$\beta_{\psi_n}$ is the 1D functional defined  in~\eqref{eq:prop1_equa2}.
Let us consider the case $x \neq 0$. Based on~\cite[Corollary to Proposition 2.2.4]{Frank99}, when $\beta_{\psi_n}$ is continuously differentiable on a neighborhood of $x$, $\partial \beta_{\psi_n}(x)$ reduces to 
$\partial \beta_{\psi_n}(x) = \left\lbrace \beta_{\psi_n}'(x) \right\rbrace$.
 
Let now consider the case $x=0$. We first compute the generalized directional derivative at $x=0$ for all $v \in \R$:
$$
\begin{aligned}
 \beta^\circ_{\psi_n}(0, v) & =\limsup _{\substack{y \rightarrow 0 \\
\eta \downarrow 0}} \;  \frac{\beta_{\psi_n}(y+\eta v)-\beta_{\psi_n}(y)}{\eta} \\
& =\limsup _{\substack{{y} \rightarrow 0 \\
\eta \downarrow 0}} \frac{1}{\eta} \left[\psi_n(0)-\psi_n\left( {y}+\eta v \right)+\psi_n'(\alpha_n^\pm) \left( {y}+\eta v \right)-\left(\psi_n(0)-\psi_n( {y})+\psi_n' (\alpha_n^\pm)  {y} \right) \right]\\
&= \limsup _{\substack{{y} \rightarrow 0 \\
\eta \downarrow 0}} -\frac{\psi_n\left( {y}+\eta v \right)-\psi_n( {y})}{\eta}+\psi_n'(\alpha_n ^\pm) v\\
& = (-\psi_n)^\circ(0,v)+\psi'_n(\alpha_n ^\pm) v 
.
\end{aligned}
$$
Since $(-\psi_n)$ is a smooth and concave function, it is evident that
$$
\beta^\circ_{\psi_n}(0, v) = \begin{cases}
     -\psi_n'(0) v+\psi'_n(\alpha_n ^-) v, & \text{if} \quad v \in \R_{\leq 0},\\
     -\psi_n'(0) v+\psi'_n(\alpha_n ^+) v, & \text{if} \quad v \in \R_{\geq 0}
     .
\end{cases}
$$
By combining the above equality with \eqref{eq:def_GG}, we  derive the following result:
$$
\begin{aligned}
  \xi \in   \partial \beta_{\psi_n} (0)& \Longleftrightarrow \begin{cases}
   -\psi_n '(0) v+\psi_n'(\alpha_n ^-) v  \geq v \xi, \quad \forall v \in \mathbb{R}_{\leq 0}
   \\
     -\psi_n'(0) v+\psi_n'(\alpha_n ^+) v  \geq v \xi, \quad \forall v \in \mathbb{R}_{\geq 0}
    \end{cases}
    \\
   & \Longleftrightarrow 
    \xi \in \left[ -\psi_n'(0)+\psi_n'(\alpha_n^-) ,   -\psi_n'(0)+\psi_n'(\alpha_n^+) \right]
    .
\end{aligned}
$$
We now set $\ell_n^\pm :=-\psi_n'(0)+\psi_n'(\alpha_n^\pm)$. Based on the previous discussions, we thus obtain:
\begin{equation}\label{eq:CP_beta}
\partial \beta_{\psi_n}(x)=\begin{cases}
    -\psi_n'(x)+\psi_n'\left(\alpha_n^-\right), &  \text{if} \quad x \in \left[\alpha_n^- , 0 \right),\\
    -\psi_n'(x)+\psi_n'\left(\alpha_n^+\right), &  \text{if} \quad x \in \left( 0, \alpha_n^+ \right],
    \\
    \left[\ell_n^-  , \ell_n^+  \right], & \text{if} \quad x=0
    ,
    \\
    0, & \text{otherwise}
    .
\end{cases}
\end{equation}

\subsection{Proof of Proposition~\ref{prop:critical_point_Jpsi}}\label{appendix:charachterization_critical_points}

Let us first focus on the unconstrained case, $\Cc = \R$. As the data term~$F_\y$ in~\eqref{eq:relaxation_Jpsi} is  differentiable, then, according to~\cite[Corollary 1]{Frank99}, the following holds:
\begin{equation}\label{eq:sup_critical_Jpsi}
    \forall \x \in \mathbb{R}^N, \quad \partial J_\Psi\left(\x \right) = \A^T \nabla F_\y(\A\x)+\lambda_2\x +\partial B_{\Psi} \left(\x\right)
    .
\end{equation}
Now, $\hat{\x}\in\mathbb{R}^N$ is a critical point of the functional $J_\Psi$ if and only if 
$$
    \bm{0} \in  \partial J_\Psi\left(\hat{\x} \right)
    .
$$
Using the expression \eqref{eq:sup_critical_Jpsi} we can write more precisely:
\begin{equation}
\bm{0} \in \{\A^T \nabla F_\y(\A\hat{\x})+\lambda_2 \hat{\x} \}+ \prod_{n \in [N]} \partial \beta_{\psi_n}(\hat{x}_n).
\end{equation}
By substituting \eqref{eq:CP_beta} into the previous inclusion, we  thus obtain:
\begin{align*}
 \forall n \in [N],\quad  & \; 
\begin{cases}
    0 \in \left[ \left<\a_n, \nabla F_\y \left( \A\hat{\x}\right)\right>+\ell_n^-, \left<\a_n, \nabla F_\y \left( \A\hat{\x}\right)\right>+\ell_n^+ \right], &  \text{if} \; \hat{x}_n = 0, \\
    \left<\a_n, \nabla F_\y \left( \A\hat{\x}\right)\right>+\lambda_2 \hat{x}_n-\psi'_n(\hat{x}_n)+\psi'_n \left( \alpha_n^- \right)=0, & \text{if} \; \hat{x}_n \in \left[ \alpha_n^-, 0 \right),\\
     \left<\a_n, \nabla F_\y \left( \A\hat{\x}\right)\right>+\lambda_2\hat{x}_n-\psi'_n(\hat{x}_n)+\psi'_n \left( \alpha_n^+ \right)=0, & \text{if} \; \hat{x}_n \in \left(0, \alpha_n^+ \right],\\
     \left<\a_n, \nabla F_\y \left( \A\hat{\x}\right)\right>+\lambda_2\hat{x}_n=0, & \text{if} \; \hat{x}_n \in \mathbb{R} \backslash \left[ \alpha_n^- , \alpha_n^+ \right]
     ,
\end{cases}
\end{align*}
which can be equivalently expressed as:
\begin{align*}
 \forall n \in [N], \quad &  \;
 \begin{cases}
-\left<\a_n, \nabla F_\y \left( \A\hat{\x}\right)\right> \in \left[\ell_n^- ,\ell_n^+  \right], &  \text{if} \; \hat{x}_n = 0 ,\\
\left<\a_n, \nabla F_\y \left( \A\hat{\x}\right)\right>+\lambda_2\hat{x}_n-\psi'_n(\hat{x}_n)+\psi'_n \left( \alpha_n^- \right)=0, & \text{if} \; \hat{x}_n \in \left[ \alpha_n^-, 0 \right),\\
     \left<\a_n, \nabla F_\y \left( \A\hat{\x}\right)\right>+\lambda_2\hat{x}_n-\psi'_n(\hat{x}_n)+\psi'_n \left( \alpha_n^+ \right)=0, & \text{if} \; \hat{x}_n \in \left(0, \alpha_n^+ \right],\\
     \left<\a_n, \nabla F_\y \left( \A\hat{\x}\right)\right>+\lambda_2\hat{x}_n=0, & \text{if} \; \hat{x}_n \in \mathbb{R} \backslash \left[ \alpha_n^- , \alpha_n^+ \right]
 .    
 \end{cases}
\end{align*}
Now, for the case $\Cc = \R_{\geq 0}$, we get the following optimality condition from~\cite[Theorem 8.15]{Rockwtes-optimalcond04}:
\begin{equation}
    \bm{0} \in \{\A^T \nabla F_\y(\A\hat{\x})+\lambda_2 \hat{\x} \}+ \prod_{n \in [N]} \partial \beta_{\psi_n}(\hat{x}_n) + \mathcal{N}_\Cc(\hat{x}_n)
    ,
\end{equation}
where $\mathcal{N}_\Cc$ is the normal cone of the set $\Cc = \R_{\geq 0}$ defined by 
$$
\mathcal{N}_{\Cc}(x) = \left\{ v \in \R \mid \langle v, z - x \rangle \leq 0, \ \forall z \in \Cc \right\}
.
$$
Note that one the constraint qualification condition required in~\cite[Theorem 8.15]{Rockwtes-optimalcond04} is satisfied in the present case.
Then, one can easily see that $\mathcal{N}_\Cc(0) = (-\infty, 0]$ and $\mathcal{N}_\Cc(x) = \{0\}$ if $x > 0$. Therefore, we conclude that:
\begin{align*}
 \forall n \in [N], \quad &  \;
 \begin{cases}
-\left<\a_n, \nabla F_\y \left( \A\hat{\x}\right)\right> \in \left(-\infty ,\ell_n^+  \right], &  \text{if} \; \hat{x}_n = 0 ,\\
     \left<\a_n, \nabla F_\y \left( \A\hat{\x}\right)\right>+\lambda_2\hat{x}_n-\psi'_n(\hat{x}_n)+\psi'_n \left( \alpha_n^+ \right)=0, & \text{if} \; \hat{x}_n \in \left(0, \alpha_n^+ \right],\\
     \left<\a_n, \nabla F_\y \left( \A\hat{\x}\right)\right>+\lambda_2\hat{x}_n=0, & \text{if} \; \hat{x}_n \in \Cc \backslash \left[0 , \alpha_n^+ \right]
 ,   
 \end{cases}
\end{align*}
which complete the proof.

\subsection{Proof of Theorem~\ref{th:exact_relax}}\label{appendix:exact_relax}
\paragraph{Proof that minimizers of $J_\Psi$ (global or not) are minimizers of $J_0$}
Let $\hat{\x} \in \Cc^N$ be a minimizer (global or not) of $J_\Psi$. Then, from condition \eqref{eq:general-exact-condition} 
we get that 
\begin{equation}\label{proof_min-1}
    \forall n \in [N], \quad \hat{x}_n =0 \; \text{ or } \; \hat{x}_n \notin (\alpha_n^-,\alpha_n^+).
\end{equation}
It thus follows from the expression of $B_\Psi$ (see Proposition~\ref{prop:Breg_l0}) that 
\begin{equation}\label{proof_min-1-1}
    J_\Psi(\hat{\x}) = J_0(\hat{\x}).
\end{equation}
Now assume that $\hat{\x}$ is not a minimizer of $J_0$. Then, for all neighborhood $\mathcal{V} \ni \hat{\x}$, there exists $\tilde{\x} \in \mathcal{V} \cap \Cc^N$ such that $J_0(\tilde{\x}) < J_0(\hat{\x})$.
Moreover, by definition of $B_\Psi$ (see Definition~\ref{def:Breg_L0}), we have that $B_\Psi \leq \|\cdot \|_0$. Thus we have 
\begin{equation}
    J_\Psi(\tilde{\x}) \leq J_0(\tilde{\x}) <  J_0(\hat{\x}) = J_\Psi(\hat{\x}),
\end{equation}
which contradicts the fact that $\hat{\x}$ is a minimizer (global or not) of $J_\Psi$.\\

This part proves~\eqref{eq:result_local_min} on local minimizers and one implication of the result on global minimizers~\eqref{eq:result_global_min}. It remains to prove the reciprocal for global minimizers.

\paragraph{Proof that global minimizers of $J_0$ are global minimizer of $J_\Psi$}

Let $\hat{\x} \in \Cc^N$ be a global minimizer of $J_0$. Assume that it is not a global minimizer of $J_\Psi$. Hence, there exists  a global minimizer $\tilde{\x} \in \Cc^N \backslash \{\hat{\x}\}$ (from Theorem~\ref{coro:existence_min_global_jpsi}) of $J_\Psi$ such that $J_\Psi(\tilde{\x})< J_\Psi(\hat{\x})$. 
\begin{equation}\label{eq:min-glb}
    J_0(\tilde{\x}) = J_\Psi(\tilde{\x}) < J_\Psi(\hat{\x}) \leq J_0(\hat{\x}),
\end{equation}
where the first equality in~\eqref{eq:min-glb} comes from~\eqref{proof_min-1} and~\eqref{proof_min-1-1},
while the last comes from the fact that  $B_\Psi \leq \change{\lambda
_0} \|\cdot \|_0$. 
Finally, \eqref{eq:min-glb} contradicts the fact that $\hat{\x}$ is a global minimizer of $J_0$ and completes the proof.

\subsection{Proof of Proposition~\ref{Local_minimizers_of_J0_preserved_by_JPsi}}\label{appendix:jpsi-removes-loc}

Let $\hat{\x} \in \Cc^N $ be a local minimizer of~$J_0$ and set $\hat{\sigma} =\sigma (\hat{\x})$. Then,  from Corollary~\ref{coro:criter_locmin_j0}, $\hat{\x}_{\hat{\sigma}}$  solves
    \begin{align}
         & \A_{\hat{\sigma}}^T
        \nabla F_\y\left(\A_{\hat{\sigma}}\hat{\x}_{\hat{\sigma}}  \right)+\lambda_2\hat{\x}_{\hat{\sigma}}  = \mathbf{0} \notag \\
          \Longleftrightarrow \;  & \forall n \in \hat{\sigma} , \; \left< \a_n , \nabla F_\y \left( \A \hat{\x}  \right) \right>+\lambda_2 \hat{x}_n = {0}. \label{eq:proof_removed_min_loc-0}
    \end{align}
We now proceed by proving both implications of the equivalence stated in the proposition.
\begin{itemize}
    \item[$\bullet$] [\eqref{eq:condition1}--\eqref{eq:condition2} $ \Rightarrow$ $\hat{\x}$ local minimizer of $J_\Psi$].
     Under~\eqref{eq:condition1} and~\eqref{eq:condition2}, we deduce from~\eqref{eq:proof_removed_min_loc-0} that
   \begin{align} \label{eq:proof_removed_min_loc-1}
      \begin{cases}
            \forall n \in \hat \sigma, \; \hat{x}_n \in \Cc \backslash  [\alpha_n^-,\alpha_n^+] \; \text{and} \, \left< \a_n , \nabla F_\y \left( \A \hat{\x}  \right) \right> +\lambda_2 \hat{x}_n = {0}, \\
            \forall n \in \hat{\sigma}^c, \;
            \hat{x}_n =0 \; \text{and} \; -\left<\a_n ,\nabla F_\y\left(\A\hat{\x} \right) \right> \in \left[ \ell_n^-,\ell_n^+  \right].
       \end{cases}
   \end{align}
   It then follows from Proposition~\ref{prop:critical_point_Jpsi} that $\hat{\x}$ is a critical point of $J_\Psi$. It now remains to prove that $\hat{\x}$ is actually a local minimizer of  $J_\Psi$. First, let us remark that from~\eqref{eq:condition1} we have, $\forall n \in \hat{\sigma}$, either $\hat{x}_n > \alpha_n^+$ or  $\hat{x}_n < \alpha_n^-$. Hence
   $$\rho_1 := \min \left( \min_{n \in \hat{\sigma}, \hat{x}_n >0} (\hat{x}_n - \alpha_n^+),  \min_{n \in \hat{\sigma}^c , \hat{x}_n <0} (\alpha_n^- 
 - \hat{x}_n ) \right) >0.$$ Setting $\rho := \min(\rho_1,\min_n (\alpha_n^+), \min_n (-\alpha_n^-))$ we obtain that, $\forall \bm{\varepsilon} \in \Binf(\bm{0},\rho)$,
 \begin{align}
     B_\Psi(\hat \x + \bm{\varepsilon}) & = \sum_{n \in \hat{\sigma}^c} \beta_{\psi_n}(\varepsilon_n) +  \sum_{n \in \hat{\sigma}}  \beta_{\psi_n}(\hat{x}_n) \notag \\
     & \geq  \sum_{\substack{n \in \hat{\sigma}^c\\ \varepsilon_n > 0}} \frac{\beta_{\psi_n}(\alpha_n^+)}{\alpha_n^+} \varepsilon_n  + \sum_{\substack{n \in \hat{\sigma}^c\\ \varepsilon_n < 0}} \frac{\beta_{\psi_n}(\alpha_n^-)}{\alpha_n^-} \varepsilon_n +   \sum_{n \in \hat{\sigma}}  \beta_{\psi_n}(\hat{x}_n)  \label{eq:proof_removed_min_loc-2}
 \end{align}
where the second inequality comes from the concavity of $\beta_{\psi_n}$ over the intervals $[\alpha_n^-,0]$ and $[0,\alpha_n^+]$. It follows that, for all $ \bm{\varepsilon} \in \Binf(\bm{0},\rho)$ such that $\sigma(\bm{\varepsilon}) \subseteq \hat{\sigma}$, $B_\Psi(\hat \x + \bm{\varepsilon}) = B_\Psi(\hat \x )$ and, as $\hat{\x}_{\hat{\sigma}}$ solves the subproblem~\eqref{eq:restriction_j0}, $J_\Psi(\hat \x + \bm{\varepsilon}) \geq J_\Psi(\hat \x)$ with a strict inequality if and only if $\lambda_2>0$ or $\mathrm{rank}(\A_{\hat{\sigma}}) = \sharp \hat{\sigma}$ (from Lemma~\ref{lem:strict-min}).

We now study the case where we take $\bm{\varepsilon} \in \Binf(\bm{0},\rho)$ such that $\sigma(\bm{\varepsilon}) \nsubseteq \hat{\sigma}$, i.e., there exists $n \in \hat{\sigma}^c$ such that $\varepsilon_n \neq 0$. From the convexity of $H_\y = F_\y(\A \cdot) + \frac{\lambda_2}{2}\|\cdot\|_2^2$, we get
\begin{equation}
    H_\y(\hat{\x} + \bm{\varepsilon}) \geq H_\y(\hat{\x}) + \left\langle \nabla H_\y(\hat{\x}), \bm{\varepsilon} \right\rangle
\end{equation}
where we recall that $\forall n \in [N]$, $[\nabla H_\y(\hat{\x})]_n = \left< \a_n , \nabla F_\y \left( \A \hat{\x}  \right) \right>+\lambda_2 \hat{x}_n $. From~\eqref{eq:proof_removed_min_loc-1}, we have $[\nabla H_\y(\hat{\x})]_n =0 \; \forall n \in \hat{\sigma}$ and $[\nabla H_\y(\hat{\x})]_n \geq - \ell_n^+ \; \forall n \in \hat{\sigma}^c$. We then deduce
\begin{equation}
    H_\y(\hat{\x} + \bm{\varepsilon}) \geq H_\y(\hat{\x}) - \sum_{n \in \hat{\sigma}^c} \ell_n^+ \varepsilon_n \geq H_\y(\hat{\x}) - \sum_{\substack{n \in \hat{\sigma}^c \\ \varepsilon_n >0}} \ell_n^+ \varepsilon_n \label{eq:proof_removed_min_loc-3}
\end{equation}
where the last inequality is obtained using the fact that $\ell_n^+ = -\psi_n'(0) + \psi_n'(\alpha_n^+) >0$ by strict convexity of $\psi_n$. Combining~\eqref{eq:proof_removed_min_loc-2} and~\eqref{eq:proof_removed_min_loc-3} we get
\begin{align}
    J_\Psi(\hat{\x} + \bm{\varepsilon}) & \geq J_\Psi(\hat{\x}) +  \sum_{\substack{n \in \hat{\sigma}^c\\ \varepsilon_n > 0}} \left( \frac{\beta_{\psi_n}(\alpha_n^+)}{\alpha_n^+} -\ell_n^+ \right) \varepsilon_n  + \sum_{\substack{n \in \hat{\sigma}^c\\ \varepsilon_n < 0}} \frac{\beta_{\psi_n}(\alpha_n^-)}{\alpha_n^-} \varepsilon_n \label{eq:proof_removed_min_loc-4} \\
   & \geq  J_\Psi(\hat{\x}) +  \sum_{\substack{n \in \hat{\sigma}^c\\ \varepsilon_n > 0}} \left( \frac{\beta_{\psi_n}(\alpha_n^+)}{\alpha_n^+} -\ell_n^+ \right) \varepsilon_n   
\end{align}
 because ${\beta_{\psi_n}(\alpha_n^-)}/{\alpha_n^-} <0$ and thus $({\beta_{\psi_n}(\alpha_n^-)}\varepsilon_n)/{\alpha_n^-} >0$ for $\varepsilon_n <0$. Finally, from the expressions of $\beta_{\psi_n}$ and $\ell_n^+$ we get
 \begin{align*}
      \frac{\beta_{\psi_n}(\alpha_n^+)}{\alpha_n^+} -\ell_n^+& = \frac{1}{\alpha_n^+}(\psi_n(0) - \psi_n(\alpha_n^+) + \psi_n'(\alpha_n^+)\alpha_n^+) + \psi_n'(0) - \psi_n'(\alpha_n^+) \\
      & = \frac{1}{\alpha_n^+}(\psi_n(0) - \psi_n(\alpha_n^+)) + \psi_n'(0) > 0
 \end{align*}
due, again, to the strict convexity of $\psi_n$. Hence, we have shown that $\forall \bm{\varepsilon} \in \Binf(\bm{0},\rho)$ such that $\sigma(\bm{\varepsilon}) \nsubseteq \hat{\sigma}$, $J_\Psi(\hat{\x} + \bm{\varepsilon}) > J_\Psi(\hat{\x})$. Note that if $\forall n \in \hat{\sigma}^c$ we have $\varepsilon_n \leq 0$, then the strict inequality is due to the sums over $\{n \in \hat{\sigma}^c : \varepsilon_n <0\}$ in~\eqref{eq:proof_removed_min_loc-3} and~\eqref{eq:proof_removed_min_loc-4}.

Gathering the two cases, we have proved that  $\forall \bm{\varepsilon} \in \Binf(\bm{0},\rho)$,  $J_\Psi(\hat{\x} + \bm{\varepsilon}) \geq  J_\Psi(\hat{\x})$ with a strict inequality if and only if $\lambda_2>0$ or $\mathrm{rank}(\A_{\hat{\sigma}}) = \sharp \hat{\sigma}$. This completes the proof of the first implication.

\item[$\bullet$] [$\hat{\x}$ local minimizer of $J_\Psi$  $ \Rightarrow$ \eqref{eq:condition1}--\eqref{eq:condition2}]. Given that $\hat{\x}$ is a local minimizer of $J_\Psi$, it is a critical point and satisfies the equations of Proposition~\ref{prop:critical_point_Jpsi}. Hence~\eqref{eq:condition2} is trivially satisfied while~\eqref{eq:condition1} is deduced  from~\eqref{eq:proof_removed_min_loc-0}. 
\end{itemize}

\section{Computation of \BR{} of some generating functions}\label{appendix:compute-brex}
\subsection{Power functions}
We consider $\psi(x)=\frac{\gamma}{p(p-1)} |x|^p$ defined on $\Cc=\R$ with $\gamma >0$ and $p >1$. This function is strictly convex on   $\mathbb{R}$ with first derivative of 
 $\psi'(x)=\frac{\gamma}{p-1} \operatorname{sign}(x) |x|^{p-1}$. The Bregman distance induced at point $(0,x)$ by such $\psi$  is
 \begin{equation}
     d_{\psi}(0,x)=\frac{\gamma}{p} |x|^p.
 \end{equation}
The $\lambda_0$-sublevel set of $d_{\psi}(0,\cdot)$ is thus given by:
\begin{equation}
S_{\lambda_0} = \left\lbrace x \in \mathbb{R}  \mid \frac{\gamma}{p}|x|^p  \leq \lambda_0  \right\rbrace =\left[-\left( \frac{p\lambda_0}{\gamma} \right)^{\frac{1}{p}} , \left( \frac{p\lambda_0}{\gamma} \right)^{\frac{1}{p}}  \right]
\end{equation}
Using  Proposition \ref{prop:Breg_l0},
we can compute the $\ell_0 $ Bregman relaxation $\beta_{\psi}$ for this choice of $\psi$. The resulting expression is given by:
\begin{equation}\label{eq:l_0_Bregman_LS}
 \beta_{\psi}(x) = \left\lbrace 
    \begin{array}{ll}
         \frac{-\gamma}{p(p-1)}|x|^p-\frac{\gamma}{p-1}  \left( \frac{p\lambda_0}{\gamma} \right)^{\frac{p-1}{p}}  x, & \quad\text{ if } x \in [- \left( \frac{p\lambda_0}{\gamma} \right)^{\frac{1}{p}} ,0], \\
           \frac{-\gamma}{p(p-1)}|x|^p +  \frac{\gamma}{p-1} \left( \frac{p\lambda}{\gamma} \right)^{\frac{p-1}{p}} x,   &\quad \text{ if } x \in [0, \left( \frac{p\lambda_0}{\gamma} \right)^{\frac{1}{p}} ], \\ 
      \lambda_0,   &\quad  \text{ otherwise}.
    \end{array}
    \right.
\end{equation}
\subsection{Shannon entropy}
We now consider the function defined on $\Cc=\R_{\geq 0}$ by
\begin{equation}\label{eq:shannon}
    \psi(x)= \begin{cases}
        \gamma \left( x\log(x)-x+1 \right),\quad  & \text{if } \; x \in \left(0, + \infty \right),\quad  \\
        \gamma,\quad & \text{if } \; x=0.
    \end{cases}
\end{equation}
where $\gamma>0$. The function $\psi$ is strictly convex, l.s.c and continuously differentiable on $(0, +\infty)$. Therefore, $\psi'$ exists on $(0, +\infty)$ and we have $\psi'(x)=\gamma\log(x)$. 
By definition of Bregman distance, we have $d_{\psi}(0,x)=\gamma x $. The $\lambda_0$-sublevel set of $d_{\psi}(0,\cdot) $ is thus given by $\left[ 0,\frac{\lambda_0}{\gamma} \right] $. Applying Proposition \ref{prop:Breg_l0}, we obtain the following expression for the $\ell_0$ Bregman relaxation $\beta_{\psi}$ associated with this choice of $\psi$
\begin{equation}~\label{eq:Bpsi_kull_leibler}
\beta_{\psi}(x) = \left\lbrace
\begin{array}{ll}
\gamma x \left(\log(\frac{\lambda_0}{\gamma})-\log(x)+1 \right),  &\quad \text{if } x \in \left[ 0,\frac{\lambda_0}{\gamma} \right],\\
\lambda_0, &\quad \text{otherwise.}
\end{array}
\right.
\end{equation}

\subsection{Kullback-Leibler divergence}
For $\gamma >0$ and $c,y>0$, we choose the generator function defined on $\Cc=\R_{\geq 0}$ as follows 
\begin{equation}\label{eq:kl-generating-a}
\psi(x)=\gamma \text{kl}(cx+b,y)=\gamma \left(cx+b-y\log(cx+b) \right)
\end{equation}
which is strictly convex, l.s.c and continuously differentiable on $\left( 0 , +\infty \right) $. We have for all $x \in \left( 0 , +\infty \right)  $
$$
\psi'(x)=\gamma c \left(1-\frac{y}{cx+b} \right)
.
$$
By definition of the  Bregman distance, we get
\begin{align*}
d_{\psi}(0,x)&=\gamma y \left( b-\log(b)-cx-b+\log(cx+b)+cx-\frac{cx}{cx+b} \right)\\
&=\gamma y \left[\log\left(\frac{cx+b}{b}\right)-\frac{cx}{cx+b} \right]\\
\end{align*}
Since we are looking for the values of $x$ that satisfy the following inequality  
\begin{equation}\label{eq:inequaliti_kl_similair-choice}
    d_{\psi}(0,x) \leq \lambda_0 
    .
\end{equation}
we look at the corresponding equation:
\begin{align*}
d_{\psi}(0,x)&= \lambda_0 \Rightarrow \log\left(\frac{cx+b}{b}\right)-\frac{cx}{cx+b} =\frac{\lambda_0}{\gamma y} \\
&\Rightarrow  \log(z)-\frac{z-b}{z}=\frac{\lambda_0}{\gamma y}+\log(b)\\
&\Rightarrow \log(z)+\frac{b}{z}=\frac{\lambda_0}{\gamma y}+\log(b)+1:=\kappa \\
&\Rightarrow ze^{\frac{b}{z}}=e^{\kappa} 
\Rightarrow \frac{-b}{z}e^{\frac{-b}{z}}=-be^{-\kappa} \Rightarrow cx^*+b=z^{*}=\frac{-b}{W(-be^{-\kappa})}
\end{align*}
where $W(\cdot)$ is the Lambert function, and the last equality come from the fact that $-be^{-\kappa} \geq -e^{-1}$, which is equivalent to $\kappa=\frac{\lambda_0}{\gamma y}+\log(b)+1 \geq \log(b)+1$ that always holds since $\frac{\lambda_0}{\gamma y} >0$. Notice that $-be^{-\kappa} \geq -e^{-1}$ is required for the Lambert function to be defined. We deduce that the inequality~\eqref{eq:inequaliti_kl_similair-choice} is verified as soon as $x \geq \frac{-1}{c}(\frac{b}{W(-be^{-\kappa})}+b)$. 
Therefore the $\lambda_0$-sublevel set is $[0,\alpha^+]$ where $\alpha^+=x^* = \frac{-1}{c}(\frac{b}{W(-be^{-\kappa})}+b)$.
Using  Proposition \ref{prop:Breg_l0}. The resulting expression of  the $\ell_0 $ Bergman relaxation  $\beta_{\psi}$ for this choice of $\psi$ is given by:
\begin{equation}\label{eq:bergman_envolpe_Kullback}
\beta_{\psi}(x) = \left\lbrace 
    \begin{array}{ll}
       \gamma y \left[
       \log \left ( \frac{cx+b}{b} \right )+\frac{W(-be^{-\kappa})}{b} cx \right],  & \quad\text{ if } x \in [0 , \alpha^+], \\
      \lambda_0,   &  \quad\text{ if } x \geq \alpha^+.
    \end{array}
    \right.
\end{equation}
\section{Computing the proximal operator}

\subsection{Proof of Proposition~\ref{prop:prox_formula}}\label{appendix:bregman_prox}

 Let $n \in [N]$ and $x \in \Cc$. The proximal operator of $\rho \beta_{\psi_n}$  
     is given by
    \begin{equation} \label{eq:prox1D_proof}
       \prox_{\rho\beta_{\psi_n}}(x) = \underset{u \in \R}{\argmin} \left\lbrace \beta_{\psi_n}(u)+\frac{1}{2\rho}(u-x)^2
     \right\rbrace
     .
    \end{equation}
Using the first-order conditions where the formula for $\partial\beta_{\psi_n}$ is given in~\eqref{eq:CP_beta}, the possible solutions are $0$, $x$ and $u^*$ that solves $u-\rho\psi_n'(u)=x-\rho\psi_n'(\alpha_n^\pm)$. Hence, defining the sets $S_x=\{u \in \R : u-\rho\psi_n'(u)=x-\rho\psi_n'(\alpha_n^\pm) \}$ and $\mathcal{U}(x)=\{ 0,x\} \cup S_x$, we get that 

\begin{equation}
       \prox_{\rho\beta_{\psi_n}}(x) = \underset{u \in \mathcal{U}(x)}{\argmin} \left\lbrace \beta_{\psi_n}(u)+\frac{1}{2\rho}(u-x)^2
     \right\rbrace
     .
    \end{equation}
Which completes the proof.

\subsection{Explicit computation of $\prox_{\rho\beta_{\psi_n}}$ for some generating functions $\psi_n$}\label{appendix:expl-prox}

In this section, we present the details of the computations of the set $S_x$, where $x\in \Cc$, defined in Proposition~\ref{prop:prox_formula}. This set is defined as the solutions of the equation $u-\rho\psi_n'(u)=x$, with $\rho>0$ and $\psi_n$ given in Table~\ref{tab:generating-functions}.
\subsubsection{Power functions}
Let $n \in [N]$, $\rho>0$ and $p \in (1,2]$. Let $q=p-1$.
We assume without loss of generality that $u \geq 0$. The computations remains similar for $u<0$. We look the solutions of the following equation 
\begin{equation}
u-\frac{\rho\gamma_n}{q} u^q=x-\rho\psi_n'(\alpha_n^+).
\end{equation}
Let $z:=u^q$. We thus get: 
\begin{equation}\label{eq:p-power-equations}
z^{\frac{1}{q}}-\frac{\rho\gamma_n}{q} z=x-\rho\psi_n'(\alpha_n^+).
\end{equation}
Let $p=2$. For $\rho \gamma_n < 1$ we have $
u=z=\frac{x-\rho\psi_n'(\alpha_n^+)}{1-\rho\gamma_n}
$.

Now, let $p=3/2$. We get the following equation 
\[
z^2-2\rho\gamma z= x-\rho\psi_n'(\alpha_n^+),
\]
which admits two real solutions for $\Delta=(\rho\gamma)^2+x-\rho\psi_n'(\alpha_n^+) \geq 0$, which are given by
\[
z=\rho\gamma\pm \sqrt{(\rho\gamma)^2+(x-\rho\psi_n'(\alpha_n^+))}
,
\]
We thus deduce that $u=x-\rho\psi_n'(\alpha_n^+)+2(\rho\gamma)^2\pm 2\rho\gamma  \sqrt{(\rho\gamma)^2+x-\rho\psi_n'(\alpha_n^+)}$. 

For $p=4/3$, we obtain the third equation 
\begin{equation}\label{eq:3-eme-degrees}
z^3=3\rho\gamma_n z+x-\rho\psi_n'(\alpha_n^+),
\end{equation}
whose real solutions depend on the sign of $\Delta=(x-\rho\psi_n'(\alpha_n^+))^2-4 (\rho\gamma)^3$. We note 
$$
A=\sqrt[3]{\frac{(x-\rho\psi_n'(\alpha_n^+))}{2}+\frac{1}{2}\sqrt{\Delta}},
$$
and 
$$
B=\sqrt[3]{\frac{(x-\rho\psi_n'(\alpha_n^+))}{2}-\frac{1}{2}\sqrt{\Delta}}.
$$
The solutions of the equation~\eqref{eq:3-eme-degrees} are of the from $z_1 = A+B$, $z_2 = \omega A + \omega^2 B$, and $z_3 = \omega^2 A + \omega B$, where $\omega=-\frac{1}{2}+i\frac{\sqrt{3}}{2}$. 

\subsubsection{Shannon entropy}
Let $n \in [N]$ and $\rho>0$. We have 
\begin{align*}
    u-\rho\gamma_n \log(u)=x-\rho\psi_n'(\alpha^+_n) &\Rightarrow \frac{u}{\rho\gamma_n}-\log(u)=\frac{x-\rho\psi_n'(\alpha^+_n) }{\rho\gamma_n} \\
   &\Rightarrow s-\log(s)=\frac{x-\rho\psi_n'(\alpha^+_n) }{\rho \gamma_n}+\log(\rho\gamma_n) 
    \Rightarrow \frac{1}{s} e^s= {\rho\gamma_n} e^{\frac{x-\rho\psi_n'(\alpha^+_n) }{\rho \gamma_n}} \\
    &\Rightarrow -se^{-s}=-\frac{1}{\rho \gamma_n}  e^{-\frac{x-\rho\psi_n'(\alpha^+_n) }{\rho \gamma_n}} \Rightarrow -s=W_{k}\left(-\frac{1}{\rho \gamma_n} e^{-\frac{x-\rho\psi_n'(\alpha^+_n) }{\rho \gamma_n}}\right)
    \\
    &\Rightarrow u=
    -\rho\gamma_n W_{k}\left(-\frac{1}{\rho \gamma_n}   e^{-\frac{x-\rho\psi_n'(\alpha^+_n) }{\rho \gamma_n}}\right)
    .
\end{align*}
where $s=\frac{u}{\rho\gamma_n}$. When dealing with real solutions, we have $k=0$ or $k=1$ the principal and the negative branches of the Lambert function, respectively for $0<e^{-\frac{x-\rho\psi_n'(\alpha^+) }{\rho \gamma_n}} \leq \rho \gamma_n e^{-1}$.
\subsubsection{KL divergence}
Let $n\in [N]$ and $\rho>0$. We need to solve:
\begin{align*}
       u-\rho\gamma_n+\frac{y\rho\gamma_n}{u+b}&=x-\rho\psi_n'(\alpha_n^+) \\
        &\Rightarrow u(u+b)-\rho\gamma_n(u+b)-(x-\rho\psi_n'(\alpha_n^+))(u+b)+y\rho\gamma_n=0 \\
        &\Rightarrow u^2+(b-\rho\gamma_n-x+\rho\psi_n'(\alpha_n^+))u-\rho\gamma_n b-b(x-\rho\psi_n'(\alpha_n^+))+y\rho\gamma_n=0.
    \end{align*}
    We have 
    \begin{align*}
        \Delta(x)&=(x-\rho\psi_n'(\alpha_n^+)+\rho\gamma_n-b)^2+4(bx-b\rho\psi_n'(\alpha_n^+)+\rho\gamma_n b-y\rho\gamma_n)
                 ,
    \end{align*}
whence the claim follows.

\end{appendices}

\bibliography{sn-bibliography}

\end{document}